\documentclass[12pt]{amsart}
\usepackage{graphicx}
\usepackage{fourier}
\usepackage{cancel}
\usepackage{graphicx}
\usepackage{epsfig}
\usepackage{color}

\newtheorem{thm}{Theorem}[section]
\newtheorem{prop}[thm]{Proposition}
\newtheorem{defn}[thm]{Definition}
\newtheorem{lemma}[thm]{Lemma}
\newtheorem{conj}[thm]{Conjecture}
\newtheorem{cor}[thm]{Corollary}

\newtheorem{example}[thm]{Example}

\usepackage{amsmath}
\usepackage{amsxtra}
\usepackage{amscd}
\usepackage{amsthm}
\usepackage{amsfonts}
\usepackage{amssymb}
\usepackage{eucal}
\newcommand{\bmb}{\left( \begin{array}{rr}}
\newcommand{\enm}{\end{array}\right)}

\newcommand{\Z}{{\mathbb Z}}

\numberwithin{equation}{section}

\begin{document}

\title[20V model with Domain Wall Boundaries and Domino tilings]{Twenty-Vertex model with Domain Wall Boundaries  
\\and Domino tilings}
\author{Philippe Di Francesco} 
\address{
Department of Mathematics, University of Illinois, Urbana, IL 61821, U.S.A. 
and 
Institut de physique th\'eorique, Universit\'e Paris Saclay, 
CEA, CNRS, F-91191 Gif-sur-Yvette, FRANCE\hfill
\break  e-mail: philippe@illinois.edu
}
\author{Emmanuel Guitter}
\address{
Institut de physique th\'eorique, Universit\'e Paris Saclay, 
CEA, CNRS, F-91191 Gif-sur-Yvette, FRANCE.
\break  e-mail: emmanuel.guitter@ipht.fr
}

\begin{abstract}
We consider the triangular lattice ice model (20-Vertex model) with four types of domain-wall type boundary conditions. In types 1 and 2, the configurations are shown to be equinumerous to the quarter-turn symmetric domino tilings of an Aztec-like holey square, with a central cross-shaped hole. The proof of this statement makes extensive use of integrability and of a connection to the 6-Vertex model. The type 3 configurations are conjectured to be in same number as domino tilings of a particular triangle. The four enumeration problems are reformulated in terms of four types of Alternating Phase Matrices with entries $0$ and sixth roots of unity, subject to suitable alternation conditions. Our result is a generalization of the ASM-DPP correspondence. Several refined versions of the above correspondences are also discussed.

\end{abstract}

\maketitle
\date{\today}
\tableofcontents

\section{Introduction}
\label{sec:introduction}

Few combinatorial objects have fostered as many and as interesting developments as Alternating Sign Matrices (ASM).
As described in the beautiful saga told by Bressoud \cite{Bressoud}, these had a purely combinatorial life of their own 
from their discovery in the 80's by Mills, Robbins and Rumsey \cite{mills1983alternating} to their enumeration \cite{zeilberger1,zeilberger2} in  
relation to other intriguing
combinatorial objects such as plane partitions with maximal symmetries \cite{tsscpp} or descending plane partitions (DPP)\cite{AndrewsDPP,Lalonde,KrattDPP,BDFPZ1,BDFPZ2},
until they crashed against the tip of the iceberg of two-dimensional integrable lattice models of statistical physics
\cite{Baxter}. 
This allowed not only for an elegant proof of the so-called ASM conjecture \cite{kuperberg1996another} and its variations by changing
its symmetry classes \cite{kuperberg2002symmetry}, but set the stage for future developments, 
with some new input from physics of the underlying
six-vertex (6V) ``ice" model in the presence of special domain wall boundary conditions (DWBC), 
and its relations to a fully-packed loop gas, giving eventually rise to the Razumov-Stroganov 
conjecture \cite{RS}, finally proved by Cantini and Sportiello in 2010 \cite{CanSpo}. 

The ice model involves configurations on a domain of square lattice, obtained by orienting each individual edge
in such a way that the \emph{ice rule} is obeyed at each node, namely that exactly two edges point in and two edges point 
out of the node. It is now recognized that the statistical model
for two-dimensional ice, solved  in \cite{ice} is at the crossroads of many combinatorial wonders, in relation with
loop gases, osculating paths, rhombus and domino tiling problems, and even equivariant cohomology of the nilpotent 
matrix variety. Moreover, due to its inherent integrable structure, the model offers a panel of powerful techniques
for solving, such as transfer matrix techniques, the various available Bethe Ans\"atze, Izergin and Korepin's determinant,
the quantum Knikhnik-Zamolodchikov equation, etc. \cite{ice,Baxter,IKdet,qKZ,RSDFZJ,qKZDFZJ}.

In the present paper, we develop and study the combinatorics of the ice model on the regular triangular lattice, known as the 
Twenty-vertex (20V) model \cite{Kel,Baxter}. Focussing on the combinatorial content, we introduce particular 
``square" domains ($n\times n$ rhombi of the triangular lattice) with special boundary 
conditions meant to create domain walls, i.e. separations between domains of opposite directions of oriented edges,
in an attempt to imitate the 6V situation. Among the many possibilities offered by the triangular lattice geometry,
we found two particularly interesting classes of models, which we refer to as 20V-DWBC1,2 (where DWBC1 and DWBC2 are
two flavors of the same class viewed from different perspectives) and 20V-DWBC3. These are respectively enumerated by the sequences:
\begin{eqnarray} A_n&=&1,\ 3,\ 23,\ 433,\ 19705, \ 2151843, \ 561696335,\  349667866305...\qquad (\hbox{DWBC1,2})\label{seqone}\\
B_n&=&1,\, 3,\, 29, \, 901,\, 89893, \, 28793575 ... \qquad (\hbox{DWBC3})\label{seqtwo}
\end{eqnarray}
for $n=1,2,...$

From their definition, both models can be interpreted as generalizations of ASMs, in which non-zero entries may now
belong to the set of sixth roots of unity, and we shall refer to them as Alternating Phase Matrices (APM) of
types 1,2,3 respectively. To further study both sets of objects, we use the integrable version of the 20V model
\cite{Kel,Baxter} to (i) decorate the model's configurations with Boltzmann weights parameterized by
complex spectral parameters; and (ii) reformulate whenever possible the partition function. In this paper, we succeed in
performing this program in the case of 20V-DWBC1,2, which is eventually reformulated as an ordinary 6V-DWBC
model on a square grid, but with non-trivial Boltzmann weights $(a,b,c)=(1,\sqrt{2},1)$. The case of DWBC3
is more subtle as the model can be rephrased as a 6V model with staggered boundary conditions.

Among other possibilities, the 20V configurations may be represented as configurations of some non-intersecting 
paths with steps along the
edges of the triangular lattice, with the possibility of double or triple kissing (osculation) points at vertices. Individually,
the same
paths, once drawn on a straightened triangular lattice equal to the square lattice
supplemented with a second diagonal edge on each face, are nothing but the Schr\"oder paths on $\Z^2$, with 
horizontal, vertical and diagonal steps $(1,0)$, $(0,-1)$ and $(1,-1)$. These are intimately related to problems
of tiling of domains of the quare lattice by means of $1\times 2$ and $2\times 1$ dominos, such as the
celebrated Aztec diamond domino tiling problem \cite{CEP}, or the more recently considered Aztec rectangles with 
boundary defects \cite{BuKni,DFG3}.

Looking for candidates in the domino tiling world for being enumerated by the sequences $A_n$ or $B_n$, we 
found the two following remarkable models. 

In the case of $A_n$, the associated domino tiling problem is a natural generalization of the descending plane
partitions introduced by Andrews \cite{AndrewsDPP} and reformulated as the rhombus tiling model of a hexagon with a central
triangular hole, with $2\pi/3$ rotational symmetry \cite{Lalonde,KrattDPP}, which allows to interpret it also as rhombus tilings
of a cone. In the same spirit, we find that $A_n$ counts the \emph{domino tiling configurations of an Aztec-like, quasi-square domain, 
with a cross-shaped central hole, and with $\pi/2$ rotational symmetry}, or equivalently the domino tilings of a cone. The fact that 20V DWBC1,2 configurations 
and quarter-turn symmetric domino tilings of a holey square are enumerated by the sane sequence $A_n$ is proved in the present paper, together 
with a refinement, which parallels the refinement in the so-called ASM-DPP conjecture of Mills, Robbins and 
Rumsey \cite{mills1983alternating}. We use similar techniques to those of \cite{BDFPZ1,BDFPZ2}, namely identify 
both counting problems
as given by determinants of finite truncations of infinite matrices, whose generating functions are simply related.

In the case of $B_n$, a natural candidate was found by using the online encyclopedia of integer sequences (OEIS). The
number of domino tilings of a square of size $2n\times 2n$ is $2^n b_n^2$ \cite{Kasteleyn,sandpiles}, 
where $b_n$ itself counts the number
of domino tilings of a triangle obtained by splitting the square into two equal domains \cite{Patcher}. 
We found perfect agreement between our data for $B_n$ and $b_n$, which we conjecture to be equal for all $n$.
Despite many interesting properties of the model, we have not been able to  prove this correspondence.

The paper is organized as follows.

In Section \ref{sec:20VDWBC}, we introduce the 20V model and define a first class of models in its two flavors of 
domain wall boundary conditions DWBC$1,2$, conveniently expressed in terms of osculating Schr\"oder paths.
We also define the special integrable weights, parameterized by triples of complex spectral parameters $z,t,w$,
and obeying the celebrated Yang-Baxter relation.

These models are mapped onto the 6V-DWBC model (Theorem~\ref{An6V}) with anisotropy parameter $\Delta=1/\sqrt{2}$
in Section \ref{sec:20to6}, by use of the integrability property,
leading to compact formulas for the partition function and in particular the numbers $A_n$, in the form of some simple determinant, with known asymptotics leading to the free energy 
$f=\frac{3}{2}{\rm Log}\frac{4}{3}$ per site. 
The integrability of the model allows us moreover to keep track of a refinement of the number $A_n$
according to the number of occupied vertical edges in the last column of the domain, in the osculating Schr\"oder path 
formulation (Theorem~\ref{20to6}). The generating polynomial of the refined numbers is also given by some simple determinant.

Section \ref{sec:cone} is devoted to the definition and enumeration of the domino tiling problem of a cone corresponding to 
the sequence $A_n$, which generalizes the notion of Andrews' descending plane partitions in domino terms. 
After defining the tiling problem, we perform enumeration using a non-intersecting 
Schr\"oder path formulation and Gessel-Viennot determinants (Theorem~\ref{thm:T4}). We also introduce refinements in the same spirit as
the refinements of the DPP conjecture of \cite{mills1983alternating} (Theorem~\ref{refinT4}).

The equivalence between the enumerations in Sections \ref{sec:20to6} and   \ref{sec:cone} is proved
in Section \ref{sec:proof}, in the same spirit as the refined ASM-DPP proof of \cite{BDFPZ1}. We first evaluate the
homogeneous limit of the Izergin-Korepin determinantal formula for the 6V-DWBC partition function, and
write it in the form of the determinant of the finite $n\times n$ truncation of an infinite matrix independent of $n$.
We then identify this determinant with the Gessel-Viennot determinant of Section \ref{sec:cone} (Theorem~\ref{thm:Z20T4}). 
We also work out the refined version of this result,
by keeping one non-trivial spectral parameter in the 20V model, and identifying it in a special weighting of the
Schr\"oder path configurations for the domino tiling of the cone of Section \ref{sec:cone} (Theorem~\ref{thmref20VT4}).

We turn to other possible domain wall boundary conditions in Section  \ref{sec:other}.
We first define the 20V-DWBC3 model and formulate the Conjecture \ref{DWBC3conj} that its configurations are also enumerated
by the domino tilings of the triangle of \cite{Patcher}. We extend this to a sequence of models corresponding
to half-hexagonal shapes with the same boundary conditions, conjecturally enumerated by the domino tilings
of Aztec-like extensions of the former triangle (Conjecture \ref{pentaconj}). We complete the section with another possible DWBC4
for which no general conjecture was formulated.

In Section \ref{sec:apm}, we identify the various 20V-DWBC configurations considered in  this paper with sets 
of matrices with entries
made of \emph{triples} of elements in $\{0,1,-1\}$, or equivalently taking values among $0$ and
the sixth roots of unity, that generalize the notion of Alternating Sign Matrix. 

We gather a few concluding remarks in Section~\ref{sec:conclusion}.

\medskip

\noindent{\bf Acknowledgments.}  

\noindent  

PDF is partially supported by the Morris and Gertrude Fine endowment and the NSF grant DMS18-02044. EG acknowledges the support of the grant ANR-14-CE25-0014 (ANR GRAAL).

\section{The 20V model with Domain Wall Boundary Conditions}
\label{sec:20VDWBC}

\subsection{Definition of the model: ice rule and osculating paths}
\label{sec:def20V}

\begin{figure}
\begin{center}
\includegraphics[width=11cm]{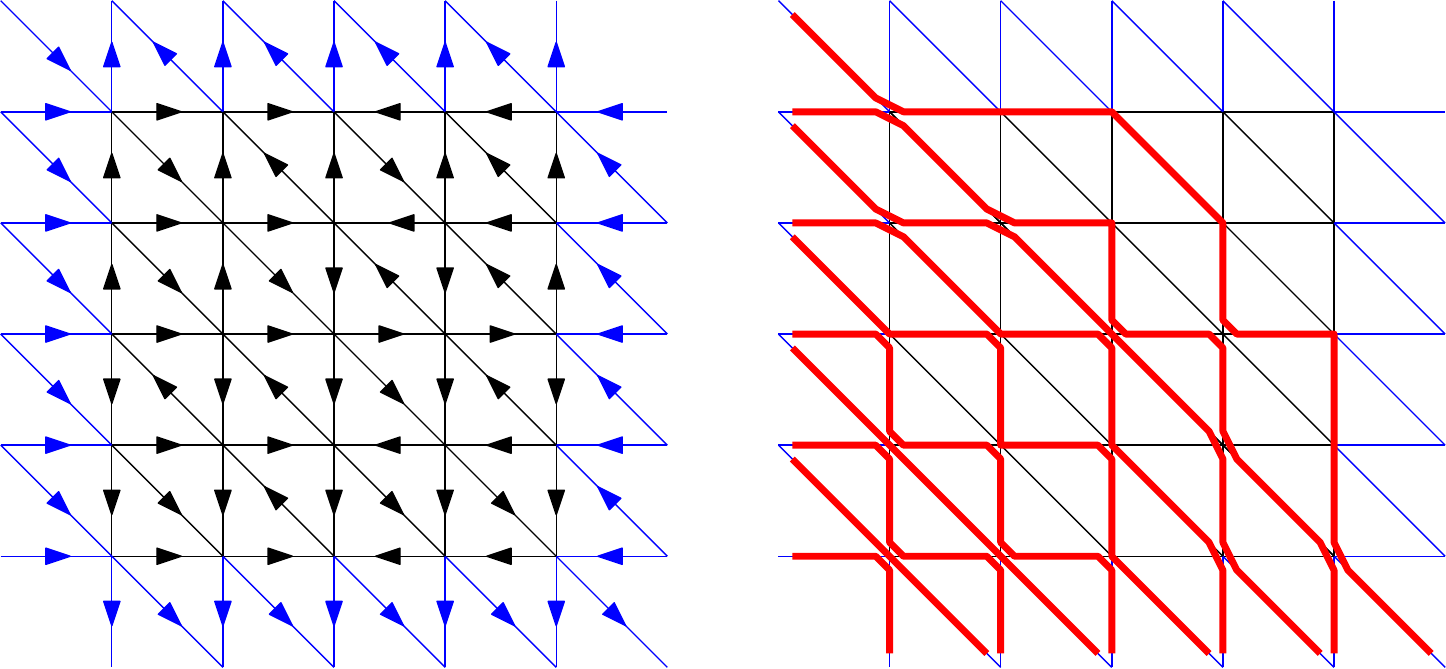}
\end{center}
\caption{\small Left: A sample configuration of a 20V model configuration with DWBC1. Right: The equivalent osculating path configuration.
}
\label{fig:DWBC1}
\end{figure}

\begin{figure}
\begin{center}
\includegraphics[width=11cm]{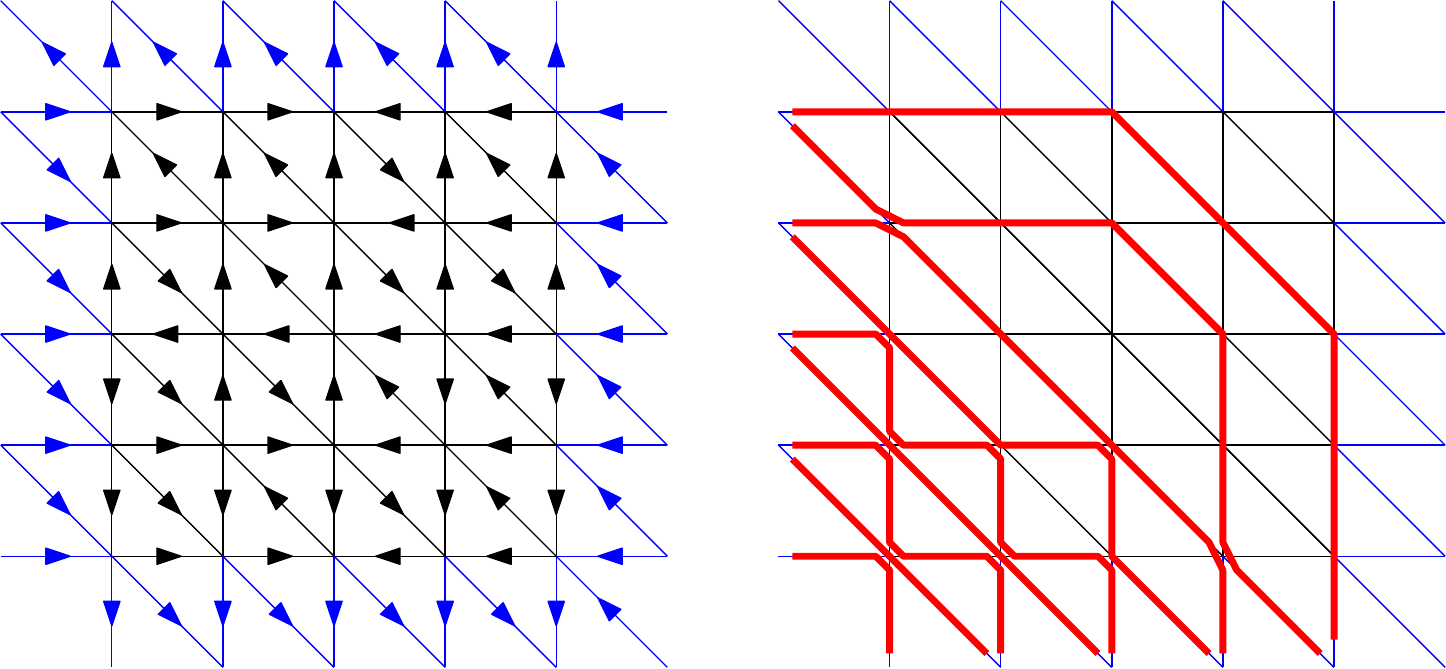}
\end{center}
\caption{\small Left: A sample configuration of a 20V model configuration with DWBC2. The boundary condition differs from that of Fig.~\ref{fig:DWBC1} by the orientation 
imposed on the upper-left and lower-right edges. Right: The equivalent osculating path configuration.
}
\label{fig:DWBC2}
\end{figure}

The combinatorial problem that we wish to address corresponds to a particular instance of the 20V model on a finite regular domain
with specific boundary conditions. The corresponding geometry is directly inspired from that of the 6V model on a portion
of square lattice with Domain Wall Boundary Conditions (DWBC) \cite{Korepin}, suitably adapted to the triangular lattice as follows:
we first straighten the triangular lattice into a square lattice supplemented with a second diagonal within each face. In this setting, the regular domain underlying our 20V model is an
$n\times n$ square portion of lattice\footnote{Before straightening, this domain is an $n\times n$ rhombus of the triangular lattice.},  whose vertices occupy all points with integer coordinates
$(i,j)$ for $i,j=1,2,\dots,n$. Its set of inner edges is made of all the elementary horizontal segments $(i,j)\to (i+1,j)$ ($i<n$) and vertical segments $(i,j)\to (i,j+1)$ ($j<n$) 
joining neighboring vertices, as well as all the second diagonals $(i,j+1)\to (i+1,j)$ ($i,j<n$). This edge set is completed by a set of \emph{oriented} boundary edges, 
with the following prescribed orientations:
\begin{itemize}
\item{West boundary: horizontal edges oriented from $(0,j)$ to $(1,j)$, $j=1,2,\dots,n$ and diagonal edges oriented from $(0,j+1)$ to $(1,j)$ for 
$j=1,2,\dots n-1$;}
\item{South boundary: vertical edges oriented from $(i,1)$ to $(i,0)$, $i=1,2,\dots,n$ and diagonal edges oriented from $(i,1)$ to $(i+1,0)$ for 
$i=1,2,\dots n-1$;}
\item{East boundary: horizontal edges oriented from $(n+1,j)$ to $(n,j)$, $j=1,2,\dots,n$ and diagonal edges oriented from $(n+1,j)$ to $(n,j+1)$ for 
$j=1,2,\dots n-1$;}
\item{North boundary: vertical edges oriented from $(i,n)$ to $(i,n+1)$, $i=1,2,\dots,n$ and diagonal edges oriented from $(i+1,n)$ to $(i,n+1)$ for 
$i=1,2,\dots n-1$;}
\end{itemize}  
The boundary edge set itself is finally completed by two diagonal corner edges, and we distinguish two types of DWBC, referred to as DWBC1 and DWBC2 
respectively, depending on the orientation of these corner edges:
\begin{itemize}
\item{DWBC1: the diagonal edge oriented from $(0,n+1)$ to $(1,n)$ and the diagonal edge oriented from $(n,1)$ to $(n+1,0)$;}

or  
\item{DWBC2: the diagonal edge oriented from $(1,n)$ to $(0,n+1)$ and the diagonal edge oriented from $(n+1,0)$ to $(n,1)$.}
\end{itemize}

\begin{figure}
\begin{center}
\includegraphics[width=15cm]{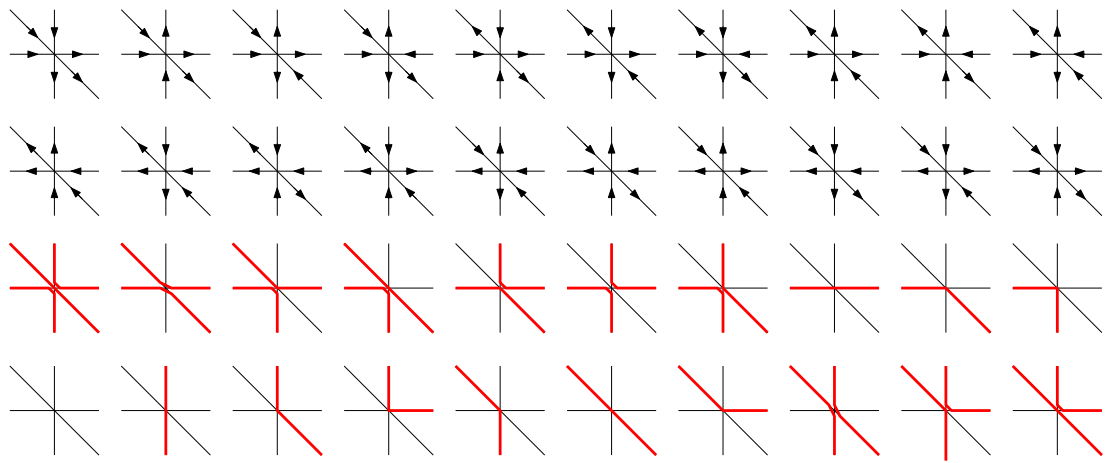}
\end{center}
\caption{\small Top: The $20$ possible environments allowed by the ice rule for a node of the triangular lattice.  Bottom: the equivalent $20$ vertices 
in the osculating path language. A path edge is drawn whenever the underlying orientation runs from North, Northwest or West to East, Southeast or South.
Path steps are then concatenated into non-crossing paths. 
}
\label{fig:twentyV}
\end{figure}
The domain thus defined is clearly a portion of triangular lattice where each inner node $(i,j)$, $i,j=1,2,\dots,n$ is incident to six edges. A configuration 
of the 20V model on this domain consists in the assignment of an orientation to all the inner edges satisfying the (triangular) \emph{ice rule} condition that
each node is incident to three outgoing and three incoming edges. Figures \ref{fig:DWBC1} 
and  \ref{fig:DWBC2} show examples of configurations corresponding to the DWBC1 and DWBC2 ensembles respectively. The ice rule gives rise
to exactly 20 possible environments around a given node, as displayed in Fig.~\ref{fig:twentyV}, hence the name of the model. In the following,
unless otherwise stated, we will be interested in enumerating the 20V configurations without discrimination on the possible node
environments. In other words, we attach the same weight $1$ to each of the 20 vertices of the model.

As in the case of the 6V model, the edge orientations of a configuration of the 20V model may be coded bijectively by configurations of so-called
\emph{osculating paths} visiting the edges of the lattice and obtained as follows:
we first assign to each edge of the lattice a natural orientation, namely from West to East for the horizontal edges, from North to South for the vertical edges, and
from Northwest to Southeast for the diagonal edges. Each edge of the lattice is then covered by a path step if and only of its actual orientation matches the natural orientation.
Note that the path steps are de facto naturally oriented from North, Northwest or West (NW for short) to East, Southeast or South (SE) and the ice rule ensures that the number of
paths steps coming from NW at each node is equal to that leaving towards SE. This allows to concatenate the path steps into global paths. When four or six path steps
are incident to a given node, the prescription for concatenation is the unique choice ensuring that the paths \emph{do not cross each other}, even though they meet at the node at hand.
Such paths are called ``osculating". Note that no two paths can share a common edge. 
By construction, the path configuration consists of $2n$ (respectively $2n-1$) paths in the DWBC1 (respectively DWBC2) ensemble,
connecting the $2n$ edges of the West boundary plus the upper-left corner edge (respectively the  $2n-1$ edges of the West boundary) to the 
$2n$ edges of the South boundary plus the lower-right corner edge (respectively the  $2n-1$ edges of the South boundary) without crossing.
Figures \ref{fig:DWBC1} 
and  \ref{fig:DWBC2} show examples of such osculating path configurations.

The simplest question we may ask about 20V configuration with DWBC is that of the number $A_n=Z^{20V}(n)$ of configurations for a given $n$. 
First we note that this number is \emph{the same} for the prescriptions DWBC1 and DWBC2 due to a simple duality between the two models:
performing a rotation by $180^\circ$ in the plane sends a configuration of arrows obeying the DWBC1 prescription
to one obeying the DWBC2 and conversely (the symmetry being an involution). Indeed, the ice rule is invariant under this rotation and the boundary conditions
are unchanged, except for the orientation of the corner edges which are reversed. This gives a bijection between the configurations in the two ensembles
which thus have the same cardinality. In the osculating path language, the DWBC2 path configuration is obtained by taking the complement of the DWBC1 
path configuration (i.e. covering uncovered edges and conversely) and rotating it by $180^\circ$ . The configuration in Fig.~\ref{fig:DWBC2} is the
image of that of Fig.~\ref{fig:DWBC1} by this bijection. The distinction between the two ensembles will still be significant when we address more refined question 
in the next sections.

\begin{figure}
\begin{center}
\includegraphics[width=11cm]{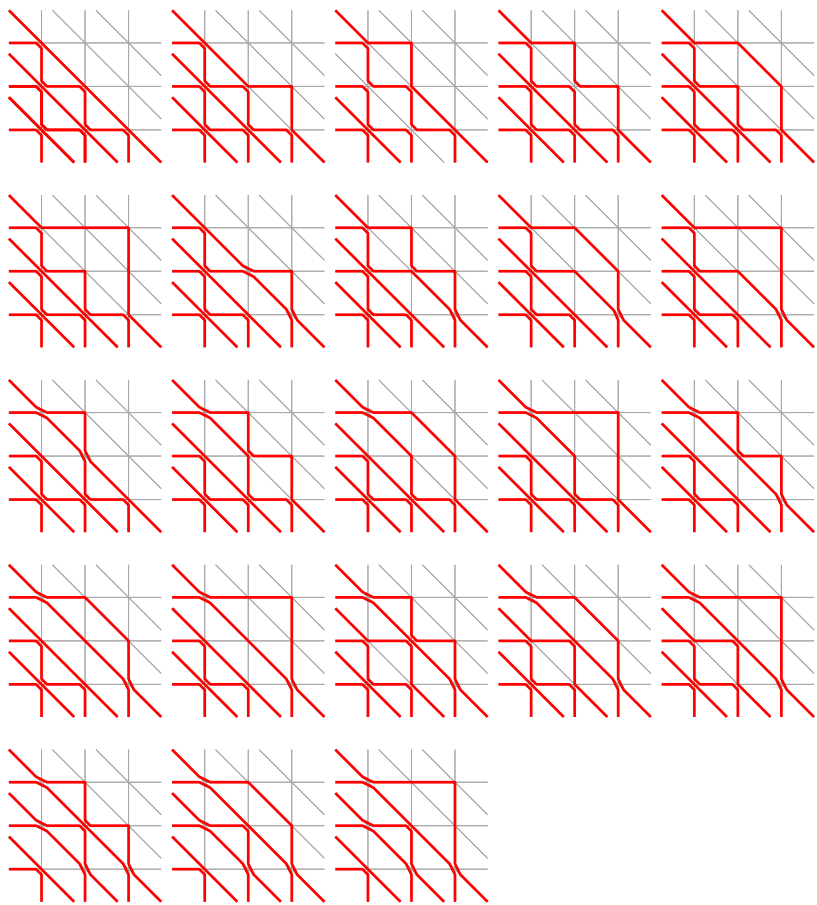}
\end{center}
\caption{\small The $23$ configurations of the 20V model with DWBC1 for $n=3$, represented in the osculating path language.
}
\label{fig:23configs}
\end{figure}
The sequence of the first values of the numbers  $A_n=Z^{20V}(n)$ for $n$ up to 8 are listed 
in \eqref{seqone}. The $23$ configurations for $n=3$ are represented for illustration in Fig.~\ref{fig:23configs} 
\subsection{General properties}
\begin{figure}
\begin{center}
\includegraphics[width=13cm]{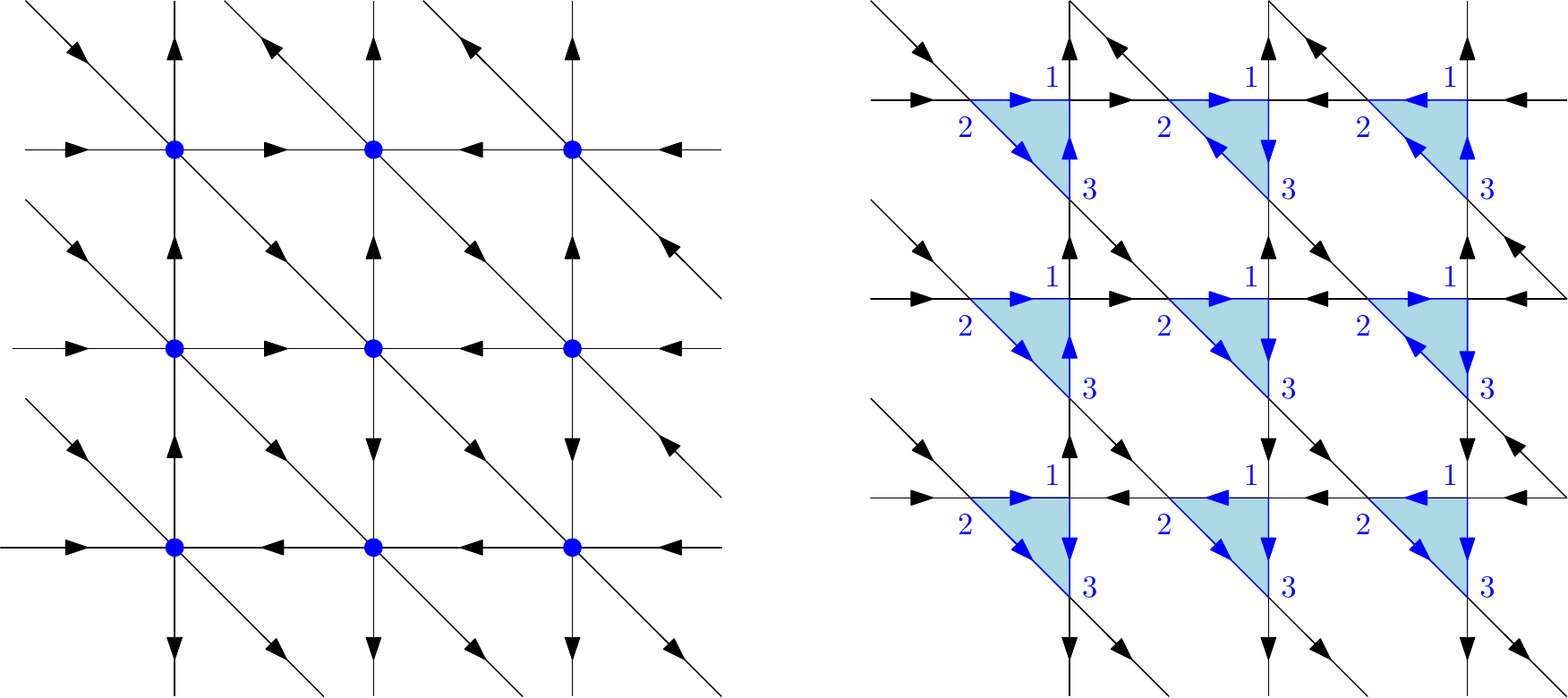}
\end{center}
\caption{\small The deformation of the triangular lattice into the Kagome lattice by sliding the horizontal lines to the North, splitting each node into a small triangle (colored
in light blue). The Kagome lattice is naturally split into three sublattices denoted $1$,
$2$ and       
$3$ as shown. Each configuration of the 20V model (left) may be completed into a configuration satisfying 
the ice rule at each vertex of the Kagome lattice by some (non-unique) appropriate choice of orientation for the newly created edges (blue arrows).
}
\label{fig:Kagome}
\end{figure}
Following Baxter \cite{Baxter}, we may transform the 20V model into an ice model on the Kagome lattice as follows:
starting from our portion of triangular lattice and splitting each node into a small triangle, say by slightly sliding each horizontal line to the North,
results in a portion of Kagome lattice, as shown in Fig.~\ref{fig:Kagome}. Clearly, any orientation of the edges of the Kagome lattice satisfying the ice rule
(i.e. with two incoming and two outgoing edges incident to each node) results into a configuration \emph{where the six 
edges of the original triangular lattice} satisfy the ice rule around any small triangle replacing an original node. Conversely, any 
orientation satisfying the ice rule on the original triangular lattice may be completed via some appropriate choice of orientation of the newly formed edges so as to create 
an ice model configuration on the Kagome lattice (note that the choice of orientation for the new edges is not unique in general).
This construction allows to rephrase our 20V 
model in terms of an ice model on the Kagome lattice. Let us now discuss how to recover the desired weight $1$ per vertex of the 20V model by some appropriate
weighting of the vertex configurations around each node of the Kagome lattice.
The Kagome lattice is naturally decomposed into three sublattices, denoted $1$, $2$ and $3$ with the following choice: 
\begin{itemize}
\item{lattice $1$: vertices at the crossing of a horizontal and a vertical line;}
\item{lattice $2$: vertices at the crossing of a horizontal and a diagonal line;}        
\item{lattice $3$: vertices at the crossing of a vertical and a diagonal line.}
\end{itemize}
\label{sec:20to6}
\begin{figure}
\begin{center}
\includegraphics[width=10cm]{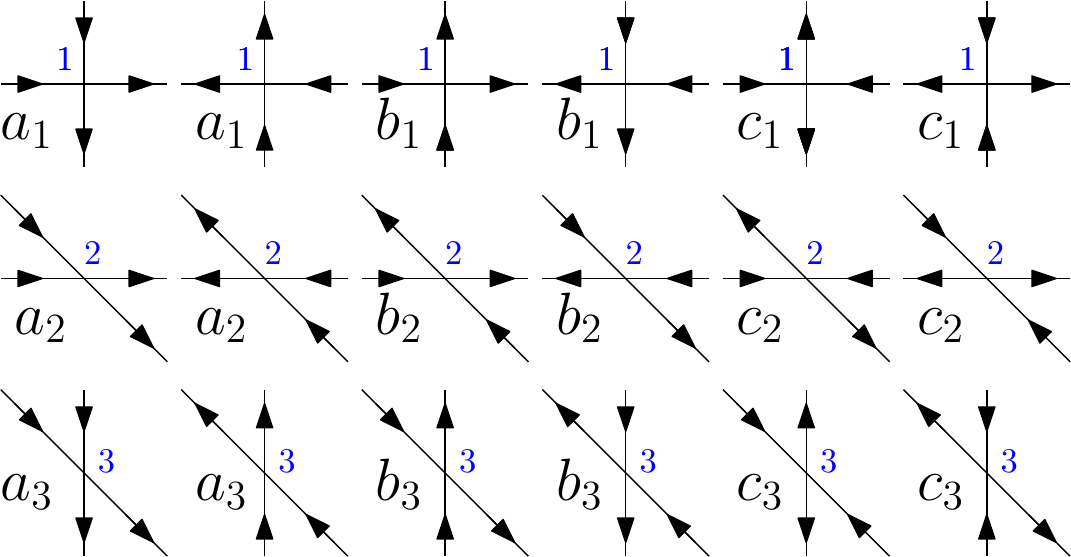}
\end{center}
\caption{\small The weight denomination for the three copies of 6V models on the Kagome lattice. The index of the weights corresponds to that of the underlying 
sublattice. Vertices related by a global reversing of all arrows are chosen to have the same weight.
}
\label{fig:KagomeWeights}
\end{figure}
Due to the ice rule, each vertex of sublattice $1$ (respectively $2$ and $3$) matches one of the six vertex configurations of the 6V model and we may naturally weight
these configurations with three weights $a_1,b_1,c_1$ (respectively $a_2,b_2,c_2$ and $a_3,b_3,c_3$) according to the rules of Fig.~\ref{fig:KagomeWeights}.
In order to recover the desired weight $1$ for each vertex of the 20V model, the weights of the Kagome model must satisfy the following $10$ relations\footnote{Note that
our definition of the weights on the Kagome lattice differ from that of Baxter in \cite{Baxter} upon the exchange $a_1\leftrightarrow b_1$. The 20 relations for the 20 possible 
vertices reduce to 10 by symmetry.}
\begin{equation}
\begin{split}
1&=a_1a_2a_3=b_1a_2b_3=b_1a_2c_3=c_1a_2a_3=b_1c_2a_3=b_1b_2a_3\\ &= a_1b_2c_3+c_1c_2b_3=a_1b_2b_3+c_1c_2c_3=c_1b_2b_3+a_1c_2c_3=c_1b_2c_3+a_1c_2b_3\\
\end{split}
\label{eq:weightsone}
\end{equation}
\begin{figure}
\begin{center}
\includegraphics[width=10cm]{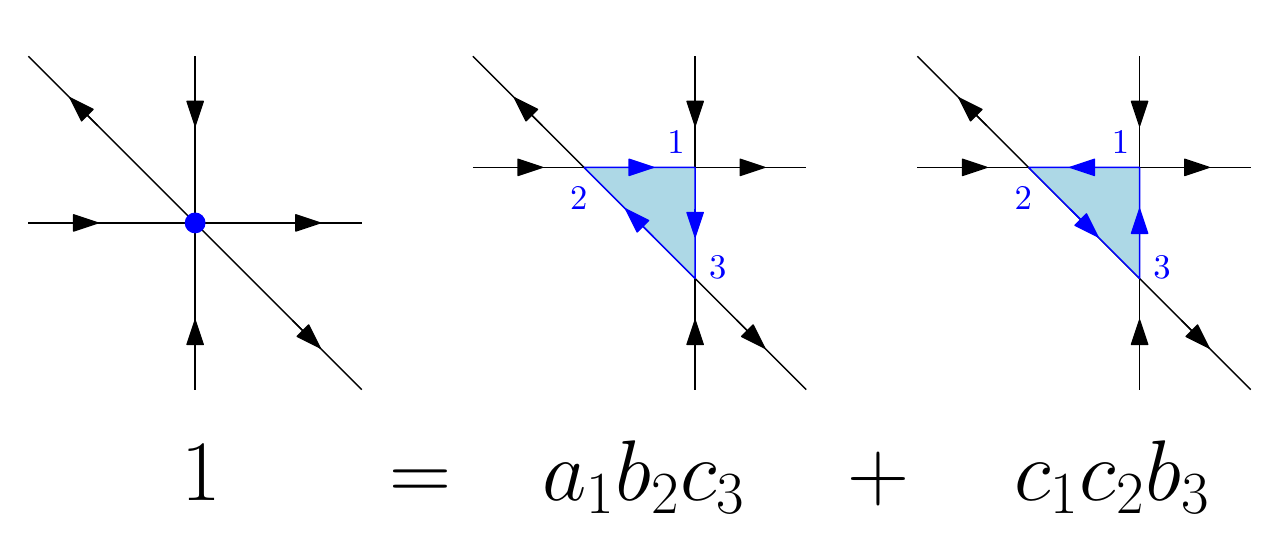}
\end{center}
\caption{\small The relation $a_1b_2c_3+c_1c_2b_3=1$ ensuring a weight $1$ for the vertex of the 20V model shown on the left, as obtained by
summing over the two possible orientations for the edges of the central small triangle in the equivalent Kagome lattice setting (right).
}
\label{fig:rel1}
\end{figure}
For instance, the relation $a_1b_2c_3+c_1c_2b_3=1$ comes from the summation over the two possible orientations in the small triangle shown in Fig.~\ref{fig:rel1}.
A possible choice of solution for the system of equations \eqref{eq:weightsone} is
\begin{equation}
(a_1,b_1,c_1)=\frac{\alpha}{2^{1/3}}(1,\sqrt{2},1)\ , \qquad (a_2,b_2,c_2)=\frac{\beta}{2^{1/3}}(\sqrt{2},1,1)\ , \qquad (a_3,b_3,c_3)=\frac{\gamma}{2^{1/3}}(\sqrt{2},1,1)
\label{eq:solweights} 
\end{equation} 
for any choice of $\alpha$, $\beta$ and $\gamma$ such that $\alpha\beta\gamma=1$.
\begin{figure}
\begin{center}
\includegraphics[width=12cm]{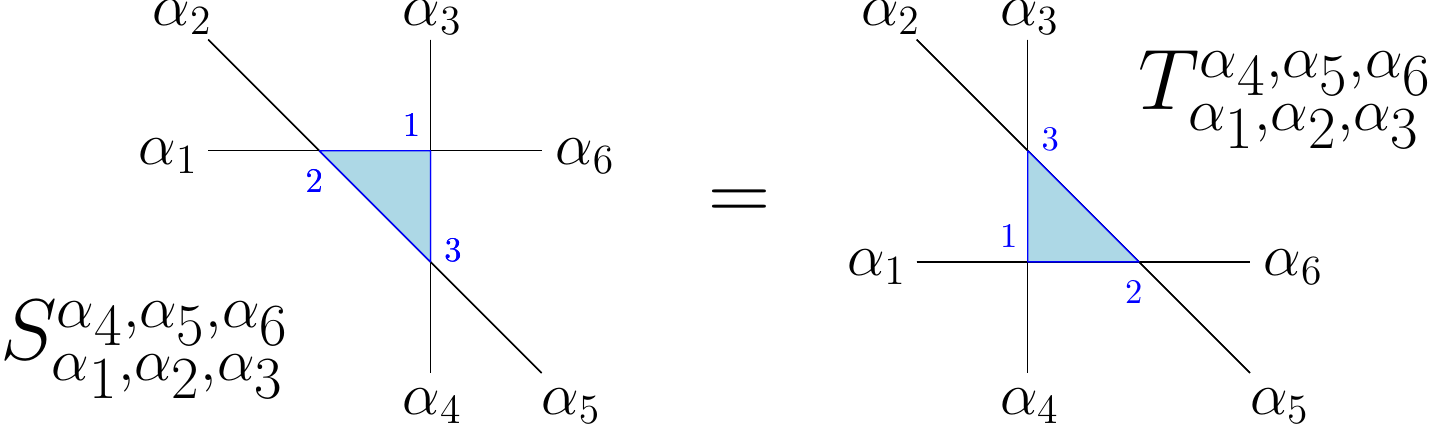}
\end{center}
\caption{\small A schematic picture of the Yang-Baxter equation $S_{\alpha_1,\alpha_2,\alpha_3}^{\alpha_4,\alpha_5,\alpha_6}=T_{\alpha_1,\alpha_2,\alpha_3}^{\alpha_4,\alpha_5,\alpha_6}$. For fixed 
orientations $\alpha_1,\cdots,\alpha_6$, $S_{\alpha_1,\alpha_2,\alpha_3}^{\alpha_4,\alpha_5,\alpha_6}$ and $T_{\alpha_1,\alpha_2,\alpha_3}^{\alpha_4,\alpha_5,\alpha_6}$ are obtained by summing over the orientations of the edges of the central triangle allowed by the ice rule on the Kagome lattice, with their associated weights
of Fig.~\ref{fig:KagomeWeights}.
}
\label{fig:YB}
\end{figure}
A very efficient tool in solving the 6V model is the use of the so called Yang-Baxter relations which allow to deform and eventually unravel the underlying lattice into a simpler graph.
In the above Kagome lattice setting, denoting by $\alpha_1,\cdots,\alpha_6$ the six orientations around a small triangle as shown in Fig.~\ref{fig:YB}, with $\alpha_i=1$ if
the orientation matches the natural (from NW to SE) orientation and $0$ otherwise, these relations ensure that the weight $S_{\alpha_1,\alpha_2,\alpha_3}^{\alpha_4,\alpha_5,\alpha_6}$
obtained by summing over the possible orientations of the edges of the small triangle is equal, for any choice of the $\alpha_i$'s, to that, $T_{\alpha_1,\alpha_2,\alpha_3}^{\alpha_4,\alpha_5,\alpha_6}$,
obtained by sliding the diagonal line (passing through vertices of sublattices $2$ and $3$) to the other side of the node of the sublattice $1$ (see Fig.~\ref{fig:YB}).
In terms of the weights $(a_i,b_i,c_i)$, it is easily checked that this equality holds if and only if we have the three relations:
\begin{equation}
(a_1b_2-b_1a_2)c_3+c_1c_2b_3=0\ , \quad (a_1 b_3-b_1a_3)c_2+c_1c_3b_2=0\ , \quad (b_2b_3-a_2a_3)c_1+c_2c_3 a_1=0\ .
\label{eq:YB}
\end{equation}
Note that these relations are in practice weaker than the relations \eqref{eq:weightsone} in the sense that imposing \eqref{eq:weightsone} automatically implies
\eqref{eq:YB}. For instance, the relation $(a_1b_2-b_1a_2)c_3+c_1c_2b_3=0$ is a direct consequence of the two identities $b_1a_2c_3=a_1b_2c_3+c_1c_2b_3=1$.
In particular, \eqref{eq:YB} is satisfied by the solution \eqref{eq:solweights}, as easily verified by a direct computation.

\subsection{Integrable weight parametrization}
\begin{figure}
\begin{center}
\includegraphics[width=8cm]{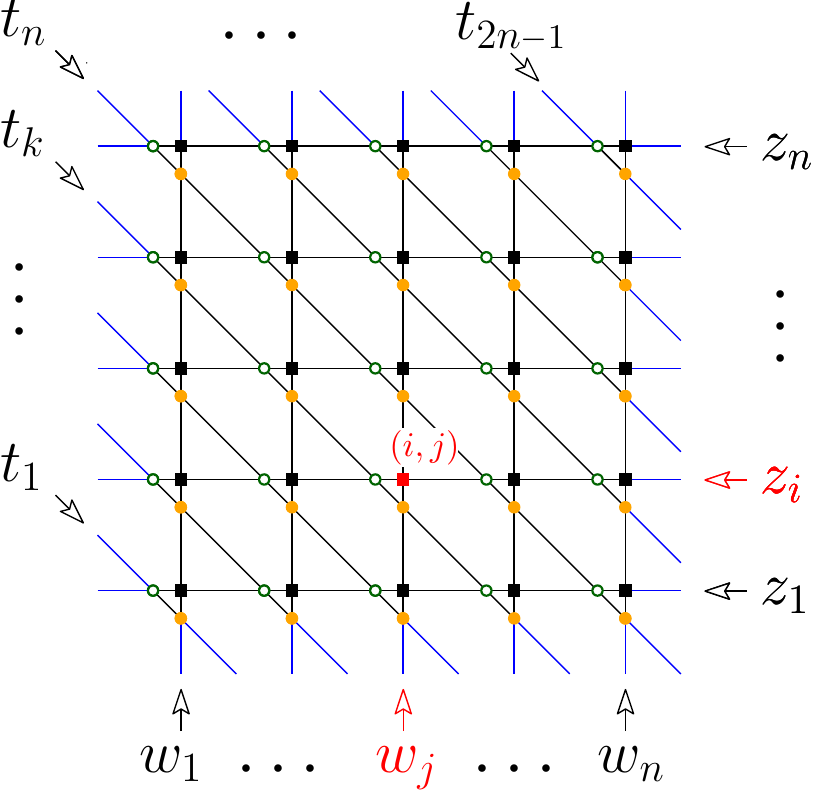}
\end{center}
\caption{\small Spectral parameters for the Kagome lattice version of the 20V model.
}
\label{fig:spectral}
\end{figure}
It is useful to introduce more general weights for our 20V model, or equivalently for its Kagome reformulation, by introducing so called spectral parameters
in the following way, mimicking the well known use of spectral parameters for the 6V model. 
Let us number the horizontal lines of our lattice by $i=1,2,\dots, n$ from bottom to top and attach a (complex) parameter $z_i$ to the $i$'th line. Similarly,
we label the vertical lines by $j=1,2,\dots, n$ from left to right and attach a parameter $w_j$ to the $j$'th line. Finally,
the diagonal lines are labeled by  $k=1,2\dots, 2n-1$ from bottom left to top right and we attach a parameter $t_k$ to the $k$-th line.
This labeling induces a similar labeling for the horizontal, vertical and diagonal lines of the Kagome lattice, see Fig.~\ref{fig:spectral}. Each node of the sublattice $1$ is then at the crossing of a horizontal
and a vertical line, hence naturally labeled by a pair $(i,j)$ with $i,j=1,2,\dots,n$. Similarly, each node of the sublattice $2$ is labeled by a pair $(i,k)$, with $i=1,2,\dots,n$ and $k=i,i+1,\dots,i+n-1$, 
while each node of the sublattice $3$ corresponds to a pair $(k,j)$ with $j=1,2,\dots,n$ and $k=j,j+1,\dots,j+n-1$. This allows us to introduce \emph{non-homogeneous} weights
 $(a_1(i,j),b_1(i,j),c_1(i,j))$ for the configurations around vertices of the sublattice $1$ according to the dictionary of Fig.~\ref{fig:KagomeWeights}, and similarly weights 
 $(a_2(i,k),b_2(i,k),c_2(i,k))$ and $(a_3(k,j),b_3(k,j),c_3(k,j))$.  

Introducing the notations
\begin{equation*}
A(u,v)=u-v\ , \qquad B(u,v)=q^{-2}\, u-q^2\, v\ , \qquad C(u,v)=(q^2-q^{-2})\, \sqrt{u\, v}\ ,
\end{equation*}
with $u$, $v$ and $q$ some complex numbers, we consider the following \emph{integrable weight parametrization}:
\begin{equation}
\begin{matrix}
a_1(i,j)=A(z_i,w_j)\ , \hfill &b_1(i,j)=B(z_i,w_j)\ , \hfill &c_1(i,j)=C(z_i,w_j)\ , \hfill \\
a_2(i,k)=A(q\, z_i,q^{-1}\, t_k)\ , \hfill &b_2(i,k)=B(q\, z_i,q^{-1}\, t_k)\ , \hfill &c_2(i,k)=C(q\, z_i,q^{-1}\, t_k)\ , \hfill \\
a_3(k,j)=A(q\, t_k,q^{-1}\, w_j)\ , \hfill &b_3(k,j)=B(q\, t_k,q^{-1}\, w_j)\ , \hfill &c_3(k,j)=C(q\, t_k,q^{-1}\, w_j)\ , \hfill \\
\end{matrix}
\label{eq:integrable}
\end{equation}
where the complex numbers $z_i$, $i-1,2,\dots,n$, $w_j$, $j=1,2,\dots,n$, and $t_k$, $k=1,2,\dots,2n-1$ are arbitrarily fixed spectral parameters. 
The main feature of this integrable parametrization is that, for any choice of the spectral parameters, the Yang Baxter relations \eqref{eq:YB} are automatically satisfied
\emph{for any triple} $(i,j,k)$ in \eqref{eq:integrable}, as easily checked by a direct computation.  

The solution \eqref{eq:solweights}  of \eqref{eq:weightsone} may be recovered in this framework by choosing 
\begin{equation}
t_k=t\ , \qquad z_i=q^6\, t\ , \qquad w_j=q^{-6}\, t
\label{eq:tzw}
\end{equation}
for all $k$, $i$ and $j$, 
leading to the homogeneous weights:
\begin{equation*}
\begin{matrix}
a_1=(q^6-q^{-6})t , \hfill &b_1=(q^4-q^{-4})t\ , \hfill &c_1=(q^2-q^{-2})t\ , \hfill \\
a_2=(q^7-q^{-1})t\ , \hfill &b_2=(q^5-q)t\ , \hfill &c_2=(q^5-q)t\ , \hfill \\
a_3=(q-q^{-7})t\ , \hfill &b_3=(q^{-1}-q^{-5})t\ , \hfill &c_3=(q^{-1}-q^{-5})t\ . \hfill \\
\end{matrix}
\end{equation*}
Upon taking the particular value
\begin{equation}
q=\rm{e}^{{\rm i}\pi/8}\ ,
\label{eq:valq}
\end{equation}
these weights reduce, using $q^8=-1$ and $(q^4-q^{-4})=\sqrt{2}\, (q^2-q^{-2})$, to 
\begin{equation}
\begin{matrix}
(a_1,b_1,c_1)=(q^2-q^{-2})t\ (1,\sqrt{2},1)\ ,\hfill \\
(a_2,b_2,c_2)=q^3\, (q^2-q^{-2})t\ (\sqrt{2},1,1) \ ,\hfill \\
(a_3,b_3,c_3)=q^{-3}\, (q^2-q^{-2})t\  (\sqrt{2},1,1) \ ,\hfill \\
\end{matrix}
\label{eq:homval}
\end{equation}
a form which matches precisely that of \eqref{eq:solweights} with $\alpha=1$, $\beta=q^3$ and $\gamma=q^{-3}$  whenever
$\left((q^2-q^{-2})t\right)^3=1/2$, namely, say
\begin{equation}
t=\frac{1}{2^{1/3}\, (q^2-q^{-2})}=-\frac{{\rm i}}{2^{5/6}}\ .
\label{eq:valt}
\end{equation}

\section{Mapping to a 6V model with Domain Wall Boundary Conditions}
\label{sec:20to6}

\subsection{Unraveling the 20V configurations}\label{sec:unravel}
\begin{figure}
\begin{center}
\includegraphics[width=14cm]{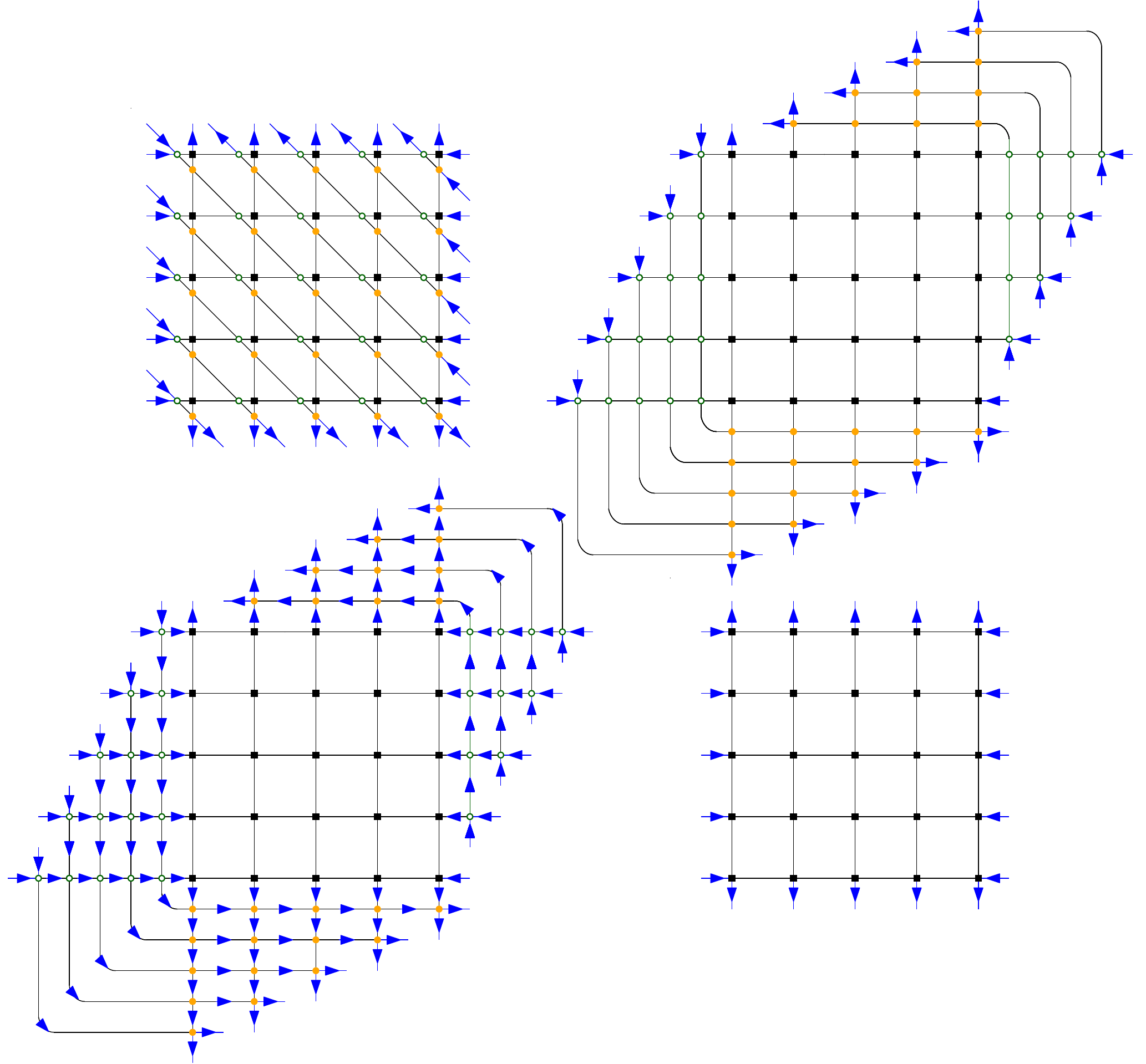}
\end{center}
\caption{\small The unraveling of a configuration of the 20V model with DWBC1 (top left). Using the Yang Baxter property allows to deform the diagonal lines 
and expel them out of the central square grid (top right). Note the that the main diagonal is expelled towards the lower-left of the square grid, a choice adapted to the DWBC1 prescription.
Due to the ice rule and the boundary conditions, the orientation of all the edges outside of the central square grid are entirely fixed (lower left), leaving as only degrees of freedom
the orientation of the edges inside the central square grid, reproducing a 6V model with DWBC (lower right). 
}
\label{fig:unraveling}
\end{figure}
We now return to our 20V model with, say the DWBC1 prescription and with a weight $1$ per vertex and consider its Kagome equivalent formulation
with the weights given by \eqref{eq:solweights}. As already mentioned, as solution of the equation \eqref{eq:weightsone}, these weights automatically satisfy the 
conditions \eqref{eq:YB} ensuring the Yang Baxter property. This allows to deform the lattice by expelling the diagonal lines from the $n\times n$ square grid
as shown in Fig.~\ref{fig:unraveling}. The diagonal lines with index $k\leq n$ are expelled towards the lower-left of the square grid and those with index $k>n$ towards the upper-right. 
The choice for the main diagonal ($k=n$) is adapted to the DWBC1 prescription. 
For the DWBC2 prescription, the proper choice would be to move this diagonal towards the upper-right instead.
In the deformed configuration, all the vertices of the sublattices $2$ and $3$ have been expelled outside of the $n\times n$ square grid which contains only
vertices of the sublattice $1$. More interestingly, due to the ice rule and to the prescribed boundary conditions, the orientations of all the edges outside the 
central  $n\times n$ square grid are entirely \emph{fixed} (see Fig.~\ref{fig:unraveling}), and all correspond to configurations of "type $a$", namely receive the weight 
$a_2$ if they belong to sublattice $2$ and $a_3$ is they belong to the sublattice $3$. This leads to a global contribution $(a_2\, a_3)^{n^2}$ while the remaining configuration
is that of a standard 6V model on the $n\times n$ square grid with the celebrated Domain Wall Boundary Conditions. 
We immediately deduce the following:
\begin{figure}
\begin{center}
\includegraphics[width=10cm]{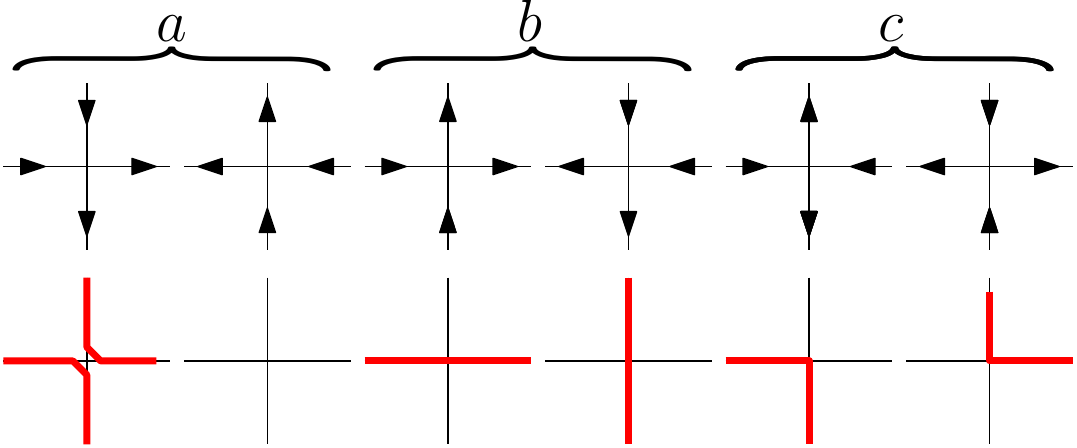}
\end{center}
\caption{\small Vertex configurations of the 6V model on the square grid and the associated weights $a$, $b$ and $c$ (top row) and their
equivalent osculating path representation (bottom row).  
}
\label{fig:6VWeights}
\end{figure}
\begin{thm}
\label{An6V}
The number $A_n$ of configurations of the 20V model with DWBC1,2 on an $n\times n$ grid reads
\begin{equation*}
A_n=Z^{20V}(n)=Z^{6V}_{\left[1,\sqrt{2},1\right]}(n)\ .
\end{equation*} 
where $Z^{6V}_{\left[a,b,c\right]}(n)$ denotes the partition function of the 6V model on an $n\times n$ square grid with DWBC and weights $(a,b,c)$ 
according to the dictionary of Fig.~\ref{fig:6VWeights}. 
\end{thm}
\begin{proof}
We indeed have 
\begin{equation*}
\begin{split}
A_n=Z^{20V}(n)&=(a_2\, a_3)^{n^2} Z^{6V}_{\left[a_1,b_1,c_1\right]}(n)=
\left(\beta\, \gamma\left(\frac{\sqrt{2}}{2^{1/3}}\right)^2\right)^{n^2} Z^{6V}_{\left[\alpha\frac{1}{2^{1/3}},\alpha\frac{\sqrt{2}}{2^{1/3}},\alpha\frac{1}{2^{1/3}}\right]}(n)\\
&= Z^{6V}_{\left[\alpha\, \beta\, \gamma\left(\frac{\sqrt{2}}{2^{1/3}}\right)^2\, \frac{1}{2^{1/3}},\alpha\, \beta\, \gamma\left(\frac{\sqrt{2}}{2^{1/3}}\right)^2\, \frac{\sqrt{2}}{2^{1/3}},\alpha\, \beta\, \gamma\left(\frac{\sqrt{2}}{2^{1/3}}\right)^2\, \frac{1}{2^{1/3}}\right]}(n)\\
\end{split}
\end{equation*} 
where we used the multiplicative nature of the weights
to redistribute the prefactor within the weights of the $n^2$ nodes of the sublattice $1$. Since $\alpha\, \beta\, \gamma=1$, the theorem follows.
\end{proof} 

\begin{figure}
\begin{center}
\includegraphics[width=11cm]{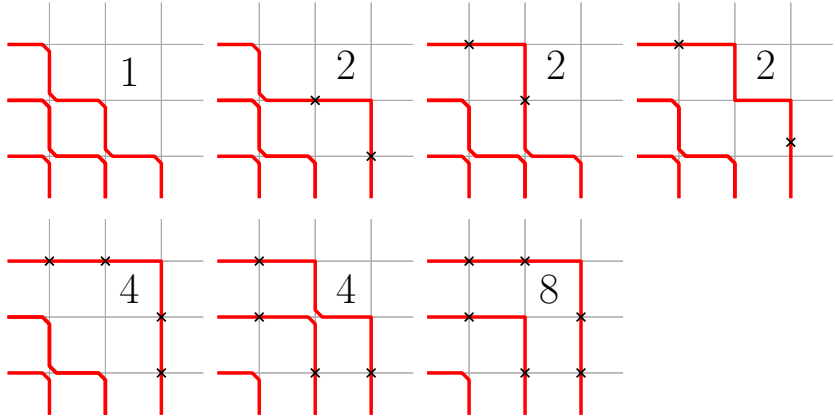}
\end{center}
\caption{\small Illustration of the property $Z^{6V}_{\left[1,\sqrt{2},1\right]}(3)=1+2+2+2+4+4+8=23$ obtained by listing all the 6V model configurations with DWBC (here in the osculating path language)
and their associated weight (as indicated for each configuration), corresponding to attaching a factor $\sqrt{2}$ to each node traversed vertically or horizontally by a path (as shown by cross marks).
}
\label{fig:23via6V}
\end{figure}
Using Theorem~\ref{An6V} and straightforward generalizations, we will recourse to known results on the 6V model with DWBC to 
address a number of enumeration results for the 20V model. In the following, we will mainly use the osculating path formulation of the 20V model and the corresponding one
for the 6V model according to the correspondence of Fig.~\ref{fig:6VWeights}. Fig.~\ref{fig:23via6V} shows how to recover the value $A_3=Z^{20V}(3)=23$ from that of $Z^{6V}_{\left[1,\sqrt{2},1\right]}(3)$
in the osculating path framework.

\subsection{Refined enumeration}
\label{sec:refinedZ20}
A refined enumeration of the 20V model configurations consists, in the osculating path language, in keeping track of the position $i=\ell$ ($\ell=1,2,\dots,n$) where the uppermost 
path\footnote{The uppermost path corresponds to the $2n$-th path from the bottom for DWBC1 and to the $(2n-1)$-th path for DWBC2.}
first hits the vertical line $j=n$. Alternatively, $\ell-1$ is the number of occupied inner vertical edges in the last column.
We denote by $Z^{20V_{BC1}}_\ell= Z^{20V_{BC1}}_\ell(n)$ the number of configurations with a given $\ell$ for the DWBC1 prescription 
and $Z^{20V_{BC2}}_\ell$ this number for the  DWBC2 prescription. These numbers are encoded in the generating functions 
\begin{equation*}
\begin{split}
&\hat{Z}^{20V_{BC1}}(\tau)=\sum_{\ell=1}^n Z^{20V_{BC1}}_\ell \tau^{\ell-1}\ , \\
&\hat{Z}^{20V_{BC2}}(\tau)=\sum_{\ell=1}^n Z^{20V_{BC2}}_\ell  \tau^{\ell-1}\ , \\
\end{split}
\end{equation*}
(with an implicit $n$-dependence) which clearly satisfy $\hat{Z}^{20V_{BC1}}(1)=\hat{Z}^{20V_{BC2}}(1)=Z^{20V}(n)$. 
Similarly we denote by $Z^{6V}_{\left[1,\sqrt{2},1\right];\ell}= Z^{6V}_{\left[1,\sqrt{2},1\right];\ell}(n)$ the number of configurations of the 6V model with DWBC 
for which the uppermost osculating path first hits the vertical line $j=n$ at position $i=\ell$ and set
\begin{equation}
\hat{Z}^{6V}_{\left[1,\sqrt{2},1\right]}(\sigma)=\sum_{\ell=1}^n Z^{6V}_{\left[1,\sqrt{2},1\right];\ell} \sigma^{\ell-1} 
\label{ref6V}
\end{equation}
with $\hat{Z}^{6V}_{\left[1,\sqrt{2},1\right]}(1)=Z^{6V}_{\left[1,\sqrt{2},1\right]}(n)$.

Let us now show the following:
\begin{thm}
\label{20to6}
The generating polynomials $\hat{Z}^{20V_{BC1,2}}(\tau)$ for the refined 20V model are determined by the relations 
\begin{equation}
\hat{Z}^{20V_{BC2}}(\tau)=\hat{Z}^{6V}_{\left[1,\sqrt{2},1\right]}\left(\frac{1+\tau}{2}\right)=\hat{Z}^{20V_{BC1}}(0)+\frac{1+\tau}{2\tau}\left(\hat{Z}^{20V_{BC1}}(\tau)-\hat{Z}^{20V_{BC1}}(0)\right)\ .
\label{eq:20to6}
\end{equation}
Equivalently, coefficient-wise:
\begin{equation*}
\begin{split}
&Z^{20V_{BC1}}_1=Z^{6V}_{\left[1,\sqrt{2},1\right];1}\quad \hbox{and}\quad   Z^{20V_{BC1}}_\ell=\sum_{m=\ell}^n {m-2\choose \ell-2} \frac{1}{2^{m-2}} Z^{6V}_{\left[1,\sqrt{2},1\right];m}
\quad \hbox{for}\ \ell\geq 2\ ,\\
& Z^{20V_{BC2}}_\ell=\sum_{m=\ell}^n {m-1\choose \ell-1} \frac{1}{2^{m-1}} Z^{6V}_{\left[1,\sqrt{2},1\right];m}\ .\\
\end{split}
\end{equation*}
\end{thm}
Note that the second relation in \eqref{eq:20to6} may alternatively be rewritten as 
\begin{cor}
\label{cor20to6}
\begin{equation}
\hat{Z}^{20V_{BC1}}(\tau)=\frac{2\tau}{1+\tau}\ \hat{Z}^{6V}_{\left[1,\sqrt{2},1\right]}\left(\frac{1+\tau}{2}\right)+\frac{1-\tau}{1+\tau}\ \hat{Z}^{6V}_{\left[1,\sqrt{2},1\right]}(0)
\ .
\label{eq:6to20BC1}
\end{equation}
\end{cor}

\begin{figure}
\begin{center}
\includegraphics[width=15cm]{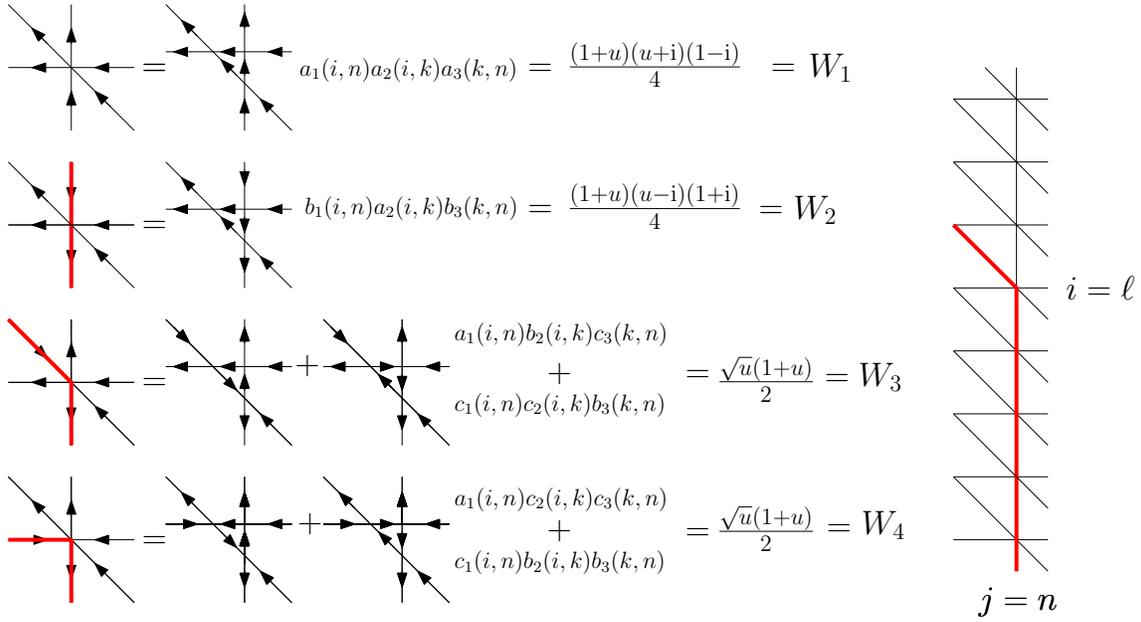}
\end{center}
\caption{\small Left: modification of the weights of the 20V model with DWBC2 when performing the change $w_n\to w_n\, u$ for the spectral parameter attached
to the last column from the special values \eqref{eq:tzw},  \eqref{eq:valq} and \eqref{eq:valt}. We display only the four vertices which may appear in the last column due to the boundary condition at the East
boundary.
The weights are easily computed in the Kagome formulation,  with the result $W_1,\dots, W_4$ shown, satisfying $W_i\to 1$ when $u\to 1$, as required. Right:
the configuration in the the last column is, from bottom to top, made of a sequence of vertices weighted by $W_2$, then of a single vertex with weight $W_3$ or $W_4$ (hitting point) and
finally of a complementary sequence of vertices weighted by $W_1$.
}
\label{fig:newweights}
\end{figure}
To prove Theorem \ref{20to6}, let us start with the simplest case of DWBC2. The generating function $\hat{Z}^{20V_{BC2}}(\tau)$ may easily be obtained, in the equivalent Kagome formulation
of the 20V model, by slightly modifying the spectral parameter $w_n$ for the last column. Choosing the integrable parametrization \eqref{eq:integrable} 
for the Kagome vertex weights with $q$ as in \eqref{eq:valq}, $t_k=t$ as in \eqref{eq:valt} for all $k$, $z_i=q^6\, t$ for all $i$ and $w_j=q^{-6}\, t$ for all $j<n$
while $w_n=q^{-6}\, t\ u$ for some parameter $u$, only the weights $(a_1(i,n),b_1(i,n),c_1(i,n))$ and $(a_3(k,n),b_3(k,n),c_3(k,n))$ are modified with respect to the 
homogeneous values of \eqref{eq:homval}. The new values are
\begin{equation*}
\begin{matrix}
a_1(i,n)=(q^2\, u-q^{-2})t\ , \hfill &b_1(i,n)=(q^4-q^{-4}\, u)t\ , \hfill &\quad c_1(i,n)=(q^2-q^{-2})\, \sqrt{u}\, t\ ,\hfill \\
a_3(i,n)=q^{-3}(q^4-q^{-4}\, u)t\ , \hfill &b_3(i,n)=q^{-3}(q^2-q^{-2}\, u)t\ ,\hfill &\quad c_3(i,n)=q^{-3}(q^2-q^{-2})\, \sqrt{u}\, t\ .\hfill \\
\end{matrix}
\end{equation*}
This in turns leaves all the vertex weights of the 20V model equal to $1$, except for those of the last column ($j=n$).
Due to the boundary condition on the right of this column, only four vertex configurations are possible, as displayed in Fig.~\ref{fig:newweights},
corresponding to a vertex not visited by the uppermost path (weight $W_1$), a vertex crossed vertically by the uppermost path (weight $W_2$), or 
a vertex where the uppermost path hits the last column for the first time after a diagonal step (weight $W_3$) or a horizontal step (weight $W_4$). 
The respective new weights $W_1,\dots W_4$ are easily computed from the new Kagome weights above (see Fig.~\ref{fig:newweights}),
with the result:
\begin{equation*}
W_1=\frac{(1+u)(u+{\rm i})(1-{\rm i})}{4}\ , \quad W_2=\frac{(1+u)(u-{\rm i})(1+{\rm i})}{4} \, \quad W_3=W_4=\frac{\sqrt{u}(1+u)}{2}\ .
\end{equation*}
Clearly, a configuration for which the uppermost path hits the last column at position $\ell$ corresponds to 
a last column formed (from bottom to top) of $\ell-1$ vertices with weight $W_2$ (below the hitting point), one vertex with weight $W_3$ or $W_4$ (the hitting point)
and
$(n-\ell)$ vertices with weight $W_1$ (above the hitting point).
Note the crucial property $W_3=W_4$ which ensures that configurations, when hitting the last column, are weighted independently on the way 
(horizontal or diagonal) they reach this column and receive the weight 
$W_2^{\ell-1}W_3 W_1^{n-\ell}$. 
To summarize, setting $w_n=q^{-6}\, t\ u$ instead of $q^{-6}\, t$ changes the partition function $Z^{20V}(n)$ into the quantity
\begin{equation}
\sum_{\ell=1}^n Z^{20V_{BC2}}_\ell  \left(\frac{(1+u)(u-{\rm i})(1+{\rm i})}{4}\right)^{\ell-1}\ \frac{\sqrt{u}(1+u)}{2}\ \left(\frac{(1+u)(u+{\rm i})(1-{\rm i})}{4}\right)^{n-\ell}\ .
\label{eq:pf20}
\end{equation}
This quantity may be computed alternatively in the 6V model language by unraveling our configuration of the 20V model, or more
precisely its Kagome lattice equivalent, as we did in previous section. Indeed, the Yang Baxter relations still hold with the modified value of $w_n$. Note that
since we are now considering the DWBC2 prescription, the diagonal line $k=n$ must be moved towards the upper-right of the central square grid.  
\begin{figure}
\begin{center}
\includegraphics[width=9cm]{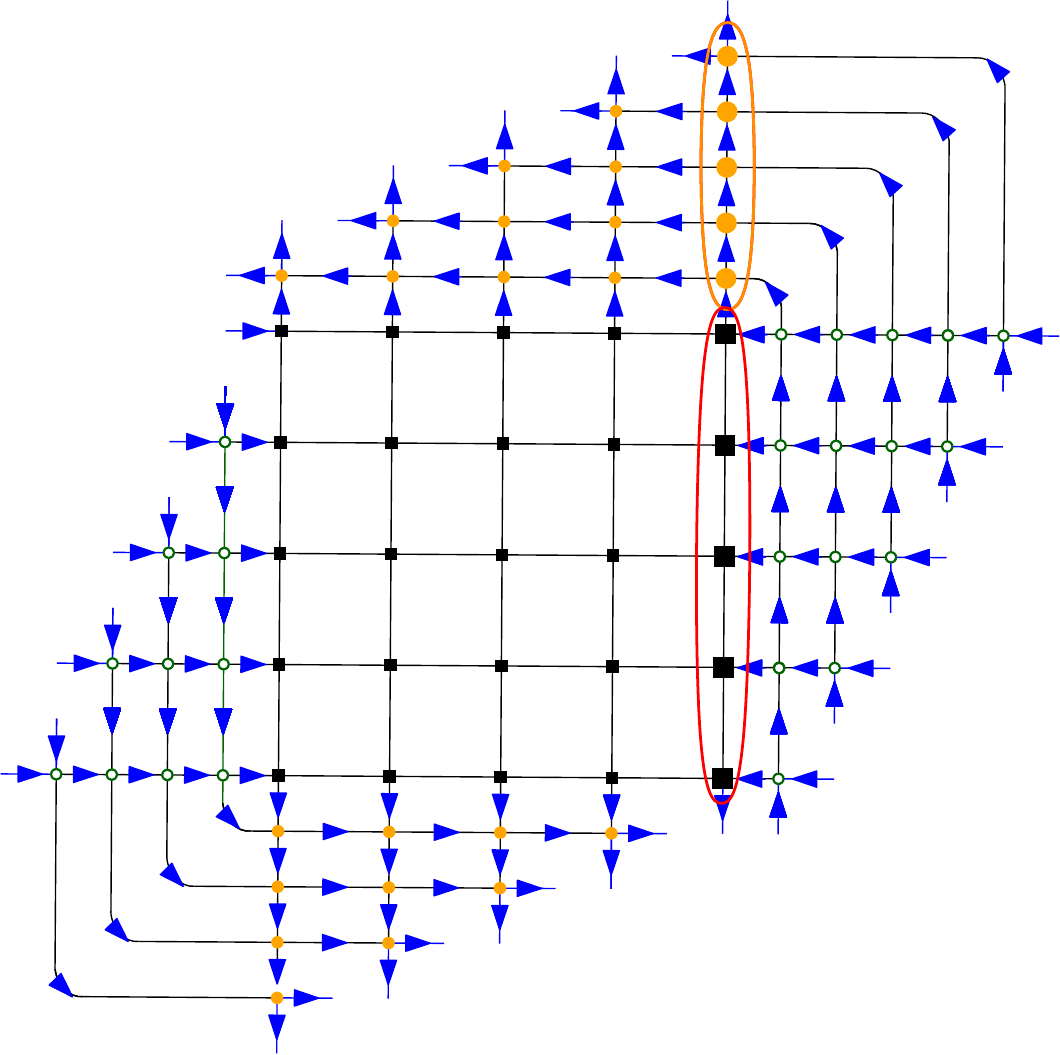}
\end{center}
\caption{\small In the unraveled configuration (here for DWBC2, hence with the main diagonal expelled towards the upper-right corner), changing $w_n\to w_n\,  u$ 
modifies only those weights corresponding to the encircled sets of nodes. The top set results into a global factor
while the bottom set corresponds to a change of the 6V weights in the last column (see text).
}
\label{fig:newunraveling}
\end{figure}
As shown in Fig.~\ref{fig:newunraveling}, changing $w_n$ from $q^{-6}\, t$ to $q^{-6}\, t\ u$ generates, compared with the fully homogeneous case, the following modifications: 
\begin{itemize}
\item{a global factor $\left(\frac{A\left(q\, t,q^{-7}\, t\, u\right)}{A\left(q\, t,q^{-7}\, t\right)}\right)^n=\left(\frac{q^4-q^{-4}\, u}{q^4-q^{-4}}\right)^n=\left(\frac{1+u}{2}\right)^n$ for 
the vertices of the sublattice $3$ crossing the vertical line $j=n$ (top encircled set of nodes in Fig.~\ref{fig:newunraveling});}
\item{a change of the weights $(1,\sqrt{2},1)$ for the equivalent DWBC 6V model \emph{in the last column of the central square grid} into weights 
$\left(\frac{A\left(q^6\, t,q^{-6}\, t\, u\right)}{A\left(q^6\, t,q^{-6}\, t\right)},\frac{B\left(q^6\, t,q^{-6}\, t\, u\right)}{B\left(q^6\, t,q^{-6}\, t\right)},
\frac{C\left(q^6\, t,q^{-6}\, t\, u\right)}{C\left(q^6\, t,q^{-6}\, t\right)}\right)=\left(1\times \frac{(u+{\rm i})(1-{\rm i})}{2},\sqrt{2}\times \frac{1+u}{2},1\times \sqrt{u}\right)$
(bottom encircled set of nodes in Fig.~\ref{fig:newunraveling})}.
\end{itemize}
Gathering the weights in the last column for a configuration where the uppermost path hits the last column at position $i=\ell$ (with the same argument as for
the 20V model), we obtain for the modified partition function \eqref{eq:pf20} the alternative expression
 \begin{equation}
\left(\frac{1+u}{2}\right)^n \sum_{\ell=1}^n Z^{6V}_{\left[1,\sqrt{2},1\right];\ell}  \left(\frac{1+u}{2}\right)^{\ell-1}\ \sqrt{u}\ \left(\frac{(u+{\rm i})(1-{\rm i})}{2}\right)^{n-\ell}\ .
\label{eq:pf6}
\end{equation}
Equating \eqref{eq:pf20} and \eqref{eq:pf6} leads directly to the announced relation \eqref{eq:20to6} identifying $\hat{Z}^{20V_{BC2}}(\tau)$ to $\hat{Z}^{6V}_{\left[1,\sqrt{2},1\right]}\left(\frac{1+\tau}{2}\right)$
upon setting $\tau={\rm i}\ \frac{u-{\rm i}}{u+{\rm i}}$.

\begin{figure}
\begin{center}
\includegraphics[width=13cm]{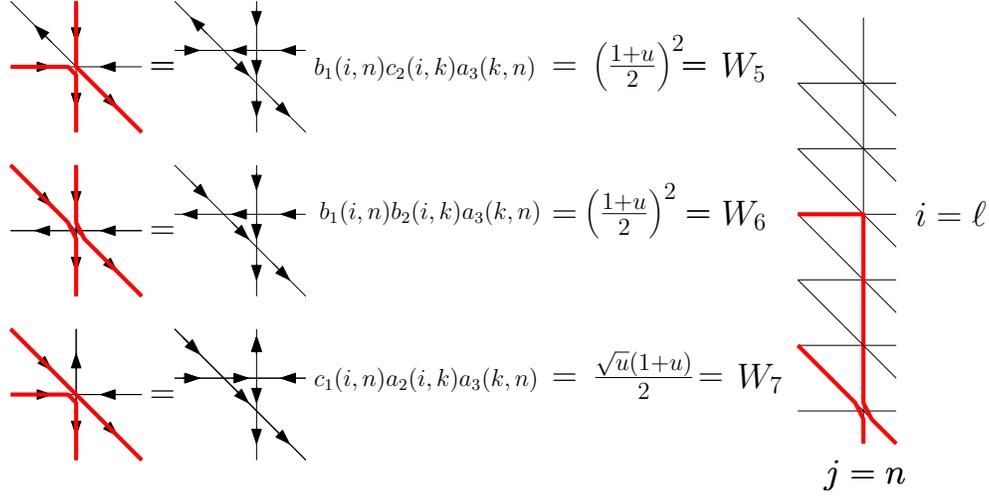}
\end{center}
\caption{\small Left: modification of the weights of the 20V model with DWBC1 when performing the change $w_n\to w_n\, u$ for the spectral parameter attached to the last
column from the special values \eqref{eq:tzw},  \eqref{eq:valq} and \eqref{eq:valt}. We display only the three new vertices specific to DWBC1 (and thus not encountered in Fig.~\ref{fig:newweights}), 
corresponding 
to the three possible environments of the lower-right node. The associated values $W_5, W_6, W_7$ satisfy $W_i\to 1$ when $u\to 1$, as required. Right:
a sample configuration of the last column.
}
\label{fig:newweightsbis}
\end{figure}
We may now easily repeat the argument in the case of the DWBC1 prescription. The modified weights are the same as those listed in Fig.~\ref{fig:newweights} 
but the lower right vertex $(i=1,j=n)$ involves new modified vertex weights $W_5$, $W_6$ and $W_7$ listed in Fig.~\ref{fig:newweightsbis}. Again we note the crucial property $W_5=W_6$ 
which ensures that configurations are weighted independently on the way the penultimate (just below the uppermost) path reaches the $(i=1,j=n)$ vertex.
For $\ell>1$, a configuration where the uppermost path hits the last column at position $i=\ell$ receives a weight $W_5 W_2^{\ell-2}W_3 W_1^{n-\ell}$ 
while for $\ell=1$, it receives the weight $W_7 W_1^{n-1}$.
The partition function $Z^{20V}$ is now transformed into the quantity
\begin{equation}
\begin{split}
&Z^{20V_{BC1}}_1  \left(\frac{\sqrt{u}(1+u)}{2}\right)\left(\frac{(1+u)(u+{\rm i})(1-{\rm i})}{4}\right)^{n-1}\\
& \qquad +\sum_{\ell=2}^n Z^{20V_{BC1}}_\ell  \left(\frac{1+u}{2}\right)^2 \left(\frac{(1+u)(u-{\rm i})(1+{\rm i})}{4}\right)^{\ell-2}\ \frac{\sqrt{u}(1+u)}{2}\ \left(\frac{(1+u)(u+{\rm i})(1-{\rm i})}{4}\right)^{n-\ell}\ .\\
\end{split}
\label{eq:pf20bis}
\end{equation}
As before, this quantity must be equal to that of \eqref{eq:pf6} for the equivalent  6V model after unraveling (note that the diagonal line with $k=n$ must now be moved to
the lower-left of the central square grid but this does not alter the global  prefactor in \eqref{eq:pf6}). Setting $\tau={\rm i}\frac{u-{\rm i}}{u+{\rm i}}$ and equating the two expressions \eqref{eq:pf20bis} and
\eqref{eq:pf6} leads directly 
to the announced relation \eqref{eq:20to6} between $\hat{Z}^{20V_{BC1}}(\tau)$ and $\hat{Z}^{6V}_{\left[1,\sqrt{2},1\right]}\left(\frac{1+\tau}{2}\right)$. This completes the proof of Theorem~\ref{20to6}
and its Corollary~\ref{cor20to6}.

\subsection{Free energy and partition function from the 6V solution}
From the above identifications, we may now rely on known results on the 6V model with DWBC to explore the statistic of the 20V model with DWBC1 or DWBC2.
A first result concerns the asymptotics of $A_n=Z^{20V}(n)$ for large $n$, directly given from that of $Z^{6V}_{\left[1,\sqrt{2},1\right]}(n)$.
For $(a,b,c)=(1,\sqrt{2},1)$, the value of the anisotropy parameter is $\Delta=\frac{a^2+b^2-c^2}{2\, a\, b}=\frac{1}{\sqrt{2}}$, meaning that the 6V model is in the so-called 
``disordered phase" region. Using the standard parametrization
\begin{equation*}
a=\rho \sin(\lambda-\phi)\ , \quad b=\rho \sin(\lambda+\phi)\ , \quad c=\rho \sin(2\lambda)\ ,
\end{equation*}
with $|\phi|<\lambda$, where we may take in our case $\lambda=3\pi/8$, $\phi=\pi/8$ and $\rho=\sqrt{2}$, the exponential growth of $Z^{6V}_{\left[1,\sqrt{2},1\right]}(n)$, hence of $Z^{20V}(n)$ at large $n$
is known to be \cite{BleFok,PZ6V}.
\begin{equation*}
A_n=Z^{20V}(n)=Z^{6V}_{\left[1,\sqrt{2},1\right]}(n)\underset{n\to \infty}{\sim} \left(\rho\, \frac{\pi \left(\cos(2\phi)-\cos(2\lambda)\right)}{4\lambda\, \cos
\left(\frac{\pi\, t}{2\lambda}\right)}\right)^{n^2}=\left(\frac{4}{3}\right)^{\frac{3}{2}n^2}\ ,
\end{equation*}
hence a free energy per site 
\begin{equation*}
f=\frac{3}{2}{\rm Log}\, \frac{4}{3}\ .
\end{equation*}
If we use for the 6V model a parametrization of the form \eqref{eq:integrable} by taking
\begin{equation}
a(i,j)=z_i-w_j\ , \qquad b(i,j)=q^{-2}\, z_i-q^2\ w_j\ , \qquad c(i,j)=(q^2-q^{-2})\sqrt{z_i\, w_j}\ ,
\label{eq:6Vweights}
\end{equation}
for the weights at the nodes $(i,j)$, the homogeneous values $(a,b,c)=(1,\sqrt{2},1)$ correspond
to choosing\footnote{Here when computing $c$, we adopt the convention that $\sqrt{(q^{2}-q^{-2})^2}=(q^{2}-q^{-2})$. Choosing the 
other branch of the square root would yield $c=-1$ but, from the relation $Z^{6V}_{\left[1,\sqrt{2},1\right]}(n)=(-1)^n Z^{6V}_{\left[1,\sqrt{2},-1\right]}(n)$
for the 6V model with DWBC, we would eventually recover for $Z^{6V}_{\left[1,\sqrt{2},1\right]}(n)$ the very same expression as that \eqref{eq:IKlimit} presented in Section~\ref{sec:proof} below.}
\begin{equation}
q={\rm e}^{{\rm i} \pi/8}\,  \qquad z_i=z=\frac{1}{1-q^4}=\frac{1+{\rm i}}{2}\, \qquad w_j=w=\frac{q^4}{1-q^4}=\frac{{\rm i}-1}{2}
\label{eq:6Vweightchoice}
\end{equation}
for all $i,j=1,\dots, n$. 

For arbitrary spectral parameters, the partition function of the 6V model with DWBC is obtained via the celebrated so-called Izergin-Korepin determinant formula \cite{Korepin,Izergin,IKdet}:
\begin{equation}
Z^{6V}=
\frac{\prod\limits_{i=1}^n c(i,i)\prod\limits_{i,j=1}^n\left(a(i,j)\, b(i,j)\right)}{\prod\limits_{1\leq i<j\leq n}(z_i-z_j)(w_j-w_i)}\det\limits_{1\leq i,j\leq n}\left( \frac{1}{a(i,j)\, b(i,j)}\right)
\label{eq:IK}
\end{equation} 
with $a(i,j)$, $b(i,j)$ and $c(i,j)$ as in 
\eqref{eq:6Vweights}. This expression is singular when the $z_i$ and $w_j$ tend to their homogeneous values \eqref{eq:6Vweightchoice} but we will explain
in Section \ref{sec:proof} how to circumvent this problem.

\section{Quarter-turn symmetric Domino tilings of a holey Aztec square}
\label{sec:cone}

Leaving the ice models aside for a while, we now turn to a different class of problems, that of domino tilings of conic domains. Our interest in these problems 
is motivated by the observation that their configurations are enumerated by the same sequence $A_n$ of \eqref{seqone}. 
A proof of this remarkable fact will given in Section \ref{sec:proof} below.

\subsection{Definition of the model: domain and domino tilings}

\begin{figure}
\begin{center}
\includegraphics[width=14cm]{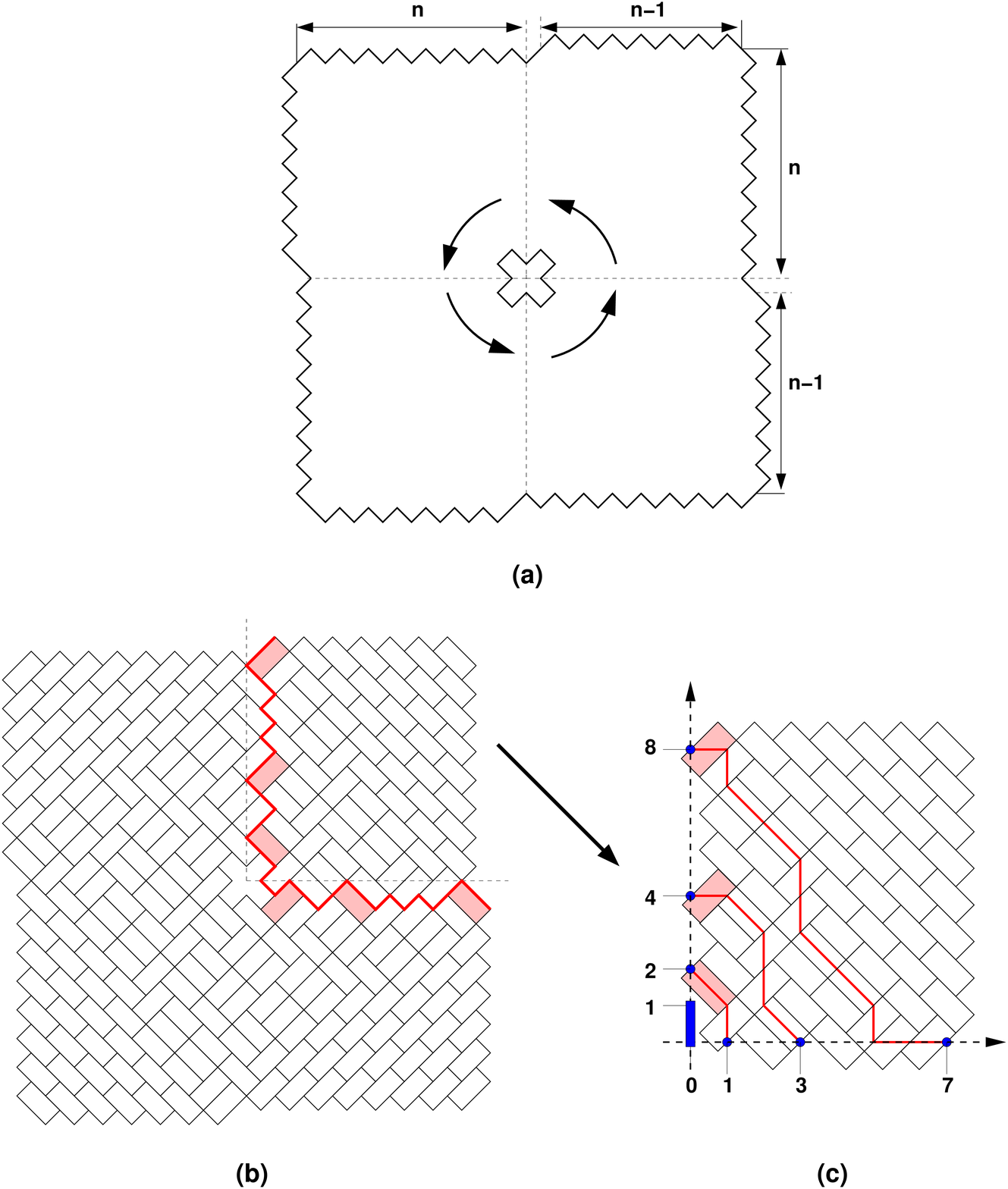}
\end{center}
\caption{\small (a) The quasi-square Aztec-like domain ${\mathcal A}_n$ with a central cross-shaped hole. (b) A sample tiling configuration of ${\mathcal A}_n$ invariant under quarter-turn rotations around the central cross. The dashed lines identify a fundamental domain w.r.t.\ rotational symmetry. We have shaded the dominos from the
top right domain that touch the vertical dashed line: these determine the zig-zag boundary of the fundamental
domain. Summing over all positions of shaded dominos and all tiling configurations of the corresponding 
fundamental domain yields the desired number of tiling configurations of ${\mathcal A}_n$ that are quarter-turn 
symmetric. (c) The tiling of the fundamental domain is in bijection with configurations of non-intersecting
Schr\"oder paths with fixed ends on the shaded dominos, with symmetric positions (here $1,3,7$ and $2,4,8$) on the
horizontal and vertical axes.}
\label{fig:DPPsq}
\end{figure}

We consider the domain ${\mathcal A}_n$ depicted in Fig.~\ref{fig:DPPsq} (a) forming a quasi-square
of Aztec-like shape of size $2n\times 2n$, with a cross-shaped hole in the middle. We wish to enumerate
the domino tilings of this domain that are {\it invariant under a quarter-turn rotation} (i.e.\ of angle $\pi/2$)
around the center of the cross. Equivalently, identifying the domain modulo quarter-turns, the
problem is equivalent to domi\-no tilings of a {\it cone} with a hole at its apex.

The present setting can be viewed as an Aztec-like generalization of that derived in \cite{Lalonde,KrattDPP} to
reformulate Andrews' DPP \cite{AndrewsDPP}. There, the DPP were shown to be in bijection with the rhombus 
tiling configurations
of a quasi-regular hexagon (of shape $(n,n+2,n,n+2,n,n+2)$) with a central
triangular hole of size 2, invariant under rotations of angle $2\pi/3$.

\subsection{Counting the tilings via Schr\"oder paths}

To perform the desired enumeration, let us delineate a fundamental domain w.r.t.\ the rotational symmetry of the 
tiling configurations as shown in Fig.~\ref{fig:DPPsq} (b). We draw axes centered at the center
of the cross, and concentrate on the first quadrant. For any tiling configuration,
we shade the dominos that touch the vertical half-axis
by a corner and belong to the first quadrant. As shown in Fig.~\ref{fig:DPPsq}, these delineate a zig-zag boundary, with a
``defect" protruding to the left for each shaded domino. We then draw a copy of this zig-zag boundary, obtained by rotation
of $-\pi/2$: these delimit the fundamental domain of the tiling. Note that the fundamental domain is entirely determined by 
the set of shaded dominos. To characterize the tiling configurations of such a fundamental domain, we use the standard mapping 
to non intersecting Schr\"oder paths (Fig.~\ref{fig:DPPsq} (c)) obtained by
first bi-coloring the underlying (tilted) square lattice so that say the center of the cross is black,
and by applying the following dictionary:
\begin{equation}\label{dominos}
\raisebox{-.6cm}{\hbox{\epsfxsize=10.cm \epsfbox{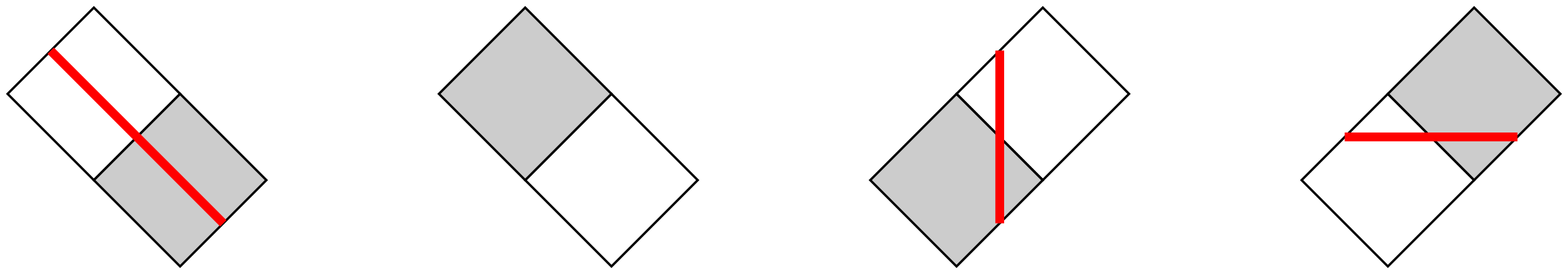}}}
\end{equation}
The Schr\"oder paths are drawn on another $\Z^2$ lattice, with coordinates shifted by $1/2$ on both axes.
 When oriented from their starting point on the (new) horizontal axis to their endpoints on the (new) vertical axis,
 these paths have left, up and diagonal steps
$(-1,0),(0,1),(-1,1)$ respectively. The endpoints of the paths  belong to the shaded dominos and
occupy positions $(0,1+i_k)$ for some integers $i_1,\ldots,\ i_\ell \in [1, n-1]$ since the first position ($i=0$)
is forbidden by the cross-shaped hole. The corresponding starting points occupy positions
$(i_k,0)$ for the {\it same set} $\{i_1,\ldots,\ i_\ell\}$. Moreover, from the
above construction, the endpoints of the paths are on the NW border of the leftmost (white) half of the shaded dominos,
and therefore the last step of all the paths can only be either left or diagonal, but \emph{cannot be up}\ : we shall call these 
{\it restricted} Schr\"oder paths. Finally, the total
number of tiling configurations of ${\mathcal A}_n$ with quarter-turn symmetry is obtained by summing
over all possible positions and orientations of the shaded dominos, i.e.\ over all possible positions $i_1, \ldots, i_\ell$ of the endpoints,
of the number of configurations of non-intersecting restricted Schr\"oder paths with these particular endpoints and their associated symmetric 
starting points.

Schr\"oder paths are readily enumerated via the generating function
$$S(r,s):= \sum_{i,j\geq 0} r^i\, s^j\, S_{i,j} = \frac{1}{1-r-s-r\, s} $$
where $S_{i,j}$ is the number of Schr\"oder paths from point $(i,0)$ to point $(0,j)$. We may think of $r$, $s$
as generators of left and up steps respectively, and $r\, s$ as the generator of a diagonal step: in any term
of the expansion of the generating function of the form $r^\ell s^u(r\, s)^d=r^i\,s^j$ for respectively $\ell,u,d$ left, up and 
diagonal steps, we indeed have $i=\ell+d$ and $j=u+d$. 
Restricted Schr\"oder paths require that the last step cannot be up.
If we now denote by ${\tilde S}_{i,j}$ the number of restricted Schr\"oder paths from $(i,0)$ to $(0,j)$,
then the desired entries ${\tilde S}_{i,j+1}$ of the Gessel-Viennot matrix are generated by:
$$\sum_{i,j\geq 0} {\tilde S}_{i,j+1} r^i\, s^j=\frac{1}{s}\big(r\, s\, S(r,s)+r \left(S(r,s)-S(r,0)\right)\big)
=\frac{2r}{(1-r)(1-r-s-r\, s)}
$$
where we have decomposed the paths according to their last step, respectively with weight $r\ s$ (if diagonal) or $r$ (if left)
and preceded respectively by an arbitrary Schr\"oder path, generated by $S(r,s)$, or by an arbitrary Schr\"oder path
with a height difference of at least one, generated by $(S(r,s)-S(r,0))$.
The global $1/s$ is simply due to the fact that we attach a weight $s^j$ in our definition instead of the natural $s^{j+1}$
associated with a height difference $j+1$. Note that , for $i=0$, ${\tilde S}_{0,j+1}=0$ as there is no restricted Schr\"oder path
with only vertical steps, while, for $j=0$, ${\tilde S}_{i,1}=2i$.  

 \medskip
The partition function of the tiling model is given by the following:
\begin{thm}
\label{thm:T4}
The number $T_4({\mathcal A}_n)$ of quarter-turn symmetric tilings of the domain ${\mathcal A}_n$ is given by 
the $n\times n$ determinant:
$$T_4({\mathcal A}_n)=\det({\mathbb I}_n+ M_n)$$
where $({\mathbb I}_n)_{i,j}=\delta_{i,j}$ and $(M_n)_{i,j}={\tilde S}_{i,j+1}$ for $i,j=0,1,...,n-1$. 
\end{thm}
\begin{proof}
Recall the Lindstr\"om Gessel-Viennot (LGV) determinant formula \cite{LGV1,GV}: the number 
of non-intersecting restricted Schr\"oder paths with fixed
starting points $(i_k,0)$ and endpoints $(0,1+i_k)$, $k=1,2,...,\ell$ is given by the sub-determinant
$ \vert M_n \vert_{i_1,...,i_\ell}^{i_1,...,i_\ell}$ of the matrix $M_n$ obtained by keeping rows and columns with labels $i_1,...,i_\ell$
(corresponding respectively to the starting and ending points).
The theorem follows from the standard Cauchy-Binet identity:
$$\det({\mathbb I}_n+ M_n)=\sum_{\ell=0}^{n}\  \sum_{0\leq i_1<i_2\cdots<i_\ell\leq n-1}\vert M_n \vert_{i_1,...,i_\ell}^{i_1,...,i_\ell} ,$$
which realizes the desired sum over all possible choices of symmetric starting and endpoints. Note that the sum includes paths with $i_1=0$,
which do not contribute as $(M_n)_{0,j}={\tilde S}_{0,j+1}=0$.
\end{proof}

Using the generating function for the infinite matrix ${\mathbb I}+M$: 
\begin{equation}
T_4(r,s)=\sum_{i,j\geq 0} ({\mathbb I}+M)_{i,j} r^i\, s^j=\frac{1}{1-r\, s}+\frac{2r}{(1-r)(1-r-s-r\, s)}
\label{eq:genT4}
\end{equation}
we easily generate the numbers $T_4({\mathcal A}_n)$ from
\begin{equation}
T_4({\mathcal A}_n)=\det\limits_{0\leq i,i\leq n-1}\left(\left.\left(\frac{1}{1-r\, s}+\frac{2r}{(1-r)(1-r-s-r\, s)}
\right)\right\vert_{r^{i}s^{j}}\right)
\label{eq:detT4}
\end{equation}
(here $f(r,s)\vert_{r^i s^j}$ denotes the coefficient of $r^i s^j$ in a double series expansion of $f(r,s)$ in $r$ and $s$).
Remarkably, \emph{these numbers match precisely the sequence $A_n$ of \eqref{seqone}}.

\subsection{Refined enumeration}

In this section, we consider refined quarter-turn symmetric tiling configurations of ${\mathcal A}_n$.
The origin of the refinement is best explained in the formulation as non-intersecting Schr\"oder paths
of Fig.~\ref{fig:DPPsq}~(c). Comparing these configurations to those attached to DPP in \cite{KrattDPP}, we are led to
consider the following two statistics for restricted Schr\"oder paths. 
For each 
such path $p$ from $(i,0)\to (0,j)$, where $0\leq i\leq n-1$ and $1\leq j\leq n$,
we define numbers $\ell_1(p)$ and $\ell_2(p)$ as:
\begin{itemize}
\item if $j<n$, $\ell_1(p)=\ell_2(p)=0$.
\item
if $j=n$, $\ell_1(p)$ is the x coordinate of the {\it last} point with y coordinate $\leq n-1$.
\item
if $j=n$, $\ell_2(p)$ is the x coordinate of the {\it first} point with y coordinate equal to $ n$.
\end{itemize}

This allows to define the refined numbers $T_{4,k}^{(m)}({\mathcal A}_n)$, $m=1,2$ of non-intersecting restricted Schr\"oder
path configurations $\mathcal P$ with\footnote{Note that at most one path contributes to this sum, namely
the topmost one if it hits the height $n$.} $\ell_m(\mathcal{P}):=\sum_{p\in {\mathcal P}} \ell_m(p)=k$, for some $k=0,1,...,n-1$. 
We refer to such models as type 1 or 2 according to the value of $m$.
In turn, these correspond to a refinement 
of the quarter-turn symmetric tiling configurations of ${\mathcal A}_n$:
\begin{itemize}
\item[Type 1:] $\ell_1(\mathcal{P})=k$ iff  the top row of the fundamental domain corresponds to the following tiling pattern
(with $k$ or $k-1$ up-right dominos):
\begin{equation*}
\raisebox{-.6cm}{\hbox{\epsfxsize=9.cm \epsfbox{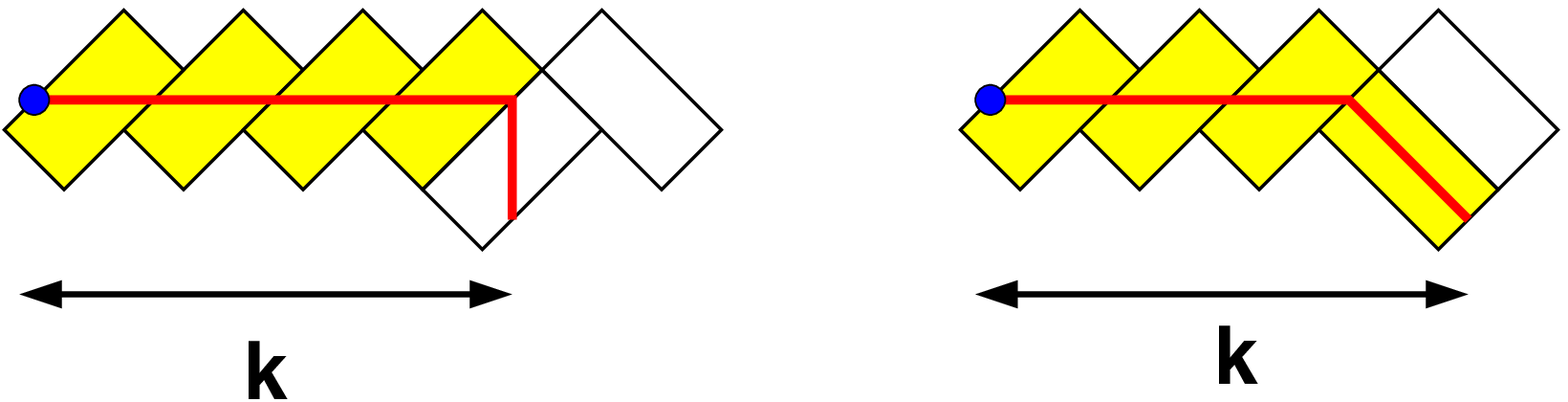}}}
\end{equation*}
\item[Type 2:] $\ell_2(\mathcal{P})=k$ iff  there are exactly $k$ up-right dominos in the top row of the fundamental domain.
\begin{equation*}
\raisebox{-.6cm}{\hbox{\epsfxsize=4.5cm \epsfbox{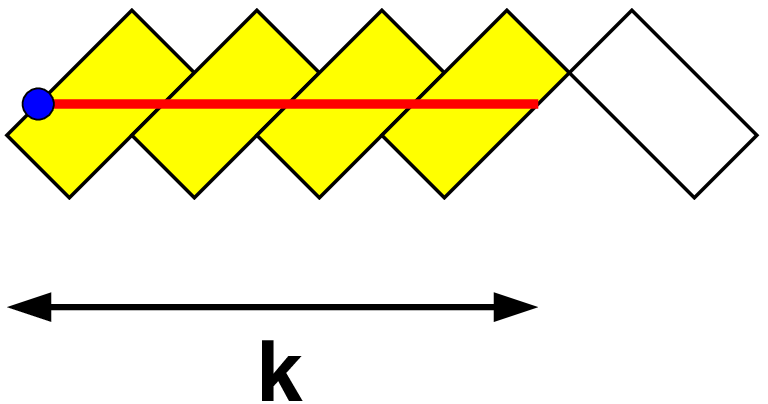}}}
\end{equation*}
\end{itemize}

The generating polynomial for the type $m=1,2$ model:
$$T_{4}^{(m)}({\mathcal A}_n;\tau):=\sum_{k=0}^{n-1} \tau^k\, T_{4,k}^{(m)}({\mathcal A}_n)$$
is interpreted as the partition functions for non-intersecting restricted Schr\"oder
path configurations with some extra weight:
\begin{itemize}
\item[Type 1:] $\tau$ per horizontal step taken at vertical position $y=n$ and  $\tau$ for a possible diagonal step from $y=n-1$ to $y=n$.
\item[Type 2:] $\tau$ per horizontal step taken at vertical position $y=n$.
\end{itemize}

\begin{thm}
\label{refinT4}
The partition functions $T_{4}^{(m)}({\mathcal A}_n;\tau)$ for type $m$ refined quarter-turn symmetric domino tilings of the domain ${\mathcal A}_n$ are given by:
\begin{equation*}
\begin{split}
T_{4}^{(1)}({\mathcal A}_n;\tau)
&=\det\limits_{0\leq i,j\leq n-1}\left(
\begin{matrix}
\left.\left(\frac{1}{1-r\, s}+\frac{2 r}{(1-r)(1-r-s-r\, s)}\right)\right\vert_{r^{i}s^{j}}& j\leq n-2\\
\left.\left(\frac{1}{1-r\, s}+\frac{2 \tau r}{(1-\tau\, r)(1-r-s-r\, s)}\right)\right\vert_{r^{i}s^{j}}& j= n-1\\
\end{matrix}
\right)
\\
&=\det\limits_{0\leq i,j\leq n-1}\left(\left.\left(
\frac{1}{1-r\, s}+\frac{2 r}{(1-r)(1-r-s-r\, s)}+s^{n-1}\,r \, \left\{\frac{2\tau }{1-\tau\, r} -\frac{2}{1-r}\right\} \frac{(1+r)^{n-1}}{(1-r)^n}\right)\right\vert_{r^i s^j}\right) .\\
\end{split}
\end{equation*}
\begin{equation*}
\begin{split}
T_{4}^{(2)}({\mathcal A}_n;\tau)
&=\det\limits_{0\leq i,j\leq n-1}\left(
\begin{matrix}
\left.\left(\frac{1}{1-r\, s}+\frac{2r}{(1-r)(1-r-s-r\, s)}\right)\right\vert_{r^{i}s^{j}}& j\leq n-2\\
\left.\left(\frac{1}{1-r\, s}+\frac{(1+\tau)r}{(1-\tau\, r)(1-r-s-r\, s)}\right)\right\vert_{r^{i}s^{j}}& j= n-1\\
\end{matrix}
\right)
\\
&=\det\limits_{0\leq i,j\leq n-1}\left(\left.\left(
\frac{1}{1-r\, s}+\frac{2 r}{(1-r)(1-r-s-r\, s)}+s^{n-1}\,r \, \left\{\frac{1+\tau}{1-\tau\, r} -\frac{2}{1-r}\right\} \frac{(1+r)^{n-1}}{(1-r)^n}\right)\right\vert_{r^i s^j}\right) .\\
\end{split}
\end{equation*}
\end{thm}

To prove the theorem, we may evaluate $T_{4}^{(m)}({\mathcal A}_n;\tau)$ by use of the LGV formula, by noticing that only the paths
ending at $(x,y)=(0,j+1)$ with $j=n-1$ receive a modified weight. More precisely,
the partition function for a restricted Schr\"oder path from $(i,0)$ to $(0,n)$ in types $m=1,2$ correspond to the coefficient of $s^{j}=s^{n-1}$ in the generating functions:
\begin{eqnarray*}
{\rm Type 1:}&&\frac{1}{s}\left\{  \left(\frac{\tau\, r}{1-\tau\, r}s +\frac{1}{1-\tau\, r}\tau\, r\, s\right) S(r,s)\right\}=\frac{2\tau\,r}{1-\tau\, r}
\frac{1}{1-r-s - r\, s}\\
{\rm Type 2:}&&\frac{1}{s}\left\{  \left(r\, s+\frac{\tau\, r}{1-\tau\, r}(s+r\, s)\right) S(r,s)\right\}=\frac{(1+\tau)r}{1-\tau\, r}
\frac{1}{1-r-s - r\, s}
\end{eqnarray*}
In type 1, we have performed a decomposition of the path according to its last visit at height $n-1$: it is either followed
by an up step (generated by $s$) 
and then by an arbitrary succession of $k$ horizontal steps, with $k\geq 1$ since the up step cannot take place at $x=0$ for a restricted Schr\"oder path
(generated by
$\tau\, r/(1-\tau\, r)$) or it is followed by a diagonal step from $y=n-1$ to $y=n$ (generated by $\tau\, r\, s$) and then by
an arbitrary succession of $k\geq 0$ horizontal steps at $y=n$ (generated by $1/(1-\tau\, r)$), all of which are preceded
by a standard Schr\"oder path, generated by $S(r,s)$. As before, the $1/s$ global prefactor comes from our choice of extracting the coefficient of $s^{j}=s^{n-1}$
instead of the natural one of $s^{j+1}$. 

In type 2, this is again obtained by decomposing the path according its last step(s): it is either a single diagonal
step (generated by $r\, s$) or an arbitrary succession of $k\geq 1$ left steps (generated by $\tau\, r/(1-\tau\, r)$)
preceded by either an up  or a diagonal step (generated by $s+r\, s$) and the preceding part of the path is a generic
Schr\"oder path, generated by $S(r,s)$.

The total partition functions $\tilde{S}_{i,n}^{(m)}(\tau)$ for a restricted Schr\"oder path from $(i,0)$ to $(0,n)$ 
in type $m=1,2$ read:
\begin{eqnarray*}
\tilde{S}_{i,n}^{(1)}(\tau)&=&\left. \frac{2 \tau\,r}{1-\tau\, r}
\frac{1}{1-r-s - r\, s}\right\vert_{r^i\,s^{n-1}}=\left.\frac{2\tau\,r}{1-\tau\, s}\frac{(1+r)^{n-1}}{(1-r)^{n}}\right\vert_{r^i}\\
\tilde{S}_{i,n}^{(2)}(\tau)&=&\left. \frac{(1+\tau)r}{1-\tau\, r}
\frac{1}{1-r-s -r\, s }\right\vert_{r^i\,s^{n-1}}=\left.\frac{(1+\tau)r}{1-\tau\, r}\frac{(1+r)^{n-1}}{(1-r)^{n}}\right\vert_{r^i}\ .
\end{eqnarray*}
Note that the partition functions for all the other paths (ending at vertical coordinate $y\leq n-1$) are the same
as before, namely equal to ${\tilde S}_{i,j}$ for paths from $(i,0)$ to $(0,j)$. 
Applying again the LGV formula, we see that the partition function $T_{4}^{(m)}({\mathcal A}_n;\tau)$ is the determinant
of a  matrix with an analogous form ${\mathbb I}_n+M_n^{(m)}(\tau)$, where the $n\times n$ matrix 
$M_n^{(m)}(\tau)$ differs from $M_n$
only in its last column, in which  the entries ${\tilde S}_{i,n}$ are replaced by the new partition
functions $\tilde{S}_{i,n}^{(m)}(\tau)$. We deduce the following:
$$T_{4}^{(m)}({\mathcal A}_n;\tau)=\det\left( {\mathbb I}_n+M_n^{(m)}(\tau) \right), \quad M_n^{(m)}(\tau)_{i,j}=\left\{ \begin{matrix}
\tilde{S}_{i,j+1}& \ {\rm for}\ j\in [0,n-2]\\
\tilde{S}_{i,n}^{(m)}(\tau)& \ {\rm for} j=n-1
\end{matrix}\right.
$$
which is nothing but Theorem \ref{refinT4} in its first form.

To put this result in the second and more compact form of Theorem~\ref{refinT4}, let us compute the generating function $T_4^{(m)}(r,s;\tau)=\sum_{i,j\geq 0}
r^i s^j ({\mathbb I}+M^{(m)}(\tau))_{i,j})$
for the new infinite matrix whose $n\times n$ truncations' determinant
yields the polynomial $T_{4}^{(m)}({\mathcal A}_n;\tau)$.
We have
\begin{lemma}\label{genqts}
\begin{eqnarray*}
T_4^{(1)}(r,s;\tau)&=&\frac{1}{1-r\, s} +\frac{2 r}{(1-r)(1-r-s-r\, s)} +s^{n-1} \, r\, \frac{2(\tau-1)}{1-\tau\, r} \frac{(1+r)^{n-1}}{(1-r)^{n+1}} \\
T_4^{(2)}(r,s;\tau)&=&\frac{1}{1-r\, s} +\frac{2 r}{(1-r)(1-r-s-r\, s)} +s^{n-1} \, r\, \frac{\tau-1}{1-\tau\, r} \frac{(1+r)^{n}}{(1-r)^{n+1}}
\end{eqnarray*}
\end{lemma}
\begin{proof}
To get the new generating function from $T_4(r,s)$ in \eqref{eq:genT4}, we must subtract the contribution of the last column,
i.e. $s^{n-1}$ times the coefficient of $s^{n-1}$ in $2r/((1-r)(1-r-s-r\, s))$ and add the new generating function
$s^{n-1}\sum_{i\geq 0}\tilde{S}_{i,n}^{(m)}(\tau)r^i$. Note that any term of order $\geq n$ in $r$ or $s$ is irrelevant and may therefore be chosen arbitrarily, as it does not affect the truncation to size $n$. The net result is the generating function:
\begin{eqnarray*}
T_4^{(1)}(r,s;\tau)&=&\frac{1}{1-r\, s}+\frac{2 r}{(1-r)(1-r-s-r\, s)}+s^{n-1}\,r \, \left\{\frac{2\tau}{1-\tau\, r} -\frac{2}{1-r}\right\} \frac{(1+r)^{n-1}}{(1-r)^n}\\
T_4^{(2)}(r,s;\tau)&=&\frac{1}{1-r\, s}+\frac{2 r}{(1-r)(1-r-s-r\, s)}+s^{n-1}\,r \, \left\{\frac{1+\tau}{1-\tau\, r} -\frac{2}{1-r}\right\} \frac{(1+r)^{n-1}}{(1-r)^n}
\end{eqnarray*}
and the Lemma follows.
\end{proof}

Theorem \ref{refinT4} allows for a very efficient calculation of the partition functions
$T_{4}^{(m)}({\mathcal A}_n;\tau)$. The first few terms read as follows.

\noindent For type 1, the polynomials $T_{4}^{(1)}({\mathcal A}_n;\tau)$ for $n=1,\ldots,7$ read:
\begin{eqnarray*}
&&1\\
&&1+2\tau\\
&&3+14\tau+6 \tau^2\\
&&23 + 198 \tau + 166 \tau^2 + 46 \tau^3\\
&&433 + 6322 \tau + 7874 \tau^2 + 4210 \tau^3 + 866 \tau^4\\
&&19705 + 468866 \tau + 777258 \tau^2 + 606026 \tau^3 + 240578 \tau^4 + 39410 \tau^5 \\
&&2151843 + 81652574 \tau + 169682406 \tau^2 + 172604734 \tau^3 + 99699558 \tau^4 + 
 31601534 \tau^5 + 4303686 \tau^6 
\end{eqnarray*}
For type 2, the polynomials $T_{4}^{(2)}({\mathcal A}_n;\tau)$ for $n=1,\ldots,7$ read:
\begin{eqnarray*}
&&1\\
&&2+\tau\\
&&10+10\tau+3 \tau^2\\
&&122 + 182 \tau + 106 \tau^2 + 23 \tau^3\\
&&3594 + 7098 \tau + 6042 \tau^2 + 2538 \tau^3 + 433 \tau^4\\
&&254138 + 623062 \tau + 691642 \tau^2 + 423302 \tau^3 + 139994 \tau^4 + 19705 \tau^5 \\
&&42978130 + 125667490 \tau + 171143570 \tau^2 + 136152146 \tau^3 + 
 65650546 \tau^4 + 17952610 \tau^5 + 2151843 \tau^6
\end{eqnarray*}

We note the identities for $n\geq 1$
\begin{equation*}
T_{4}^{(1)}({\mathcal A}_n;0)=T_{4}({\mathcal A}_{n-1})\ , 
\quad  T_{4}^{(1)}({\mathcal A}_n;\tau)\vert_{\tau^{n-1}}=2\,  T_{4}({\mathcal A}_{n-1})\ ,
\quad T_{4}^{(2)}({\mathcal A}_n;\tau)\vert_{\tau^{n-1}}=T_{4}({\mathcal A}_{n-1})\ .
\end{equation*}
All these identities have an easy explanation in terms of paths, which we leave as an exercise for the reader.

\section{Proof of the equivalence between 20V-DWBC1,2 and holey square tilings}
\label{sec:proof}

\subsection{From the Izergin-Korepin to the Gessel-Viennot 
determinant}
The aim of this Section is to prove the identity:
\begin{thm}
\label{thm:Z20T4}
The number of configurations for the 20V model with DWBC1 or DWBC2 on an $n\times n$ grid is equal to that of the quarter-turn symmetric domino tilings of the domain $\mathcal{A}_n$,
namely:
\begin{equation}
Z^{20V}(n)=T_4({\mathcal A}_n)\ .
\label{eq:Z20T4}
\end{equation}
\end{thm}

The proof goes as follows. From Theorem~\ref{An6V}, we may get $Z^{20V}(n)$ from $Z^{6V}_{\left[1,\sqrt{2},1\right]}(n)$ whose expression may itself be obtained from the
general Izergin-Korepin determinant expression \eqref{eq:IK}. Here however, we need to take as spectral parameters the specific values given by \eqref{eq:6Vweightchoice} and,
for such homogeneous values, the expression \eqref{eq:IK} cannot be used as such as both the determinant in the expression and the denominator of its prefactor vanish
identically, resulting in an indeterminate limit. Some manipulations on the Izergin-Korepin determinant are therefore required before letting $z_i$ and $w_j$ tend to their homogeneous limiting values
$z$ and $w$. 
Remarkably, the result of these manipulations is a new expression for $Z^{6V}_{\left[1,\sqrt{2},1\right]}(n)$
which resembles the Gessel-Viennot determinant encountered in Theorem \ref{thm:T4} for the expression of $T_4({\mathcal A}_n)$. The identity \eqref{eq:Z20T4} is
then proved by simple rearrangements of the determinant.

In the limit $z_i\to z$ and $w_j\to w$ for the weights \eqref{eq:6Vweights}, the expression \eqref{eq:IK} may be rewritten as:
\begin{equation}
\begin{split}
Z^{6V}=& (-1)^{\frac{n(n-1)}{2}} \left((q^2-q^{-2})\sqrt{z\, w}\right)^n \left((z-w)(q^{-2}z-q^2w)\right)^{n^2} \frac{q^{2n}}{(1-q^4)^n w^n}\\
 & \times
\det\limits_{1\leq i,j\leq n}\left(\left.\left(\frac{1}{(z+r)-(w+s)}-\frac{1}{(z+r)-q^4(w+s)}\right)\right\vert_{r^{i-1}s^{j-1}}\right)\ .\\
\end{split}
\label{eq:IKlimit}
\end{equation}
The passage from \eqref{eq:IK} to \eqref{eq:IKlimit} is explained in \cite{BDFPZ1} and we reproduce the various steps of the computation in Appendix \ref{appendixA}.

In the particular case where $q$, $z$ and $w$ take the values \eqref{eq:6Vweightchoice}, this leads immediately to
\begin{equation*}
\begin{split}
Z^{6V}_{\left[1,\sqrt{2},1\right]}(n)& =(-1)^{\frac{n(n-1)}{2}}\frac{\sqrt{2}^{n^2}}{\left(\frac{1+{\rm i}}{\sqrt{2}}\right)^n}
\det\limits_{1\leq i,j\leq n}\left(\left.\left(\frac{1}{1+r-s}-\frac{1}{1+{\rm i}+r-{\rm i}\, s}\right)\right\vert_{r^{i-1}s^{j-1}}\right)\\
& =\frac{\left({\rm i}\, \sqrt{2}\right)^{n(n-1)}}{\left(\frac{1+{\rm i}}{2}\right)^n}
\det\limits_{1\leq i,j\leq n}\left(\left.\left(\frac{1}{1+r-s}-\frac{1}{1+{\rm i}+r-{\rm i}\, s}\right)\right\vert_{r^{i-1}s^{j-1}}\right)\\
&= \det\limits_{1\leq i,j\leq n}\left(\left.\left(\frac{1-{\rm i}}{1+{\rm i}\sqrt{2}(r-s)}-\frac{1-{\rm i}}{1+{\rm i}+{\rm i}\sqrt{2}(r-{\rm i}\, s)}\right)\right\vert_{r^{i-1}s^{j-1}}\right)\\
\end{split}
\end{equation*}
expressing $Z^{6V}_{\left[1,\sqrt{2},1\right]}(n)$ as the determinant of the finite truncation of an infinite matrix with the generating function $f(r,s)$ 
explicited above.  
To go from the second to the third line, we used the identity $f(\alpha r,\beta s)\vert_{r^{i-1} s^{j-1}}=\alpha^{i-1}\beta^{j-1}f(r,s)\vert_{r^{i-1} s^{l-1}}$ 
with $\alpha=\beta={\rm i}\sqrt{2}$ and $\prod_{i,j=0}^n \alpha^{i-1}\beta^{j-1}=(\alpha\beta)^{\frac{n(n-1)}{2}}$ to insert the numerator of the prefactor inside the determinant.
As for the denominator, we also transferred it inside the determinant via the trivial identity $1/((1+{\rm i})/2)=1-{\rm i}$.
Performing the transformation 
\begin{equation}
r\to \frac{1+{\rm i}}{\sqrt{2}}\ \frac{r}{1-{\rm i}\, r}\ , \qquad s\to \frac{1-{\rm i}}{\sqrt{2}}\ \frac{s}{1+{\rm i}\, s}
\label{eq:transf}
\end{equation}
in the above infinite matrix generating function leaves the determinant unchanged.
Indeed, this amounts to first multiply $r$ by $\alpha=\frac{1+{\rm i}}{\sqrt{2}}$ and $s$ by $\beta=\frac{1-{\rm i}}{\sqrt{2}}$ changing the determinant by an overall 
multiplicative factor 
$\left(\frac{1+{\rm i}}{\sqrt{2}}\times \frac{1-{\rm i}}{\sqrt{2}}\right)^{\frac{n(n-1)}{2}}=1$ and to then to change $r\to r/(1-{\rm i}\, r)$ and $s\to s/(1+{\rm i}\, s)$.
Using 
\begin{equation*}
f\left.\left(\frac{r}{1-\gamma r},s\right)\right\vert_{r^k s^l}=f(r,s)\vert_{r^k s^l}+\sum_{m<k} \gamma^{k-m} {k-1\choose k-m} f(r,s)\vert_{r^m s^l}
\end{equation*}
the change $r\to r/(1-{\rm i}\, r)$  amounts to add to each row of the determinant a linear combination of the previous rows
while the change $s\to s/(1+{\rm i}\, s)$ amounts to add to each column a linear combination of the previous ones.
These operations leave the determinant unchanged. 
Applying the above substitution \eqref{eq:transf}, we get the alternative expression 
\begin{equation*}
\begin{split}
Z^{6V}_{\left[1,\sqrt{2},1\right]}(n)
&= \det\limits_{1\leq i,j\leq n}\left(\left.(1-{\rm i}\, r)(1+{\rm i}\, s)\left(\frac{1-{\rm i}}{1-r-s-r\, s}+\frac{{\rm i}}{1-r\, s}\right)\right\vert_{r^{i-1}s^{j-1}}\right)\\\
&= \det\limits_{1\leq i,j\leq n}\left(\left.\left(\frac{1-{\rm i}}{1-r-s-r\, s}+\frac{{\rm i}}{1-r\, s}\right)\right\vert_{r^{i-1}s^{j-1}}\right)\ .\\
\end{split}
\end{equation*}
Here again the prefactor $(1-{\rm i}\, r)(1+{\rm i}\, s)$ was removed without changing the determinant as the matrices
with and without this prefactor are obtained from one another by subtracting from each row (for the $r$-dependent factor) or respectively adding to each column 
(for the $s$-dependent factor) ${\rm i}$ times the preceding one.

Using the identity
\begin{equation}
(1+{\rm i}\ r)(1-s) \left(\frac{1-{\rm i}}{1-r-s-r\, s}+\frac{{\rm i}}{1-r\, s}\right)=
(1-r)(1-{\rm i}\, s) \left(\frac{1}{1-r\, s}+\frac{2r}{(1-r)(1-r-s-r\, s)}\right)\ ,
\label{eq:remarkableidentity}
\end{equation}
we may play once more the same trick and get the alternative expression
\begin{equation*}
\begin{split}
Z^{6V}_{\left[1,\sqrt{2},1\right]}(n)
&= \det\limits_{1\leq i,j\leq n}\left(\left.(1+{\rm i}\ r)(1-s) \left(\frac{1-{\rm i}}{1-r-s-r\, s}+\frac{{\rm i}}{1-r\, s}\right)\right\vert_{r^{i-1}s^{j-1}}\right)\\
&= \det\limits_{1\leq i,j\leq n}\left(\left.(1-r)(1-{\rm i}\, s) \left(\frac{1}{1-r\, s}+\frac{2r}{(1-r)(1-r-s-r\, s)}\right)\right\vert_{r^{i-1}s^{j-1}}\right)\\
&= \det\limits_{1\leq i,j\leq n}\left(\left.\left(\frac{1}{1-r\, s}+\frac{2r}{(1-r)(1-r-s-r\, s)}\right)\right\vert_{r^{i-1}s^{j-1}}\right)\ .\\
\end{split}
\end{equation*} 
This latter expression is nothing but that \eqref{eq:detT4} for $T_4(\mathcal{A}_n)$ up to a trivial shift by $1$ of the indices $i$ and $j$.
This proves the theorem.

\subsection{Refined equivalence}
We now wish to refine the above result and get the interpretation of $\hat{Z}^{20V_{BC1}}(\tau)$ and $\hat{Z}^{20V_{BC2}}(\tau)$ in the quarter-turn symmetric tiling language.
We have the following:
\begin{thm}
\label{thmref20VT4}
The refined partition functions for the 20V-DWBC1,2 model on an $n\times n$ grid are equal to the refined partition functions for
type 1 and 2 quarter-turn symmetric domino tilings of the domain $\mathcal{A}_n$, namely
\begin{equation*}
\hat{Z}^{20V_{BC1}}(\tau)=T_{4}^{(1)}(\mathcal{A}_n;\tau)\ , \qquad \hat{Z}^{20V_{BC2}}(\tau)=T_{4}^{(2)}(\mathcal{A}_n;\tau)\ .
\end{equation*} 
\end{thm}
The remainder of this section is devoted to the proof of this theorem.
From equation, \eqref{eq:20to6}, we may relate $\hat{Z}^{20V_{BC1}}(\tau)$ and $\hat{Z}^{20V_{BC2}}(\tau)$ to their analogue 
$\hat{Z}^{6V}_{\left[1,\sqrt{2},1\right]}(\sigma)$ \eqref{ref6V} for the 6V model,  with $\sigma=\frac{1+\tau}{2}$.
As explained in Section \ref{sec:refinedZ20}, this latter partition function may be obtained by letting $z_i$ and $w_j$ tend to their special values
$z$ and $w$ of \eqref{eq:6Vweightchoice}  except for the spectral parameter $w_n$ attached to the last column which tends instead to the value $w\, u$ for some parameter $u$.
We therefore have to evaluate the expression \eqref{eq:IK} of the Izergin-Korepin determinant in this limit. This can be done along the same lines as
in the previous section:
in the limit $z_i\to z$, $i=1,\dots, n$, $w_j\to w$, $j=1,\dots, n-1$  and $w_n\to w\, u$, the expression \eqref{eq:IK} may be rewritten as:
\begin{equation}
\begin{split}
Z^{6V}=& (-1)^{\frac{n(n-1)}{2}} \left((q^2-q^{-2})\sqrt{z\, w}\right)^{n-1} \left((z-w)(q^{-2}z-q^2w)\right)^{n(n-1)} \frac{q^{2(n-1)}}{(1-q^4)^{n-1} w^{n-1}}\\
& \times (q^2-q^{-2})\sqrt{z\, w\, u}\, \left((z-w\, u)(q^{-2}z-q^2w\, u)\right)^n\ \frac{q^2}{(1-q^4) w u}\ \frac{1}{(w\, u-w)^{n-1}}\\
 & \times
\det\limits_{1\leq i,j\leq n}\left(
\begin{matrix}
\left.\left(\frac{1}{(z+r)-(w+s)}-\frac{1}{(z+r)-q^4(w+s)}\right)\right\vert_{r^{i-1}s^{j-1}}& j\leq n-1\\
\left.\left(\frac{1}{(z+r)-w\, u}-\frac{1}{(z+r)-q^4\, w\, u}\right)\right\vert_{r^{i-1}}& j= n\\
\end{matrix}
\right)\ .\\
\end{split}
\label{eq:IKlimitbis}
\end{equation}
A derivation of this expression in given in Appendix \ref{appendixB}. For the specific values \eqref{eq:6Vweightchoice}, this yields
\begin{equation*}
\begin{split}
Z^{6V}_{\left[1,\sqrt{2},1\right];
\left[\frac{(u+{\rm i})(1-{\rm i})}{2},\sqrt{2}\frac{1+u}{2},\sqrt{u}\right]}(n)=& (-1)^{\frac{n(n-1)}{2}} \frac{\left(\sqrt{2}\right)^{n^2} }{\left(\frac{1+{\rm i}}{\sqrt{2}}\right)^n} \frac{1}{\sqrt{u}}\, \left(\frac{1+u}{2}\ 
\frac{(u+{\rm i})(1-{\rm i})}{2}\right)^n \left(\frac{1+{\rm i}}{1-u}\right)^{n-1}\\
 & \times
\det\limits_{1\leq i,j\leq n}\left(
\begin{matrix}
\left.\left(\frac{1}{1+r-s}-\frac{1}{1+{\rm i} +r-{\rm i} s}\right)\right\vert_{r^{i-1}s^{j-1}}& j\leq n-1\\
\left.\left(\frac{1}{\frac{1-{\rm i}\, u}{1-{\rm i}}+r}-\frac{1}{\frac{1+u}{1-{\rm i}}+r}\right)\right\vert_{r^{i-1}}& j= n\\
\end{matrix}
\right)\ .\\
=&\frac{1}{\sqrt{u}}\, ({\rm i}\sqrt{2})^{n-1}\left(\frac{1+u}{2}\ \frac{(u+{\rm i})(1-{\rm i})}{2}\right)^n \left(\frac{1+{\rm i}}{1-u}\right)^{n-1}\\
 & \times
\det\limits_{1\leq i,j\leq n}\left(
\begin{matrix}
\left.\left(\frac{1-{\rm i}}{1+{\rm i}\sqrt{2}(r-s)}-\frac{1-{\rm i}}{1+{\rm i} +{\rm i}\sqrt{2}(r-{\rm i} s)}\right)\right\vert_{r^{i-1}s^{j-1}}& j\leq n-1\\
\left.\left(\frac{1-{\rm i}}{\frac{1-{\rm i}\, u}{1-{\rm i}}+{\rm i}\sqrt{2}r}-\frac{1-{\rm i}}{\frac{1+u}{1-{\rm i}}+{\rm i}\sqrt{2} r}\right)\right\vert_{r^{i-1}}& j= n\\
\end{matrix}
\right)\ .\\
\end{split}
\end{equation*}
Here the notation $Z^{6V}_{\left[1,\sqrt{2},1\right];
\left[\frac{(u+{\rm i})(1-{\rm i})}{2},\sqrt{2}\frac{1+u}{2},\sqrt{u}\right]}(n)$ indicates that the weights in the last column are different from those in the other columns,
with the indicated $u$-dependent values. Again we perform the substitution \eqref{eq:transf}. Note that, as opposed to what we had before, the change in $s$ does not affect the last column $j=n$.
The effect of the substitution on the determinant is compensated by multiplying simultaneously by an overall factor $\left(\frac{1-{\rm i}}{\sqrt{2}}\right)^{n-1}$. This leads to
\begin{equation*}
\begin{split}
& Z^{6V}_{\left[1,\sqrt{2},1\right];
\left[\frac{(u+{\rm i})(1-{\rm i})}{2},\sqrt{2}\frac{1+u}{2},\sqrt{u}\right]}(n)=\frac{1}{\sqrt{u}}\, ({\rm i}\sqrt{2})^{n-1}\left(\frac{1+u}{2}\ \frac{(u+{\rm i})(1-{\rm i})}{2}\right)^n \left(\frac{1+{\rm i}}{1-u}\right)^{n-1}\left(\frac{1-{\rm i}}{\sqrt{2}}\right)^{n-1}\\
 & \qquad \qquad\qquad \qquad \qquad \qquad  \times
\det\limits_{1\leq i,j\leq n}\left(
\begin{matrix}
\left.(1-{\rm i}\, r)(1+{\rm i}\, s)\left(\frac{1-{\rm i}}{1-r-s-r\, s}+\frac{{\rm i}}{1-r\, s}\right)\right\vert_{r^{i-1}s^{j-1}}& j\leq n-1\\
\left.(1-{\rm i}\, r) \left(\frac{-2{\rm i}}{1-{\rm i}\, u+({\rm i}-u)r}-\frac{-2{\rm i}}{1+u+{\rm i}(1-u)r}\right)\right\vert_{r^{i-1}}& j= n\\
\end{matrix}
\right)\\
& =\frac{1}{\sqrt{u}}\left(\frac{(1+u)(u+{\rm i})(1+{\rm i})}{2(1-u)}\right)^{n}
\det\limits_{1\leq i,j\leq n}\left(
\begin{matrix}
\left.(1-{\rm i}\, r)(1+{\rm i}\, s)\left(\frac{1-{\rm i}}{1-r-s-r\, s}+\frac{{\rm i}}{1-r\, s}\right)\right\vert_{r^{i-1}s^{j-1}}& j\leq n-1\\
\left.(1-{\rm i}\, r) \left(\frac{u-1}{1-{\rm i}\, u+({\rm i}-u)r}-\frac{u-1}{1+u+{\rm i}(1-u)r}\right)\right\vert_{r^{i-1}}& j= n\\
\end{matrix}
\right)\\
&=\frac{1}{\sqrt{u}}\left(\frac{(1+u)(u+{\rm i})(1+{\rm i})}{2(1-u)}\right)^{n}
\det\limits_{1\leq i,j\leq n}\left(
\begin{matrix}
\left.\left(\frac{1-{\rm i}}{1-r-s-r\, s}+\frac{{\rm i}}{1-r\, s}\right)\right\vert_{r^{i-1}s^{j-1}}& j\leq n-1\\
\left.\left(\frac{u-1}{1-{\rm i}\, u+({\rm i}-u)r}-\frac{u-1}{1+u+{\rm i}(1-u)r}\right)\right\vert_{r^{i-1}}& j= n\\
\end{matrix}
\right)
\end{split}
\end{equation*}
where we again removed the factors $(1-{\rm i}\, r)$ and $(1+{\rm i}\, s)$ without changing the determinant.

We now recall from \eqref{eq:pf6} the expression 
 \begin{equation*}
 \begin{split}
Z^{6V}_{\left[1,\sqrt{2},1\right];
\left[\frac{(u+{\rm i})(1-{\rm i})}{2},\sqrt{2}\frac{1+u}{2},\sqrt{u}\right]}(n)& =\sum_{\ell=1}^n Z^{6V}_{\left[1,\sqrt{2},1\right];\ell}  \left(\frac{1+u}{2}\right)^{\ell-1}\ \sqrt{u}\ \left(\frac{(u+{\rm i})(1-{\rm i})}{2}\right)^{n-\ell}\\
&=  \sqrt{u}\ \left(\frac{(u+{\rm i})(1-{\rm i})}{2}\right)^{n-1} \hat{Z}^{6V}_{\left[1,\sqrt{2},1\right]}(\sigma)
\end{split}
\end{equation*}
where $\sigma=\frac{1+u}{({\rm i}+u)(1-{\rm i})}$, or equivalently, $u=\frac{1-(1+{\rm i})\sigma}{(1-{\rm i})\sigma-1}$. Comparing the two expressions above leads to
\begin{equation*}
\begin{split}
\hat{Z}^{6V}_{\left[1,\sqrt{2},1\right]}(\sigma)=&\frac{1}{u}\left(\frac{(1+u)(u+{\rm i})(1+{\rm i})}{2(1-u)}\right)\left({\rm i}\ \frac{1+u}{1-u}\right)^{n-1}\\
 & \times
\det\limits_{1\leq i,j\leq n}\left(
\begin{matrix}
\left.\left(\frac{1-{\rm i}}{1-r-s-r\, s}+\frac{{\rm i}}{1-r\, s}\right)\right\vert_{r^{i-1}s^{j-1}}& j\leq n-1\\
\left.\left(\frac{u-1}{1-{\rm i}\, u+({\rm i}-u)r}-\frac{u-1}{1+u+{\rm i}(1-u)r}\right)\right\vert_{r^{i-1}}& j= n\\
\end{matrix}
\right)\\
=&\frac{\sigma-1}{\sigma(1-{\rm i})+{\rm i}}\left(\frac{\sigma}{\sigma-1}\right)^n
\det\limits_{1\leq i,j\leq n}\left(
\begin{matrix}
\left.\left(\frac{1-{\rm i}}{1-r-s-r\, s}+\frac{{\rm i}}{1-r\, s}\right)\right\vert_{r^{i-1}s^{j-1}}& j\leq n-1\\
\left.\left(\frac{1-{\rm i}}{1+(1-2\sigma)r}+\frac{{\rm i}}{\sigma+(1-\sigma)r}\right)\right\vert_{r^{i-1}}& j= n\\
\end{matrix}
\right)\ .
\end{split}
\end{equation*}
Setting $\sigma=\frac{1+\tau}{2}$ and using \eqref{eq:20to6}, we deduce alternatively
\begin{equation*}
\begin{split}
\hat{Z}^{20V_{BC2}}(\tau)=&\frac{(1+{\rm i})(\tau-1)}{2(\tau+{\rm i})}\left(\frac{\tau+1}{\tau-1}\right)^n
\det\limits_{1\leq i,j\leq n}\left(
\begin{matrix}
\left.\left(\frac{1-{\rm i}}{1-r-s-r\, s}+\frac{{\rm i}}{1-r\, s}\right)\right\vert_{r^{i-1}s^{j-1}}& j\leq n-1\\
\left.\left(\frac{1-{\rm i}}{1-\tau\, r}+\frac{2{\rm i}}{(\tau+1)-(\tau-1)r}\right)\right\vert_{r^{i-1}}& j= n\\
\end{matrix}
\right)\\
&=\frac{1+{\rm i}}{\tau+{\rm i}}\left(\frac{\tau+1}{\tau-1}\right)^{n-1}\det(Q_n+P_n)\\
\end{split}
\end{equation*}
where
\begin{equation*}
(P_n)_{i,j}=
\left\{ 
\begin{matrix}
\left.\left(\frac{{\rm i}}{1-r\, s}\right)\right\vert_{r^{i-1}s^{j-1}}={\rm i}\, \delta_{i,j}& j\leq n-1\\
\left.\left(\frac{{\rm i}}{1-\frac{\tau-1}{\tau+1}r}\right)\right\vert_{r^{i-1}}={\rm i}\, \left(\frac{\tau-1}{\tau+1}\right)^{i-1}& j=n\\
\end{matrix}
\right.
\qquad 
(Q_n)_{i,j}=
\left\{ 
\begin{matrix}
\left.\left(\frac{1-{\rm i}}{1-r-s-r\, s}\right)\right\vert_{r^{i-1}s^{j-1}}& j\leq n-1\\
\left.\left(\frac{\frac{1-{\rm i}}{2}(\tau+1)}{1-\tau\, r}\right)\right\vert_{r^{i-1}}& j=n\\
\end{matrix}
\right.\ .
\end{equation*}
Now the matrix $P_n$ differs from the matrix ${\rm i}\, {\mathbb I}_n$ only in its last column and its determinant
is therefore easily obtained as
\begin{equation*}
\det(P_n)={\rm i}^{n-1}(P_n)_{n,n}={\rm i}^n\left(\frac{\tau-1}{\tau+1}\right)^{n-1}
\end{equation*}
and we deduce
\begin{equation*}
\hat{Z}^{20V_{BC2}}(\tau)=\frac{1+{\rm i}}{\tau+{\rm i}}\,\det({\rm i}\, {\mathbb I}_n+{\rm i}\ Q_nP_n^{-1})\ .
\end{equation*}
Using
\begin{equation*}
{\rm i}\, (P_n^{-1})_{i,j}=
\left\{ 
\begin{matrix}
\delta_{i,j}& j\leq n-1\\
-\left(\frac{\tau+1}{\tau-1}\right)^{n-i}& i\leq n-1, j=n\\
\left(\frac{\tau+1}{\tau-1}\right)^{n-1}& i=n, j=n\\
\end{matrix}
\right.\ ,
\end{equation*}
we get ${\rm i}\, (Q_nP_n^{-1})_{i,j}=(Q_n)_{i,j}$ for $j<n$, while
\begin{equation*}
\begin{split}
{\rm i}\, & (Q_nP_n^{-1})_{i,n}=
\left.\left(\frac{\frac{1-{\rm i}}{2}(\tau+1)}{1-\tau\, r}\right)\right\vert_{r^{i-1}}\left(\frac{\tau+1}{\tau-1}\right)^{n-1}
-\sum_{k=1}^{n-1}\left.\left(\frac{1-{\rm i}}{1-r-s-r\, s}\right)\right\vert_{r^{i-1}s^{k-1}}\left(\frac{\tau+1}{\tau-1}\right)^{n-k}
\\
&= \left(\frac{\tau+1}{\tau-1}\right)^{n-1}\underbrace{\left(\frac{\frac{1-{\rm i}}{2}(\tau+1)}{1-\tau\, r}-\frac{1-{\rm i}}{1-r-\frac{\tau-1}{\tau+1}-r\, \frac{\tau-1}{\tau+1}}\right)}_{=0}\Bigg\vert_{r^{i-1}}
+\sum_{k=n}^\infty \left(\frac{\tau+1}{\tau-1}\right)^{n-k} \left.\left(\frac{1-{\rm i}}{1-r-s-r\, s}\right)\right\vert_{r^{i-1}s^{k-1}}
\\
&=(1-{\rm i})\left.\left(\sum_{k=n}^\infty \left(\frac{\tau+1}{\tau-1}\right)^{n-k} \frac{(1+r)^{k-1}}{(1-r)^k}\right)\right\vert_{r^{i-1}}
=\frac{1-{\rm i}}{2}\left(\frac{1+r}{1-r}\right)^{n-1}\left.\left(\frac{1+\tau}{1-\tau\, r}\right)\right\vert_{r^{i-1}}\ .
\end{split}
\end{equation*}
We end up with the expression
\begin{equation*}
\hat{Z}^{20V_{BC2}}(\tau)
=\frac{1+{\rm i}}{\tau+{\rm i}}\det\limits_{1\leq i,j\leq n}\left(
\left.\left(\frac{{\rm i}}{1-r\, s}+\frac{1-{\rm i}}{1-r-s-r\, s}+s^{n-1}\left(\frac{1+r}{1-r}\right)^{n-1}\left\{\frac{1-{\rm i}}{2}\frac{1+\tau}{1-\tau\, r}-\frac{1-{\rm i}}{1-r}\right\}\right)\right\vert_{r^{i-1}s^{j-1}}\right)
\end{equation*}
where the last term corrects the wrong value $\left.\frac{1-{\rm i}}{1-r-s-r\, s}\right\vert_{r^{i-1}s^{n-1}}=\left(\frac{1+r}{1-r}\right)^{n-1}\left.\frac{1-{\rm i}}{1-r}\right\vert_{r^{i-1}}$
coming from the second term to the correct value above. As before, we may multiply the function inside the determinant by $\frac{(1+{\rm i}\ r)(1-s)}{(1-r)(1+{\rm i}\, s)}$ without changing the 
value of the determinant. 
Using again the identity \eqref{eq:remarkableidentity}, we obtain
\begin{equation}
\begin{split}
\hat{Z}^{20V_{BC2}}(\tau)&=\frac{1+{\rm i}}{\tau+{\rm i}}\det\limits_{1\leq i,j\leq n}\left(
\left(\frac{1}{1-r\, s}+\frac{2r}{(1-r)(1-r-s-r\, s)}\right.\right.\\
&\left.\left.\left.\qquad \qquad+s^{n-1}\left(\frac{1+r}{1-r}\right)^{n-1}\frac{(1+{\rm i}\ r)(1-\cancel{s})}{(1-r)(1+{\rm i}\,\cancel{s})}\left(\frac{1-{\rm i}}{2}\frac{1+\tau}{1-\tau\, r}-\frac{1-{\rm i}}{1-r}\right)\right)\right\vert_{r^{i-1}s^{j-1}}\right)\\
&=\frac{1+{\rm i}}{\tau+{\rm i}}\det\limits_{1\leq i,j\leq n}\left(K_{i,j}\right)\\
K_{i,j}&:=\left.
\left(\frac{1}{1-r\, s}+\frac{2r}{(1-r)(1-r-s-r\, s)}+s^{n-1}\frac{(1+r)^{n-1}}{(1-r)^n}\frac{1+{\rm i}\, r}{1+{\rm i}}\left\{\frac{1+\tau}{1-\tau\, r}-\frac{2}{1-r}\right\}\right)\right\vert_{r^{i-1}s^{j-1}}\
\end{split}
\label{eq:penult}
\end{equation}
where the crossed out $\cancel{s}$ play no role and were thus removed.
In this form, the expression is now very close to that of Theorem~\ref{refinT4} for $T_{4,2}(\mathcal{A}_n;\tau)$, namely (with a trivial shift by $1$ of the indices)
\begin{equation}
\begin{split}
&T_{4,2}(\mathcal{A}_n;\tau)
=\det\limits_{1\leq i,j\leq n}\left(L_{i,j}\right)\\
& L_{i,j}:=\left.
\left(\frac{1}{1-r\, s}+\frac{2r}{(1-r)(1-r-s-r\, s)}+s^{n-1}\frac{(1+r)^{n-1}}{(1-r)^n}r \left\{\frac{(1+\tau)}{1-\tau\, r}-\frac{2}{1-r}\right\}\right)\right\vert_{r^{i-1}s^{j-1}}\ .\
\end{split}
\label{eq:ult}
\end{equation}
The identification of the two formulas follows from the following simple remark:
\begin{lemma}
 We have the identities
\begin{equation}
\begin{matrix}
& K_{i,j}=L_{i,j} \hfill &\qquad  j<n \\
& &\\
& K_{i,n}=\displaystyle{\frac{\tau+{\rm i}}{1+{\rm i}}L_{i,n}+\frac{\tau-1}{1+{\rm i}}\ \sum\limits_{j=1}^{n-1}L_{i,j}}& \qquad j=n\\
\end{matrix}
\label{LKlemma}
\end{equation}
\end{lemma}
\begin{proof}
The first statement for $j<n$ is by definition. For $j=n$, we use
\begin{equation*}
L_{i,j}=\left\{
\begin{matrix}
\left.\left(r^{j-1}+\frac{(1+r)^{j-1}}{(1-r)^j}\ \frac{2\, r}{1-r} \right)\right\vert_{r^{i-1}}& j\leq n-1\\
\left.\left(r^{n-1}+\frac{(1+r)^{n-1}}{(1-r)^n}\ \frac{(1+\tau)\, r}{1-\tau\, r}\right) \right\vert_{r^{i-1}}& j=n\\
\end{matrix}
\right.
\end{equation*}
so that 
\begin{equation*}
\begin{split}
\frac{\tau+{\rm i}}{1+{\rm i}}L_{i,n}+\frac{\tau-1}{1+{\rm i}}\ \sum_{j=1}^{n-1}L_{i,j}&=
\left.\left(\frac{\tau+{\rm i}}{1+{\rm i}}r^{n-1}+\frac{\tau-1}{1+{\rm i}}\frac{1-r^{n-1}}{1-r}\right.\right.
\\
&\left.\left.+\frac{\tau+{\rm i}}{1+{\rm i}}\frac{(1+r)^{n-1}}{(1-r)^n}\ \frac{(1+\tau)\, r}{1-\tau\, r}
-\frac{\tau-1}{1+{\rm i}} \frac{1-\left(\frac{1+r}{1-r}\right)^{n-1}}{1-r}\right) \right\vert_{r^{i-1}}\\
&=\left.\left(\left(\frac{\tau+{\rm i}}{1+{\rm i}}-\frac{\tau-1}{1+{\rm i}}\frac{1}{1-\cancel{r}}\right)r^{n-1}+
\left(\frac{\tau+{\rm i}}{1+{\rm i}}\frac{(1+\tau)\, r}{1-\tau\, r}+\frac{\tau-1}{1+{\rm i}}\right)\frac{(1+r)^{n-1}}{(1-r)^n} \right)\right\vert_{r^{i-1}}
\\
&=\left.\left(r^{n-1}+
\left(\frac{\tau+{\rm i}}{1+{\rm i}}\frac{(1+\tau)\, r}{1-\tau\, r}+\frac{\tau-1}{1+{\rm i}}\right)\frac{(1+r)^{n-1}}{(1-r)^n} \right)\right\vert_{r^{i-1}}
\\
&=\left.\left(r^{n-1}+
\left(\frac{1+{\rm i}\, r}{1+{\rm i}}\left(\frac{1+\tau}{1+\tau\, r}-\frac{2}{1-r}\right)+\frac{2r}{1-r}\right)\frac{(1+r)^{n-1}}{(1-r)^n} \right)\right\vert_{r^{i-1}}
= K_{i,n}\ .
\end{split}
\end{equation*}
Here again, we removed the crossed out $\cancel{r}$ as it plays no role for $i\leq n$ and we used the easily checked identity
\begin{equation*}
\frac{\tau+{\rm i}}{1+{\rm i}}\frac{(1+\tau)\, r}{1-\tau\, r}+\frac{\tau-1}{1+{\rm i}}=\frac{1+{\rm i}\, r}{1+{\rm i}}\left(\frac{1+\tau}{1+\tau\, r}-\frac{2}{1-r}\right)+\frac{2r}{1-r}\ .
\end{equation*}
The lemma follows.
\end{proof}
With the identities \eqref{LKlemma}, the expression \eqref{eq:penult} is transformed into \eqref{eq:ult} by a simple expansion of the determinant with respect to the last column.
This completes the proof that the expression \eqref{eq:penult} for $\hat{Z}^{20V_{BC2}}(\tau)$ matches that \eqref{eq:ult} for $T_{4}^{(2)}(\mathcal{A}_n;\tau)$. This amounts
precisely to the second statement of Theorem~\ref{thmref20VT4}.

\medskip
To compute  $\hat{Z}^{20V_{BC1}}(\tau)$, we first note that the expression \eqref{eq:ult} for $\hat{Z}^{20V_{BC2}}(\tau)$ can be substituted in \eqref{eq:20to6} to get 
\begin{equation*}
\hat{Z}^{6V}_{\left[1,\sqrt{2},1\right]}(\sigma)
=\det\limits_{1\leq i,j\leq n}\left(
\begin{matrix}
\left.\left(\frac{1}{1-r\, s}+\frac{2r}{(1-r)(1-r-s-r\, s)}\right)\right\vert_{r^{i-1}s^{j-1}}& j\leq n-1\\
\left.\left(\frac{1}{1-r\, s}+\frac{2\sigma\, r}{(1-(2\sigma-1)\, r)(1-r-s-r\, s)}\right)\right\vert_{r^{i-1}s^{j-1}}& j= n\\
\end{matrix}
\right)
\end{equation*}
for $\sigma=\frac{1+\tau}{2}$.
The partition function $\hat{Z}^{20V_{BC1}}(\tau)$ is then obtained via \eqref{eq:6to20BC1}, which yields 
\begin{equation*}
\begin{split}
\hat{Z}^{20V_{BC1}}(\tau)
&=\frac{2\tau}{1+\tau}\ \det\limits_{1\leq i,j\leq n}\left(
\begin{matrix}
\left.\left(\frac{1}{1-r\, s}+\frac{2r}{(1-r)(1-r-s-r\, s)}\right)\right\vert_{r^{i-1}s^{j-1}}& j\leq n-1\\
\left.\left(\frac{1}{1-r\, s}+\frac{(1+\tau)\, r}{(1-\tau\, r)(1-r-s-r\, s)}\right)\right\vert_{r^{i-1}s^{j-1}}& j= n\\
\end{matrix}
\right)\\
&\ \ \ +\frac{1-\tau}{1+\tau}\ 
\det\limits_{1\leq i,j\leq n}\left(
\begin{matrix}
\left.\left(\frac{1}{1-r\, s}+\frac{2r}{(1-r)(1-r-s-r\, s)}\right)\right\vert_{r^{i-1}s^{j-1}}& j\leq n-1\\
\left.\left(\frac{1}{1-r\, s}\right)\right\vert_{r^{i-1}s^{j-1}}& j= n\\
\end{matrix}
\right)\\
&=\det\limits_{1\leq i,j\leq n}\left(
\begin{matrix}
\left.\left(\frac{1}{1-r\, s}+\frac{2r}{(1-r)(1-r-s-r\, s)}\right)\right\vert_{r^{i-1}s^{j-1}}& j\leq n-1\\
\left.\left(\frac{2\tau}{1+\tau}\left(\frac{1}{1-r\, s}+\frac{(1+\tau)\, r}{(1-\tau\, r)(1-r-s-r\, s)}\right)+\frac{1-\tau}{1+\tau}\frac{1}{1-r\, s}\right)\right\vert_{r^{i-1}s^{j-1}}& j= n\\
\end{matrix}
\right)\\
&=\det\limits_{1\leq i,j\leq n}\left(
\begin{matrix}
\left.\left(\frac{1}{1-r\, s}+\frac{2r}{(1-r)(1-r-s-r\, s)}\right)\right\vert_{r^{i-1}s^{j-1}}& j\leq n-1\\
\left.\left(\frac{1}{1-r\, s}+\frac{2\tau\, r}{(1-\tau\, r)(1-r-s-r\, s)}\right)\right\vert_{r^{i-1}s^{j-1}}& j= n\\
\end{matrix}
\right)\ .\\
\end{split}
\end{equation*}
We recognize the expression of Theorem~\ref{refinT4} for $T_{4}^{(1)}(\mathcal{A}_n;\tau)$ (up to a trivial shift in the indices) so that the first statement of Theorem~\ref{thmref20VT4}
follows.

\section{Other boundary conditions}
\label{sec:other}

In this section, we explore other possible DWBC-like boundary conditions. One of them,
which we call DWBC3,  leads to a striking conjecture.

\subsection{The 20V model with DWBC3}\label{sec:dwbc3}

\begin{figure}
\begin{center}
\includegraphics[width=12cm]{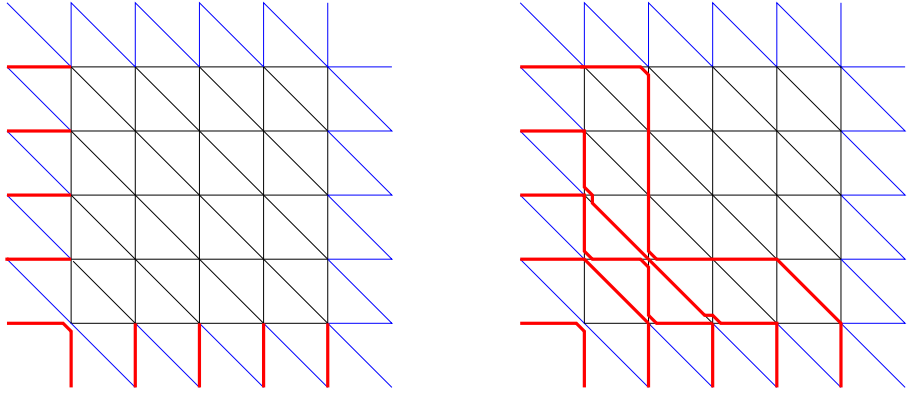}
\end{center}
\caption{\small Left: DWBC3 boundary conditions for an $n \times n$ grid for $n=5$. 
All outer horizontal edges along the West boundary,
as well as all vertical edges along the South boundary are occupied by paths. All other outer edges are empty. 
Right: This gives rise to
configurations of $n=5$ non-intersecting osculating Schr\"oder paths, such as that depicted.}
\label{fig:grid3}
\end{figure}

We consider the following variant of the DWBC1,2 of Section \ref{sec:20VDWBC} for the 20V model.
We still consider a square grid of size $n\times n$ in the square lattice with the second diagonal edge 
on each face. The boundary conditions on the external edge orientations are now as follows: (i) all horizontal
external edges point towards the square domain (ii) all vertical external edges point away from the square domain
(iii) all diagonal external edges points towards the NW. 

In the osculating Schr\"oder path formulation, we have paths entering the grid on each horizontal external edge
on the West boundary, and exiting the grid on each vertical external edge along the South boundary (see Fig.~\ref{fig:grid3} for an illustration). 

The 20V-DWBC3 configurations are easily counted by use of transfer matrices, giving rise to the sequence
$B_n$ of \eqref{seqtwo}. Remarkably, these numbers appear in the context of yet another enumeration 
problem of domino tilings, which we describe now.

\subsection{Domino tilings of a triangle and the DWBC3 conjecture}

\subsubsection{Domino tilings of a square}
Let us consider the number of tilings of a $(2n)\times (2n)$
square domain ${\mathcal S}_n$ of the square lattice by means of rectangular dominos of size $2\times 1$ and $1\times 2$.
This is part of the archetypical dimer problems solved by Kasteleyn and Temperley and Fisher \cite{Kasteleyn,TF}.
If $T({\mathcal S}_n)$ denotes this number, we have:
$$ T({\mathcal S}_n)=\prod_{i=1}^n \prod_{j=1}^n \left\{ 4 \cos^2\left(\frac{i}{2n+1}\right)+4 \cos^2\left(\frac{j}{2n+1}\right)\right\} \ .$$
It was later observed that:
\begin{equation}\label{Dtob} T({\mathcal S}_n)=2^n\, b_n^2 \end{equation}
with
$$ b_n=\prod_{1\leq i<j\leq n} \left\{ 4 \cos^2\left(\frac{i}{2n+1}\right)+4 \cos^2\left(\frac{j}{2n+1}\right)\right\}=1,\, 3,\, 29, \, 901,\, ... $$
where we recognize the first terms of the sequence $B_n$ \eqref{seqtwo} above.
The asymptotics of the numbers $b_n$ for large $n$ read:
$$\lim_{n\to \infty} \frac{1}{n^2}{\rm Log}(b_n)= \frac{2}{\pi} G =.583121808... \ ,$$
where $G$ is the Catalan constant, $G=1-\frac{1}{3^2}+\frac{1}{5^2}-\frac{1}{7^2}+\cdots$.

\subsubsection{Domino tilings of a triangle}\label{sec:triangle}
A combinatorial proof of the integrality of $b_n$ due to Patcher \cite{Patcher}
shows in fact that the sequence $b_n$ enumerates domino 
tilings of a ``triangle" ${\mathcal T}_n$ (half of the square $\mathcal{S}_n$), with the following shape of an
inverted staircase with a first step
of size 1, and all other steps of size 2:
\begin{equation*}
\raisebox{-.6cm}{\hbox{\epsfxsize=6.cm \epsfbox{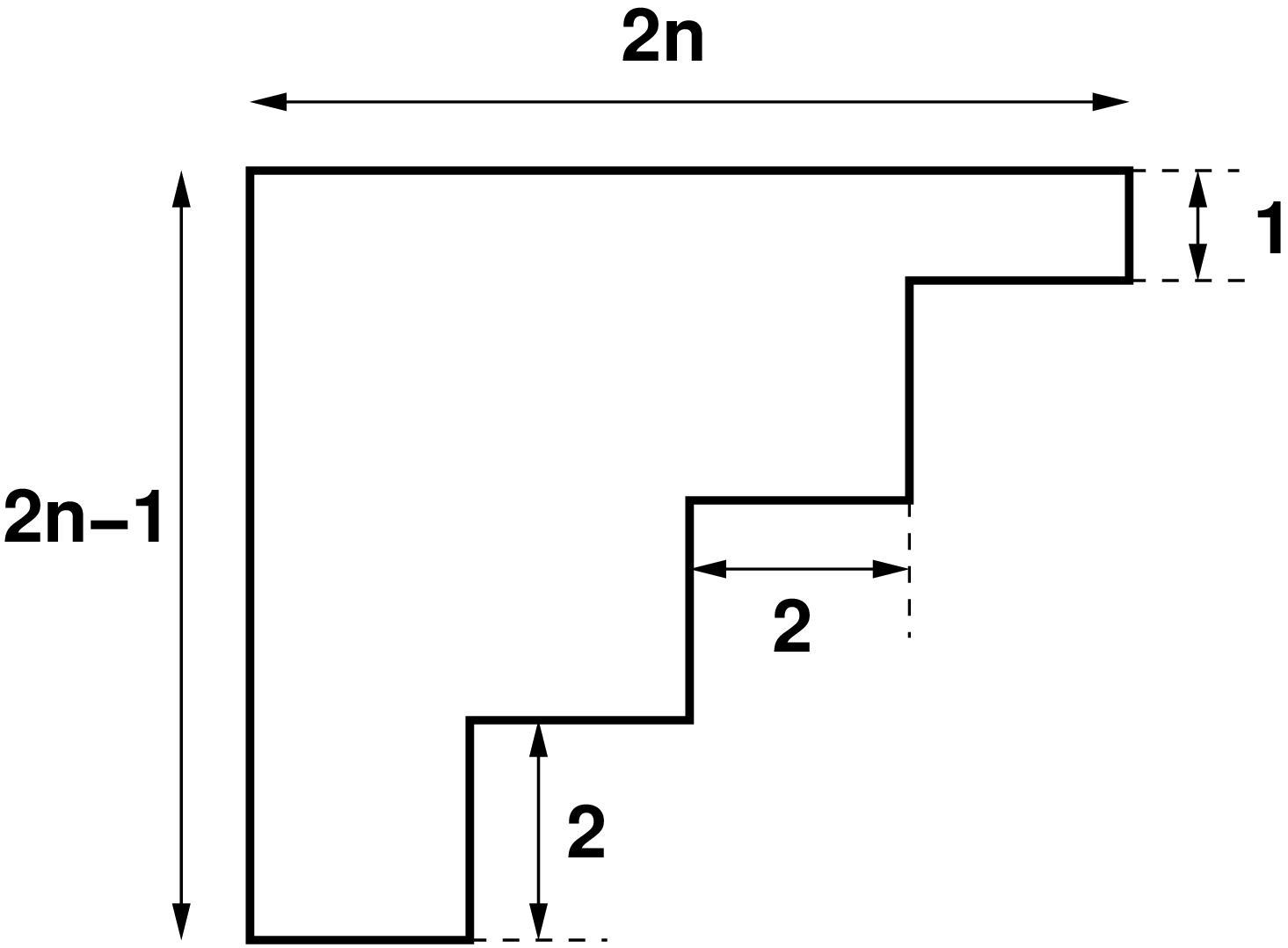}}}
\end{equation*}
In our notations, we write:
\begin{equation*}
b_n=T({\mathcal T}_n)\ .
\end{equation*}

From a practical point of view, it is interesting to recover the first terms of the sequence 
$b_n=T({\mathcal T}_n)$ from standard combinatorial techniques easily amenable to generalizations.
As before, the enumeration of domino tilings of ${\mathcal T}_n$ is easily performed by means of non-intersecting Schr\"oder paths, now with steps $(1,1)$, $(1,-1)$ and $(2,0)$
as shown below:
\begin{equation*}
\raisebox{-.6cm}{\hbox{\epsfxsize=15.cm \epsfbox{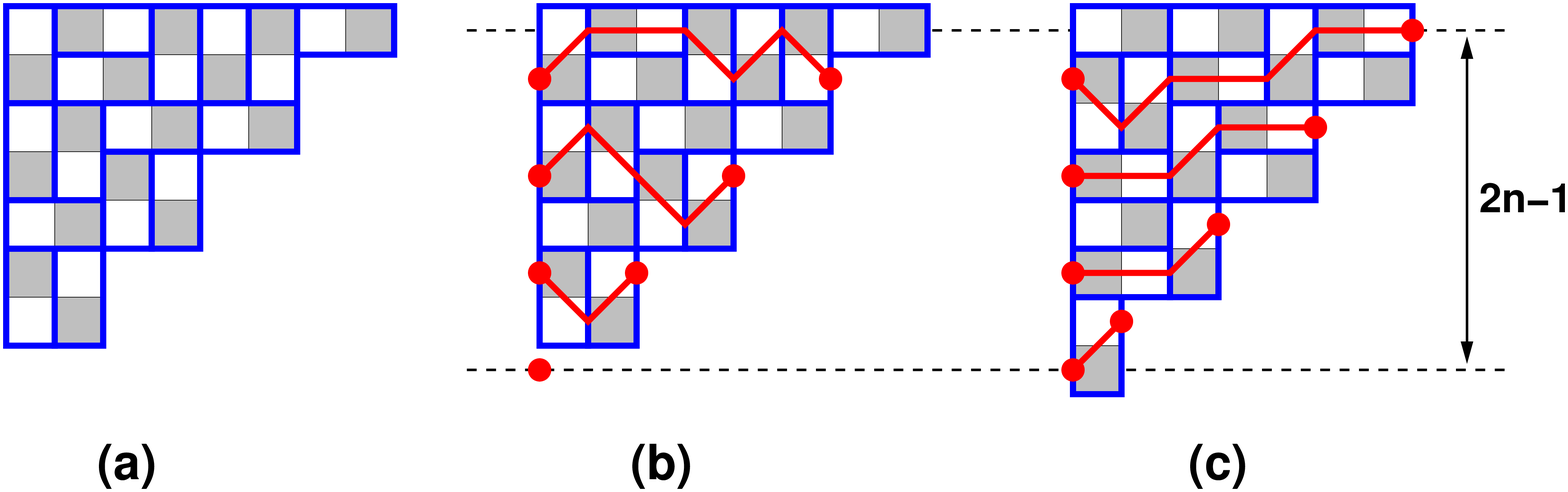}}}
\end{equation*}
Here we have first bi-colored the underlying square lattice, considered a tiling configuration (a), 
and used the dictionary \eqref{dominos} mapping each domino to a path step. We have displayed two equivalent path formulations 
(b) and (c) of the same domain (upon reflection and rotation by $90^\circ$). In both cases, the paths 
are non-intersecting Schr\"oder paths with fixed ends as shown, and {\it constrained to remain within
a strip of height $2n-1$} (we have added a trivial path of length $0$ in the case (b) for simplicity).

The path configurations are best enumerated by means of the LGV formula. 
Let $S_{a,b}^{(L)}(M)$
denote the partition function of a single Schr\"oder path (with steps $(1,1)$, $(1,-1)$ and $(2,0)$),
starting at point $(0,a)$, ending at point $(M,b)$ and
constrained to remain in the strip $0\leq y\leq L$. Then the partition function for domino tilings of ${\mathcal T}_n$
is:
$$ T({\mathcal T}_n)=\det\limits_{0\leq i,j\leq n-1}\left( S_{2i,2j}^{(2n-1)}(2j) \right) 
=\det\limits_{0\leq i,j\leq n-1}\left( S_{2i,2j+1}^{(2n-1)}(2j+1) \right)$$
corresponding respectively to situations (b) and (c).

The single path partition function $S_{a,b}^{(L)}(M)$ may easily be generated from the following recursion relation:
$$ S_{a,b}^{(L)}(M)= \left\{ \begin{matrix}
& 0 \ \  {\rm if} \ a<0\  {\rm or}\  a>L\  {\rm or}\  b<0\  {\rm or}\  b>L ,\\
&S_{a,b}^{(L)}(M-2)+S_{a-1,b}^{(L)}(M-1)+S_{a+1,b}^{(L)}(M-1) \ \  {\rm otherwise},
\end{matrix}\right. $$
together with the initial conditions $S_{a,b}^{(L)}(-1)=0$ and $S_{a,b}^{(L)}(0)=\delta_{a,b}$ when $a,b\in [0,L]$ and $0$ otherwise. This allows to recover the first terms of the sequence $b_n$.

\subsubsection{The DWBC3 conjecture}

In view of the matching of the first values of $B_n$ with the sequence $b_n$, we are led to the following:

\begin{conj}
\label{DWBC3conj}
We conjecture that the number of configurations of the 20V model with DWBC3 boundary conditions
on an $n\times n$ square grid is the same as the number of domino tilings of the triangle ${\mathcal T}_n$.
\end{conj}

We have checked this conjecture numerically up to size $n=6$.

\subsection{Pentagonal extensions of the DWBC3 conjecture}

\begin{figure}
\begin{center}
\includegraphics[width=12cm]{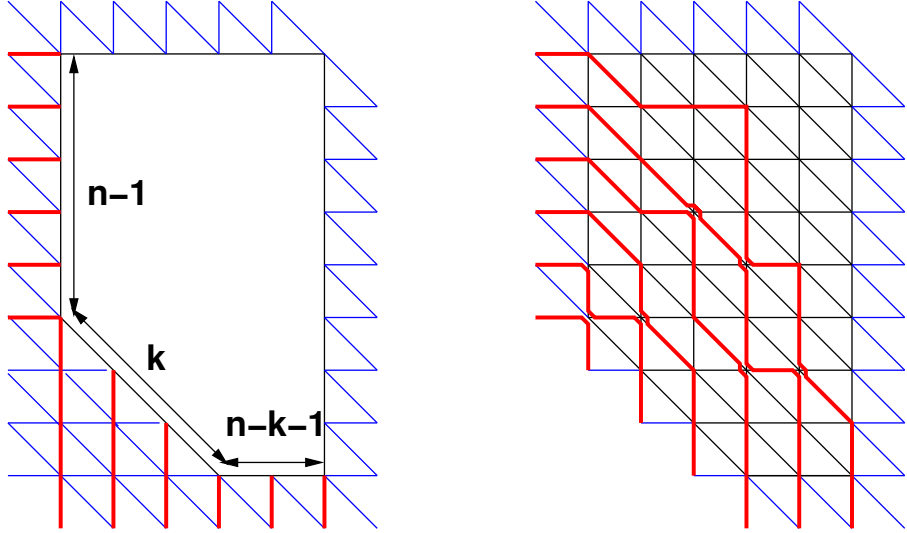}
\end{center}
\caption{\small Left: Extension of the DWB3 boundary on a $(n+k)\times n$ rectangular grid in osculating
Schr\"oder path representation: due to non-intersection
constraints the effective domain is a pentagon of shape $(n-1)\times k\times (n-k-1)\times (n+k-1)\times (n-1)$. 
Right: A sample path configuration.}
\label{fig:hexagon}
\end{figure}

\subsubsection{The 20V-DWBC3 model on a pentagon}\label{sec:penta}
We now consider a variant of the model 20V-DWBC3 on a ``pentagon" $P_{n,k}$ of the original triangular lattice.
Starting from a rectangular grid of shape $(n+k)\times n$, we impose that the $n$ top external horizontal edges along 
the West boundary be occupied by paths, while the $k$ bottom ones be empty (vacancies), and impose the same condition 
as for DWBC3 on vertical external edges along the South boundary, while all other external edges are unoccupied
(see Fig.~\ref{fig:hexagon} for an illustration with $n=6$ and $k=3$). Due to the non-intersection constraint which freezes some portions of the paths, 
the effective domain reduces to a pentagon of shape\footnote{Note the distinction between the notion of grid and that of shape: the size of a shape is measured in actual length of its sides, which is one less than the measure on a grid.} $(n-1)\times k\times (n-k-1)\times (n+k-1)\times (n-1)$ (see Fig.~\ref{fig:hexagon}). 
This holds for $k< n-1$. For $k\geq n-1$, the effective domain degenerates into a tetragon of shape
$(n-1)\times (n-1)\times (2n-2)\times (n-1)$ independent of $k$, which can be viewed as half of a regular hexagon on the
original triangular lattice.

The number $p_{n,k}$ of osculating Scr\"oder path configurations on $P_{n,k}$ is easily computed by transfer matrix techniques.
The first few are listed below for $n=1,2,...,6$
\begin{eqnarray}
p_{n,0}&=& 1,\  3, \ 29, \ \ \ 901, \ \ 89893, \ 28793575 \ ... \nonumber\\
p_{n,1}&=& 1,\  4, \ 56, \ 2640, \ 411840, \ 210613312 \ ... \nonumber\\
p_{n,2}&=& 1,\  4, \ 60, \ 3268, \ 628420, \ 417062340 \ ... \nonumber\\
p_{n,3}&=& 1,\  4, \ 60, \ 3328, \ 675584, \ 495222784 \ ... \nonumber\\
p_{n,4}&=&  1,\  4, \ 60, \ 3328, \ 678912, \ 507356160 \ ...\nonumber\\
p_{n,5}&=&  1,\  4, \ 60, \ 3328, \ 678912, \ 508035072 \ ...\label{penta}
\end{eqnarray}

As expected, we note a saturation property of $p_{n,k}$ which becomes independent of $k$ for $k\geq n-1$. 

\subsubsection{Tilings of extended triangles}\label{sec:extri}
We start from the triangular domain ${\mathcal T}_n$ of Sect.~\ref{sec:triangle} above, and consider the following extensions ${\mathcal T}_{n,k}$, $k=0,1,...,n$.
Focussing on the non-intersecting Schr\"oder path description of the tiling configurations given by the example (c) above, ${\mathcal T}_{n,k}$
corresponds to {\it raising by $k$ vertical steps the top border of the domain accessible to the paths},
while keeping identical starting and endpoints. 
In practice, for $k\leq n-1$, the new effective domain ${\mathcal T}_{n,k}$ accessible to the Schr\"oder paths takes the following shape
(represented here for  $n=4$) which can be viewed as Aztec-like extensions of  ${\mathcal T}_{n}$:
\begin{equation*}
\raisebox{.3cm}{\hbox{\epsfxsize=15.cm \epsfbox{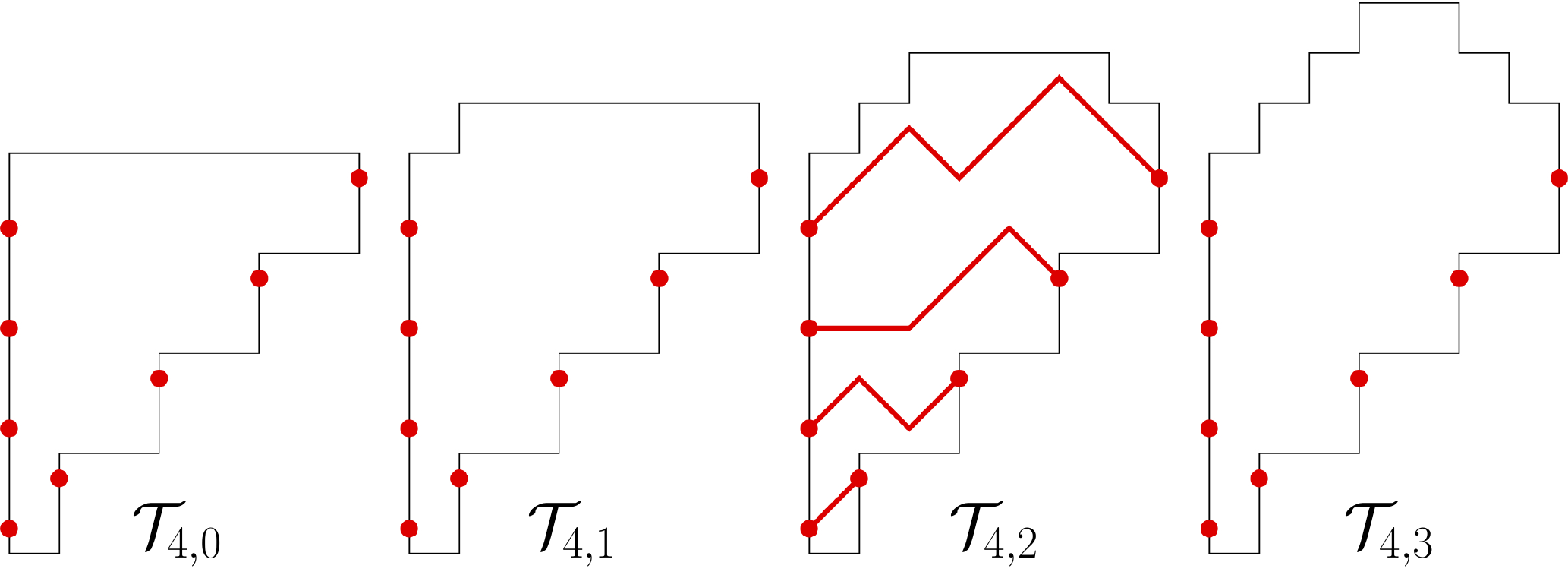}}}
\end{equation*}
As a result, the number of configurations
augments, until it reaches a threshold at $k=n-1$, since, for $k\geq n-1$, raising the top border further no longer affects the number of configurations as the paths never reach this height. 

The counting of tiling configurations is readily performed by use of the LGV formula for the corresponding Schr\"oder paths:
$$ T({\mathcal T}_{n,k}) 
=\det\limits_{0\leq i,j\leq n-1}\left( S_{2i,2j+1}^{(2n-1+k)}(2j+1) \right)$$
where we have simply raised the top boundary by $k$. As a result, we find perfect agreement between
the numbers $p_{n,k}$ and $T({\mathcal T}_{n,k})$.

\subsubsection{The extended DWBC3 conjecture}

The remarkable coincidence between the numbers in Sections \ref{sec:penta} and \ref{sec:extri} leads to the following:
\begin{conj}
\label{pentaconj}
We conjecture that the number of configurations of the 20V model with extended DWBC3 boundary conditions
on a pentagon $P_{n,k}$ is equal to the number of domino tilings of the extended triangle ${\mathcal T}_{n,k}$ for all $n,k$.
\end{conj}

The conjecture has been checked numerically for $n$ up to 6 and arbitrary $k$.

\subsection{More Domain Wall Boundary Conditions with no conjecture}\label{sec:oth}
\begin{figure}
\begin{center}
\includegraphics[width=12cm]{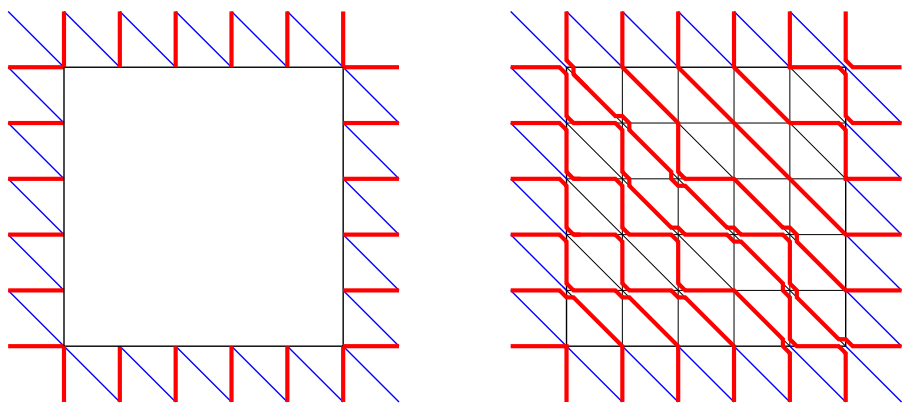}
\end{center}
\caption{\small Left: 20V-DWBC4 on an $n\times n$ square grid: all horizontal and vertical external edges are occupied by paths, while diagonals are empty.
Right: A sample osculating Schr\"oder path configuration.}
\label{fig:other}
\end{figure}

We have considered some other variants of the DWBC boundary conditions, but found no conjecture for those.
First we studied the 20V with the DWBC4 illustrated in Fig.~\ref{fig:other}
for the osculating Schr\"oder path formulation, in which all horizontal and vertical external edges 
of a square $n\times n$ grid are occupied by paths, all the other external edges being unoccupied\footnote{We use here the
denomination DWBC4 for simplicity although the boundary conditions do not infer the creation of domain walls in general, as exemplified 
by the trivial configuration where all horizontal (resp. vertical, diagonal) arrows point East (resp. South, Northwest).}.
The transfer matrix calculation leads to the following sequence:
\begin{equation}\label{sq4}
1,\  3,\  59,\  7813,\  6953685,\ 41634316343 ... \end{equation}
for which we have found no other interpretation.

\begin{figure}
\begin{center}
\includegraphics[width=8cm]{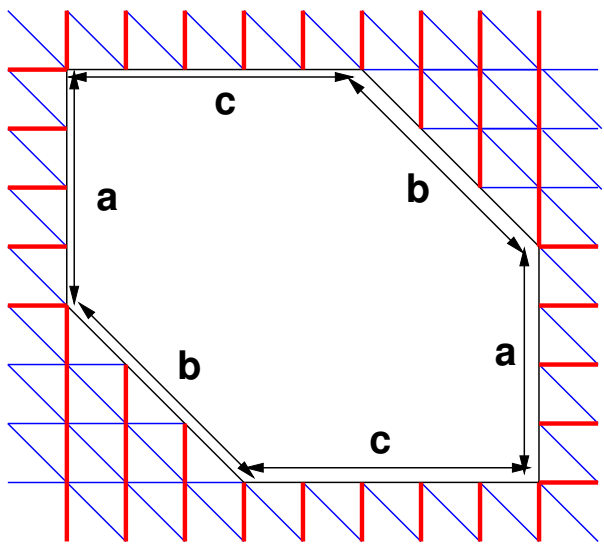}
\end{center}
\caption{\small The hexagonal extension of 20V-DWBC4 obtained on a rectangular  $(a+b+1)\times (b+c+1)$
grid
by imposing $b$ vacancies on the bottom West and top East boundaries.}
\label{fig:more}
\end{figure}

\begin{table}
  \begin{center}
    \label{Nabc}
    \begin{tabular}{|c |l |l |l|l |l | l | l|l|} 
\hline
 b,c & a=\textbf{0} & \textbf{1} & \textbf{2}& \textbf{3}& \textbf{4}& \textbf{5}& \textbf{6} 
\\
\hline
\bf 0,1 & 1& 3& 8 & 21 & 55 & 144 & 377 \\
\hline
\bf 0,2 & 1& 8& 59& 415& 2874& 19810& 136358 \\
\bf 1,1 & 3&  11& 41&  153& 571&  2131&  7953  \\
\hline
\bf 0,3 & 1& 21& 415& 7813& 143336& 2598735& 46881130 \\
\bf 1,2 &  8& 85& 959& 10934& 124869& 1426389 &16294360\\
\bf 2,1 &  5& 23& 103& 456& 2009& 8833& 38803 \\
\hline
\bf 0,4 &  1& 55& 2874& 143336& 6953685& 331859360& 15697347566 \\
\bf 1,3 &   21& 604& 19018& 615405& 20055060& 654666505 &21378877706\\
\bf 2,2 &   20& 333& 5331& 83821& 1305844& 20250090 &313317426\\
\bf 3,1 &  7& 39& 201& 1000& 4888&  23673 &114087\\
\hline
\bf 0,5 &  1& 144& 19810& 2598735& 331859360& 41634316343& 5164420164680 \\
\bf 1,4 &   55& 4194& 355234& 31391724& 2816672309& 254000932538 &22940968768675\\
\bf 2,3 &   76& 4151& 213173& 10696445& 530068706& 26081095911 &1278122145554\\
\bf 3,2 &   36 &881 &18859 &379449 &7391755 &141473217 &2681264915\\
\bf 4,1 &   9 &59 &343 &1880 &9976 &51944 &267385\\
\hline
    \end{tabular}
    \vskip.5cm
     \caption{The first numbers $N_{a,b,c}$ for $a\in [0,6]$,  $b\geq 0$, $c>0$ and $b+c\leq 5$.}
  \end{center}
\end{table}

\begin{figure}
\begin{center}
\includegraphics[width=16cm]{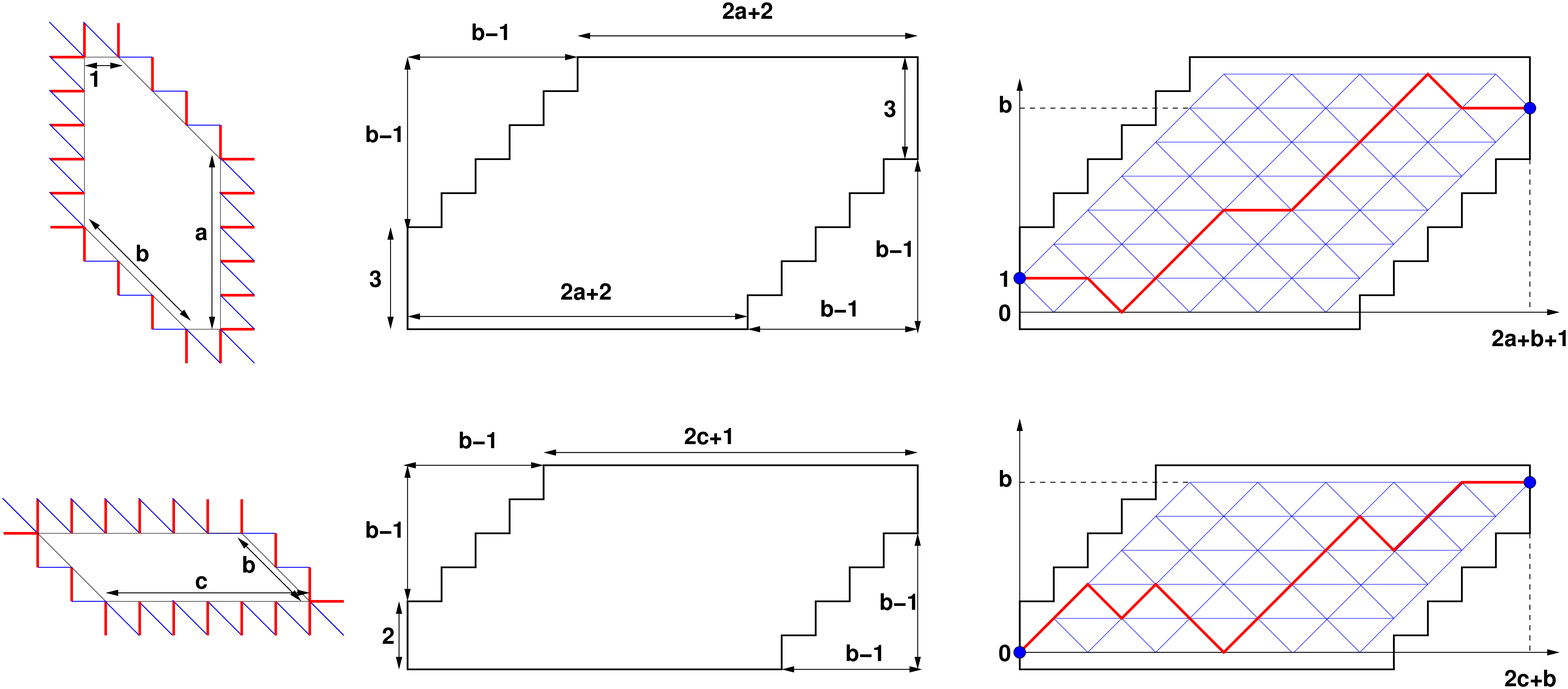}
\end{center}
\caption{\small Top: the number of 20V-DWBC4 configurations on the hexagon $(a,b,c=1)$ (left) is conjecturally identified with the number of domino tilings of a domain
(center), in bijection with configurations of a single Schr\"oder path (right) from $(0,1)$ to $(2a+b+1,b)$, constrained to stay on a strip $0\leq y\leq b+1$. 
Bottom: the number of 20V-DWBC4 configurations on the hexagon $(a=0,b,c)$ (left) is conjecturally identified with the number of domino tilings of a domain
(center), in bijection with configurations of a single Schr\"oder path (right) from $(0,0)$ to $(2c+b,b)$, constrained to stay on a strip $0\leq y\leq b$.}
\label{fig:domain3}
\end{figure}

An easy generalization consists in considering arbitrary rectangular grids of size $(a+b+1)\times (b+c+1)$, and 
imposing that all vertical external edges be occupied while only the top $a+1$ horizontal  ones on the West boundary, and bottom $a+1$ horizontal ones on the East boundary be occupied (see Fig.~\ref{fig:more}). Equivalently the bottom $b$ horizontal external edges on the West boundary and the top $b$ ones on the East boundary are empty. This leads to numbers $N_{a,b,c}$ of configurations. 

The case $b=0$ corresponds
to rectangular grids of arbitrary size $(a+1)\times (c+1)$, for which the boundary condition DWBC4 still makes sense. In particular, $N(n-1,0,n-1)$ reproduces the above sequence \eqref{sq4}.
Note also that $N_{a,b,0}=1$, as there is a unique, fully osculating, configuration in that case, as illustrated
for $(a,b,c)=(2,3,0)$ below:
$$
\raisebox{-2.5truecm}{\hbox{\epsfxsize=4.cm \epsfbox{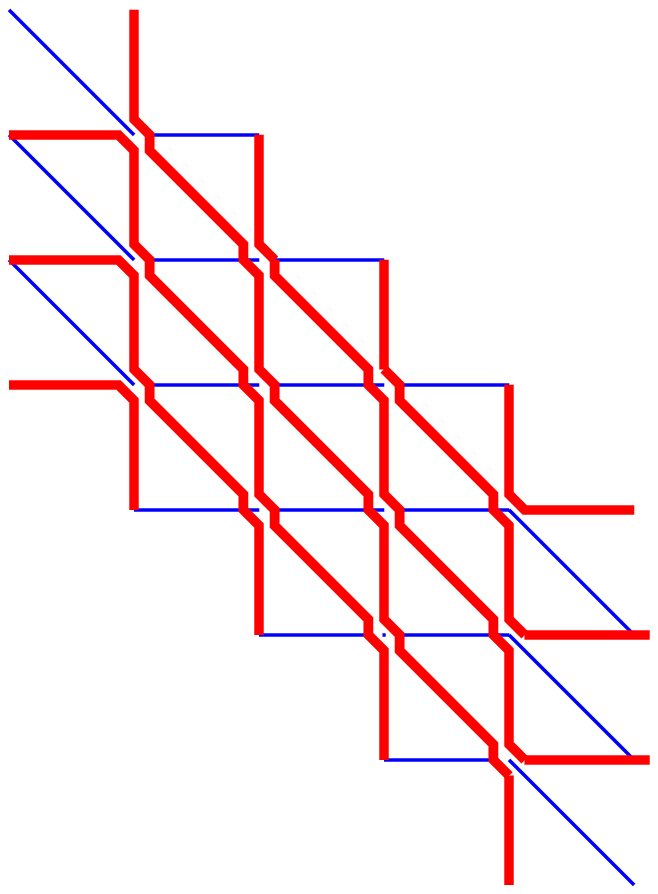}}}
$$

We have listed some of the non-trivial numbers $N_{a,b,c}$ in Table \ref{Nabc}. So far we have 
identified conjecturally only the numbers $N_{a,b,1}$ and $N_{0,b,c}$ as also counting the domino tilings of plane domains
depicted in Fig.~\ref{fig:domain3}. The latter are
easily enumerated by the configurations of a single Schr\"oder path with fixed ends, and
constrained to remain within a strip, as shown on the right. More precisely, we have, with the notation of Section \ref{sec:triangle}:
\begin{conj}
$$ N_{a,b,1}= S^{(b+1)}_{1,b}(2a+b+1),\qquad N_{0,b,c}=S^{(b)}_{0,b}(2c+b) .$$
\end{conj}

\section{Alternating Phase Matrices}\label{sec:apm}

We may reformulate the 20V models with DWBC1,2,3 in terms of Alternating Phase Matrices (APM),
which generalize the Alternating Sign Matrices (ASM),
with entries among $0$ and the sixth roots of unity, and with specific alternating conditions.

\subsection{From 20V configurations to APM}
In the case of the 6V model with DWBC, one possible construction of ASM is by viewing the six possible vertex
configurations as ``transmitters" or ``reflectors" of the orientation of the arrows when going say from left to 
right and from top to bottom. A vertex either reflects or transmits both directions as a consequence of the {\it ice rule}.
Starting from a 6V-DWBC configuration, we map vertices to entries of the ASM built 
according to the following rules:
\begin{enumerate}
\item If the vertex is a transmitter, the entry of the ASM is $0$;
\item If the vertex is a reflector, the entry is $+1$ if the horizontal arrows point inwards, and $-1$ otherwise.
\end{enumerate}

In the case of the 20V model, each vertex is now viewed as a triple of reflectors or transmitters along
the horizontal, vertical and diagonal directions, say going from NW to SE. To each vertex of a 20V-DWBC1,2 or 3
configuration, we may assign a triple $(h,v,d)$ of elements of $\{0,1,-1\}$
where $h$, $v$ and $d$ indicate the transmitter of reflector state of the horizontal, vertical and diagonal directions respectively,
with the following rules:
\begin{enumerate}
\item If the vertex is a transmitter along a direction, the corresponding entry of the triple is $0$
\item If the vertex is a reflector along a direction, the corresponding entry is $+1$ if the arrows point inwards, and $-1$ otherwise.
\end{enumerate}
For each triple at a vertex, we have the condition $h+d+v=0$ as a consequence of the {\it ice rule}, which imposes that either the vertex is a transmitter
in all three directions (and then $h=d=v=0$), or it is a transmitter in only one direction and a reflector in the other two, and in this latter case, one reflected pair
points inwards and the other outwards, so that the corresponding entries of the triple have zero sum.
This gives rise to seven possible triples: the triple $(0,0,0)$ encountered for 8 of the 20 vertices and six non-zero possible triples, 
according to the dictionary below in terms of osculating Schr\"oder paths:
\begin{eqnarray}
&&\raisebox{-.6cm}{\hbox{\epsfxsize=12.cm \epsfbox{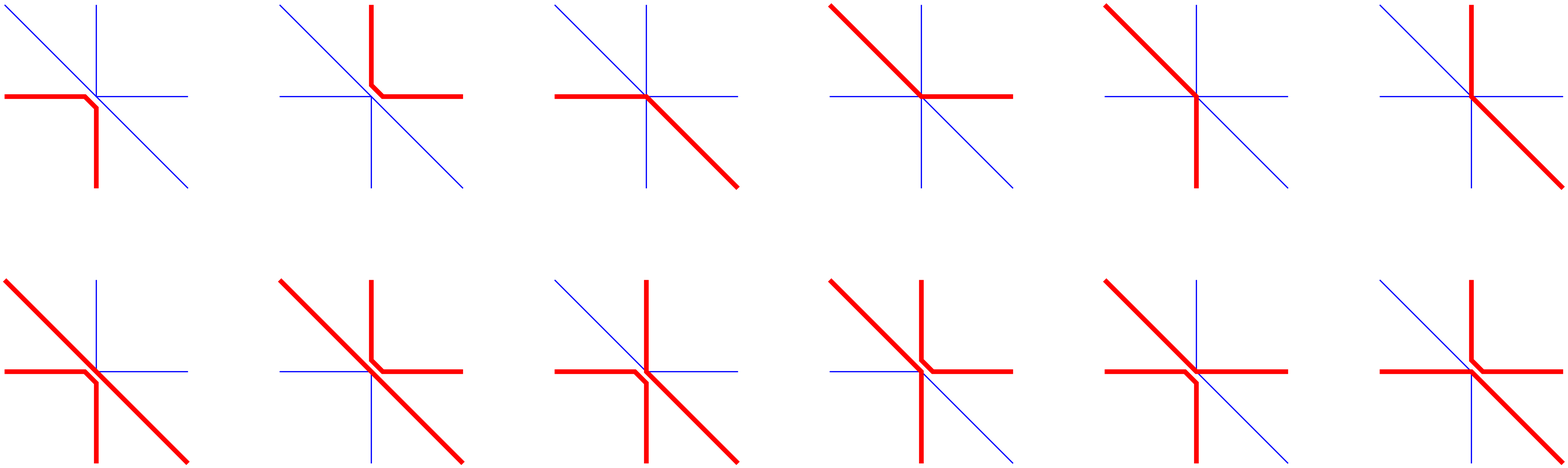}}}\nonumber \\
&&\hskip .0cm (1,-1,0) \hskip .5cm (-1,1,0) \hskip .9cm (1,0,-1)\hskip .5cm (-1,0,1)\hskip .7cm (0,-1,1)\hskip .8cm 
(0,1,-1) \label{apmtriples}\\
&&\hskip .6cm 1 \hskip 1.5cm -1 \hskip 1.5cm -\omega \hskip 1.8cm \omega \hskip 1.55cm -\omega^2\hskip 1.55cm \omega^2\label{apmsixth}
\end{eqnarray}
We have also indicated an alternative bijective coding using the sixth roots of unity with $\omega=e^{2i\pi/3}$, the weight of
a triple $(h,v,d)$ being bijectively mapped onto the complex number $-\omega h+\omega^2 v$ (which includes the coding
$(0,0,0)\mapsto 0$). 
This allows to assign to each 20V configuration of size $n$ with DWBC1,2 or 3 an $n\times n$  matrix with elements 
either $0$ or in 
the set of sixth roots of unity: we shall call such matrices Alternating Phase Matrices (APM) of type 1,2 or 3 respectively. 
Note that all the ASMs  are realized as APMs (with only $\pm 1$ non-zero entries), in all three types. To see why,
start from the osculating path formulation of ASM on the square grid, and then superimpose diagonal entirely
empty or entirely occupied lines, according to the diagonal external edge states pertaining to the chosen DWBC. This 
generates APMs with entries $0$, $1$ and $-1$ only (corresponding to the first two columns in the figure above)
equal to those of the corresponding ASMs.

Alternatively, we may understand the sixth root of unity weights as a weighting of the turns taken
by individual paths, with the rule that the total weight at a vertex is the {\it sum} of the turning weights of all the 
paths visiting it. 
More precisely, the expression $-\omega h+\omega^2 v$ for vertex weights may be understood as the sum 
over all two-step paths $p=(p_{in},p_{out})$ visiting the vertex of turning weights equal to the variation ${\rm turn}(p):=\eta(p_{out})-\eta(p_{in})$ of an edge variable $\eta$ 
with the following dictionary:
$$
\raisebox{-.6cm}{\hbox{\epsfxsize=5.cm \epsfbox{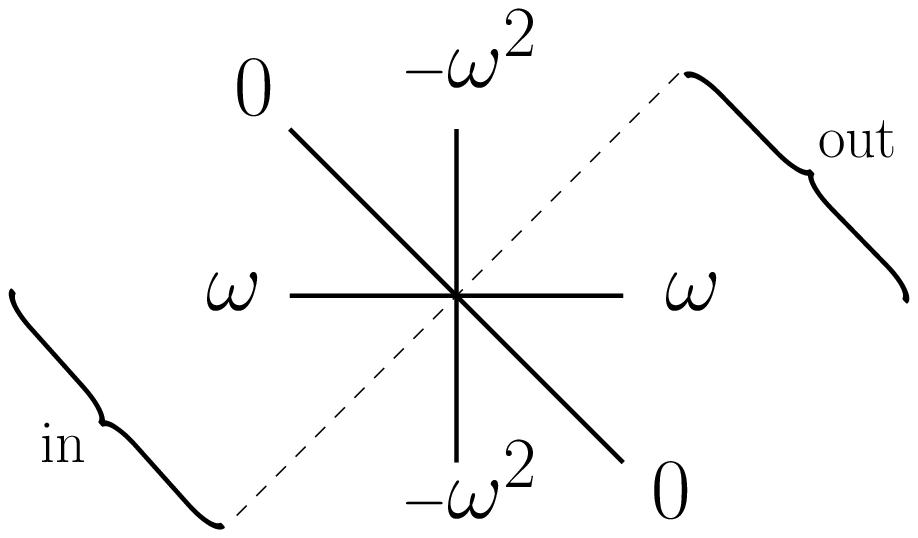}}}
$$
where we have indicated the value of $\eta$ for each in/out edge.
With this dictionary, a path which does not turn (transmitter direction) receives the turning weight $0$
while we have the following assignment of turning weight for respectively horizontal, diagonal and 
vertical incoming paths (we have indicated the direction of travel of each path by a straight arrow):
\begin{eqnarray}
&&\raisebox{-.6cm}{\hbox{\epsfxsize=12.cm \epsfbox{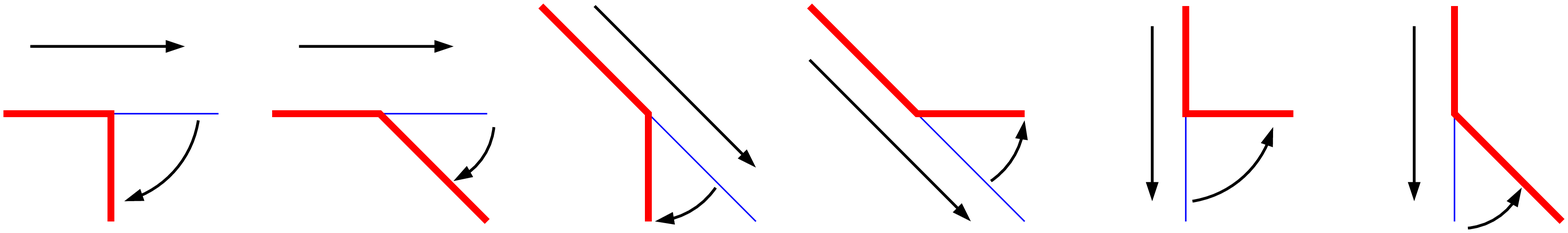}}}\nonumber \\
&&\hskip .8cm 1 \hskip 1.5cm -\omega \hskip 1.3cm -\omega^2 \hskip 1.6cm \omega \hskip 1.55cm -1\hskip 1.75cm \omega^2 \label{turning}
\label{eq:turning}
\end{eqnarray}
It is a straightforward exercise to check that all the weights of \eqref{apmsixth} are indeed the sums
of turning weights of all their two-step paths (this is indeed guaranteed by the mapping $(h,v,d)\mapsto -\omega h+\omega^2 v$). 
For instance, the fifth weight from left on the bottom line is $1+\omega=-\omega^2$, namely equal to the sum of the first and fourth 
weights on \eqref{eq:turning}. Moreover, a useful consequence of the definition is that the turning weights of a given path $p=(p_0,p_1,\ldots, p_k)$ add up along the path
into a telescopic sum equal to 
\begin{equation}
{\rm turn}(p):=\sum_{\ell=1}^k {\rm turn}((p_{\ell-1},p_\ell))=
\eta(p_k)-\eta(p_0)={\rm turn}((p_0,p_k))
\label{eq:turndef}
\end{equation}
depending only on the orientations of its first and last edges.

\subsection{APM of type 1, 2 and 3:  definitions and properties}

The matrix triple entries $(h,v,d)$ coming from 20V configurations are further constrained by reflection/transmission properties along the horizontal, vertical and diagonal directions.
In all three cases of DWBC1,2,3 all the external horizontal arrows point towards the grid and all the external vertical arrows point away from it. If we follow any horizontal line from left to right, the first arrow on the West side points to the right,
and must be reflected at least once as it points left when it exits on the East side. It must in fact be reflected an odd number
of times. The corresponding entries $h_j$ of the triples associated to the vertices $j=1,2,...,n$ visited from left to right
must therefore alternate between $1,-1,1,..,1$ whenever they are non-zero.
The same reasoning for the vertical directions leads, from top to bottom, to an alternation of the entries $v_i$ in each column 
between $-1,1,-1,...,-1$ whenever non-zero. 

Finally a last  alternance condition holds along the diagonal directions, but is different for DWBC1,2 and 3.
For DWBC3, as all external diagonal edges are empty, the entries $d_i$ along each diagonal visited from 
top to bottom, must alternate between $-1,1,-1,...,1$ when they are non-zero: note that there is indeed 
always an even, possibly zero number of reflectors as the arrows at both ends point towards Northwest.
For DWBC1, recall that the external diagonal edges are occupied in the lower triangular part and empty in the strictly upper triangular part of the $n\times n$ grid. As a consequence, the entries $d_i$ 
(labeled by, say the row index $i$ in increasing order) along each diagonal 
in the strictly upper triangular part obey the same rule as in DWBC3, namely alternate between $-1,1,-1,...,1$ when they are non-zero, but the entries $d_i$ along each diagonal 
in the lower triangular part obey the opposite rule, namely alternate between $1,-1,1,...,-1$ when they are non-zero.
Finally, for DWBC2 the entries $d_i$ along each diagonal 
in the upper triangular part alternate between $-1,1,-1,...,1$ when they are non-zero, and the entries $d_i$ along each diagonal 
in the strictly lower triangular part obey the opposite rule, namely alternate between $1,-1,1,...,-1$ when they are non-zero.

These rules determine entirely the sets of APM of type 1,2,3 whose definition is summarized below.
\begin{defn}
We define the sets of $n\times n$ Alternating Phase Matrices of types 1,2,3 as $n\times n$ arrays of triples
of the form $(h_{i,j},v_{i,j},d_{i,j})$ for $1\leq i,j \leq n$, where $h_{i,j},v_{i,j},d_{i,j}\in \{0,1,-1\}$ satisfy $h_{i,j}+v_{i,j}+d_{i,j}=0$ and are moreover subject to the following conditions for all types:
\begin{enumerate}
\item There is at least one non-zero variable $h_{i,j}$ in each row $i=1,2,...,n$, and at least one non-zero 
variable $v_{i,j}$ in each column $j=1,2,...,n$,
\item Along each row $i=1,2,...,n$, the non-zero variables $h_{i,j}$ must alternate between $1,-1,1,$ $...,1$ when $j$ ranges from $1$ to $n$,
\item Along each column $j=1,2,...,n$, the non-zero variables $v_{i,j}$ must alternate between $-1,1,-1,...,-1$ when $i$ ranges from $1$ to $n$,
\end{enumerate}
and to three different conditions $(4.1),(4.2),(4.3)$ corresponding to each type 1,2,3:
\begin{enumerate}
\item[(4.1)]  Along each diagonal\footnote{Here we label the diagonal from $1-n$ to $n-1$ from top to bottom, as opposed to the 
spectral parameter labelling $t_k$, $k=1,2,\ldots, 2n-1$ from bottom to top. The correspondence is $k=n-\ell$.} $\ell\in [1-n,n-1]$, the non-zero variables $d_{i,i+\ell}$ 
must alternate between
$-1,1,-1,...,1$ if $\ell>0$ and $1,-1,1,...,-1$ if $\ell\leq 0$ when $i$ ranges from ${\rm Max}(1,1-\ell)$ to 
${\rm Min}(n,n-\ell)$.
\item[(4.2)]  Along each diagonal $\ell\in [1-n,n-1]$, the non-zero variables $d_{i,i+\ell}$ 
must alternate between
$-1,1,-1,...,1$ if $\ell\geq 0$ and $1,-1,1,...,-1$ if $\ell< 0$ when $i$ ranges from ${\rm Max}(1,1-\ell)$ to 
${\rm Min}(n,n-\ell)$.
\item[(4.3)] Along each diagonal $\ell\in [1-n,n-1]$, the non-zero variables $d_{i,i+\ell}$ 
must alternate between
$-1,1,-1,...,1$ when $i$ ranges from ${\rm Max}(1,1-\ell)$ to 
${\rm Min}(n,n-\ell)$.
\end{enumerate}
\end{defn}

APM entries are expressible ubiquitously in terms of either the above defining triples, or 
equivalently zero or sixth roots of unity
according to the dictionaries (\ref{apmtriples}-\ref{apmsixth}).

\begin{prop}
The $n\times n$ APM of types 1,2,3 are in bijection with respectively the 20V-DWBC1,2,3 
on an $n\times n$ grid.
\end{prop}

\begin{example}
As a illustration the $5\times 5$ APM of types 1,2 and 3 corresponding to the configurations depicted respectively in Figs.~\ref{fig:DWBC1},
\ref{fig:DWBC2} and \ref{fig:grid3} read:
\begin{eqnarray*}
\raisebox{-2.cm}{\hbox{\epsfxsize=4.cm \epsfbox{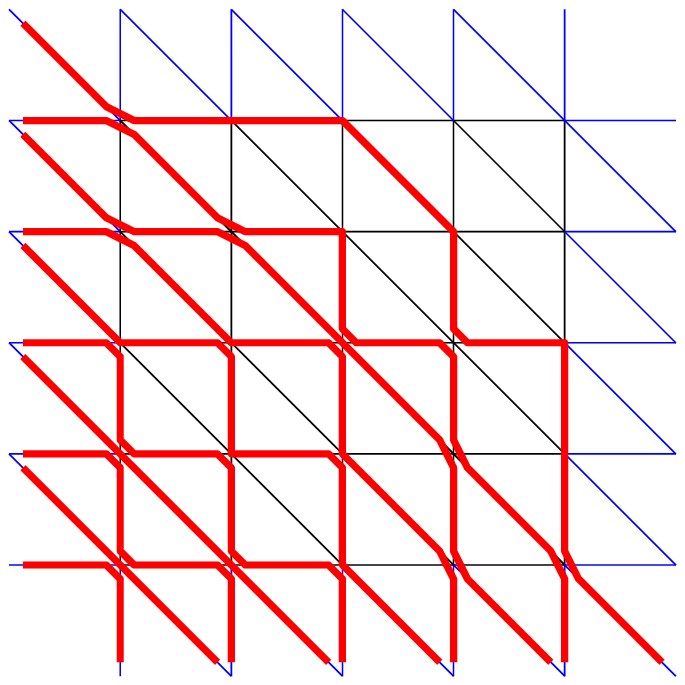}}}
&\mapsto& \begin{pmatrix}
0 &0 &  -\omega & 0 & 0 \\
0 & 0 & 1 &  -\omega^2 & 0 \\
 -\omega^2 &  -\omega^2 & 0 & 0 & 1 \\
0 & 0 & -\omega & 0 & 0\\
0 &0 & -\omega & 0 & 0
\end{pmatrix} \\
\raisebox{-2.cm}{\hbox{\epsfxsize=4.cm \epsfbox{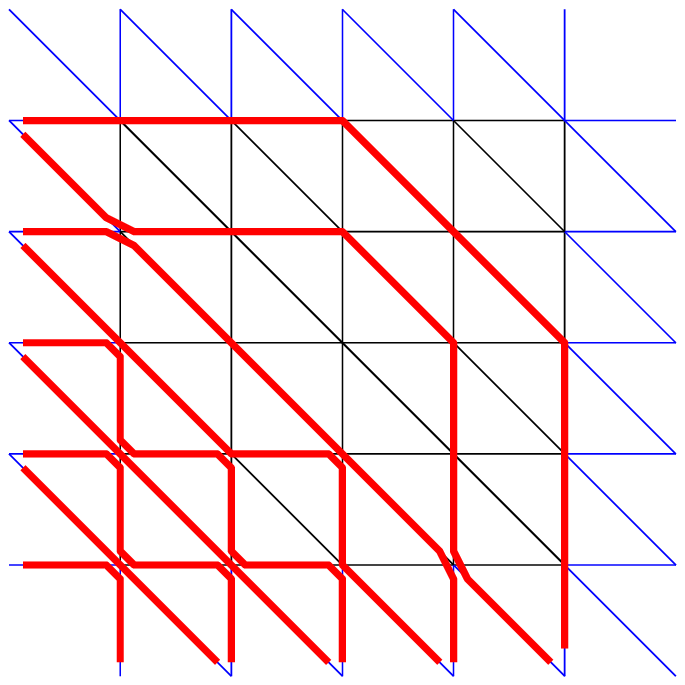}}}
&\mapsto& \begin{pmatrix}
0 &0 &  -\omega & 0 & 0 \\
0 & 0 &  -\omega & 0 &0 \\
1 &  0 & 0 & -\omega^2 & -\omega^2 \\
 0 & -\omega^2 & 1 & 0 & 0\\
0 &0 & -\omega & 0 & 0
\end{pmatrix} \\
\raisebox{-2.cm}{\hbox{\epsfxsize=4.cm \epsfbox{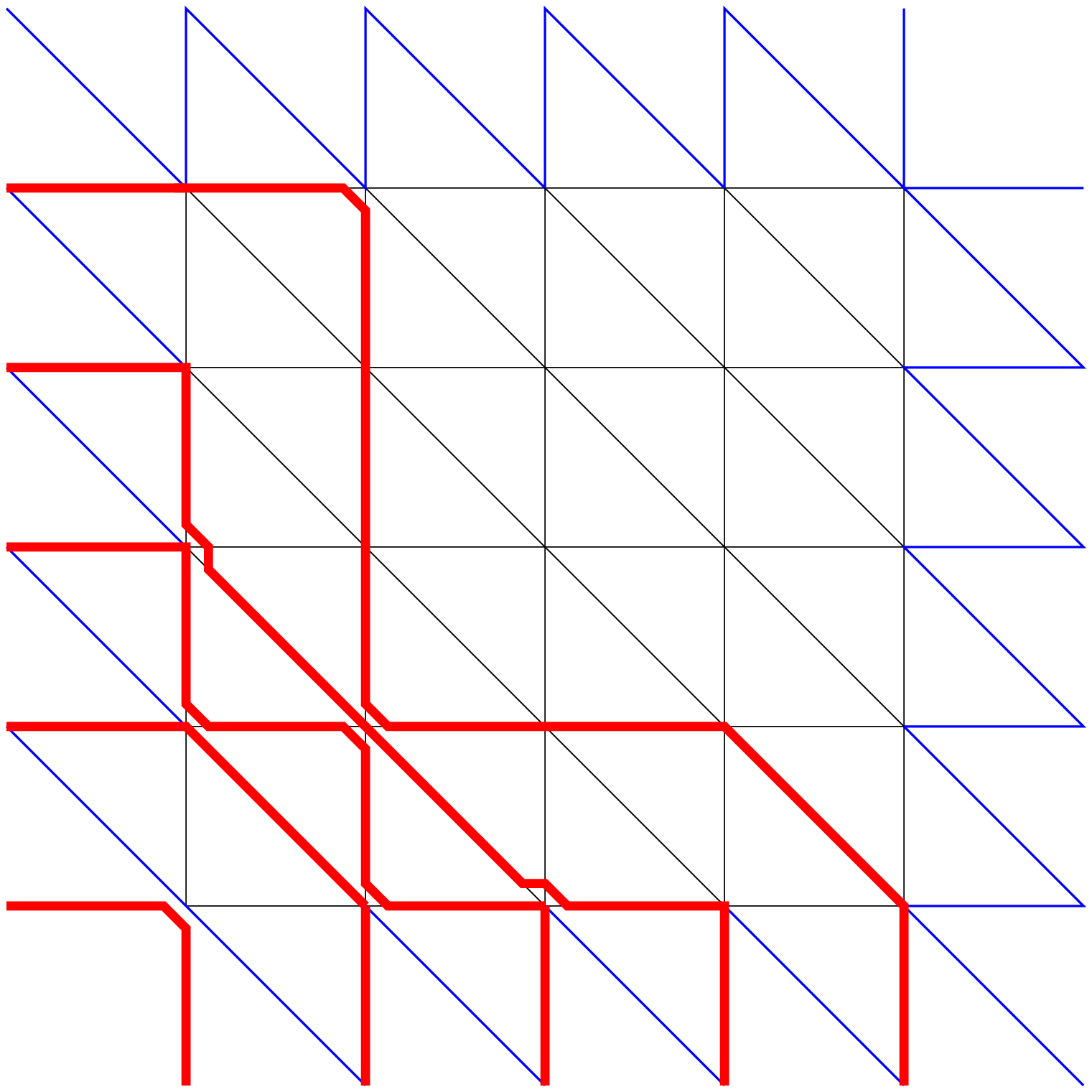}}}
&\mapsto& \begin{pmatrix}
0 &1 & 0 & 0 & 0 \\
1 & 0 & 0 & 0 & 0 \\
-\omega & 0 & 0 & 0 & 0 \\
\omega^2 & 0 & 0 & -\omega & 0\\
1 & \omega & -\omega^2 & 1 & -\omega^2
\end{pmatrix}
\end{eqnarray*}
We note that the first and second APMs of respective types 1 and 2 are exchanged under a rotation 
by $180^\circ$, which matches the fact that the corresponding 20V configurations are also interchanged
under the same transformation.
\end{example}

We conclude with a simple property satisfied by all APMs introduced so far.

\begin{prop}\label{sumrule}
APMs $A=(a_{i,j})_{1\leq i,j \leq n}$ of any type 1, 2 or 3, when expressed in terms of sixth roots of 
unity, have the following 
property:
$$ \sum_{i,j=1}^n a_{i,j} = n\ . $$
\end{prop}
\begin{proof}
We use the interpretation of weights in terms of turning weights for the paths \eqref{turning}. For each path $p$
in the osculating Schr\"oder path configuration, recall the quantity 
${\rm turn}(p)$ \eqref{eq:turndef} defined as  
the sum over its turning weights. Then it is clear that the desired quantity is
$\sum_{i,j=1}^n a_{i,j}=\sum_p {\rm turn}(p)$, where the sum extends over all the paths in the configuration.
Next, we note that irrespectively of the type of DWBC,
all paths have the following property: 
\begin{enumerate}
\item each path starting with a horizontal external edge on the West boundary
ends with a vertical edge on the South boundary (call these HV paths). From \eqref{eq:turndef}, these paths have a turning weight ${\rm turn}(p)=-\omega^2-\omega=1$.
\item each path starting with a diagonal external edge on the West boundary
ends with a diagonal edge on the South boundary.  From \eqref{eq:turndef}, these paths have a turning weight ${\rm turn}(p)=0-0=0$.
\end{enumerate}
The Proposition follows from the fact that there are always exactly $n$ HV paths in all configurations.
\end{proof}

\subsection{Other APMs}

Using the same dictionaries as in the previous sections, we now define APM of type 4 corresponding
to 20V-DWBC4 of Section \ref{sec:oth}. 

\begin{defn}
We define APM of type 4 as $n\times n$ arrays of triples
of the form $(h_{i,j},v_{i,j},d_{i,j})$ for $1\leq i,j \leq n$, where 
$h_{i,j},v_{i,j},d_{i,j}\in \{0,1,-1\}$ and $h_{i,j}+v_{i,j}+d_{i,j}=0$, moreover subject to the following conditions:
\begin{enumerate}
\item Along each row $i=1,2,...,n$, the non-zero variables $h_{i,j}$ must alternate between $1,-1,1,$ $...,-1$ 
when $j$ ranges from $1$ to $n$,
\item Along each column $j=1,2,...,n$, the non-zero variables $v_{i,j}$ must alternate between $1,-1,1,...,-1$ 
when $i$ ranges from $1$ to $n$,
\item Along each diagonal $\ell\in [1-n,n-1]$, the non-zero variables $d_{i,j}$ must alternate between $-1,1,-1,...,1$ 
when $i$ ranges from ${\rm Max}(1,1-\ell)$ to 
${\rm Min}(n,n-\ell)$.
\end{enumerate}
\end{defn}

\begin{prop}
The $n\times n$ APM of type 4 are in bijection with the configurations of the 20V-DWBC4 model on an $n\times n$ grid.
\end{prop}

In particular, the zero matrix is an APM of type 4, which corresponds to all vertices
being transmitters. The APM of type 4 are indeed very different from those of types 1,2,3: we note in particular
that ASM do not form a subset of APM of type 4, as the DWBC4
is incompatible with the DWBC of the underlying square lattice leading to ASM\footnote{This echoes the previous remark
that DWBC4 are not really domain wall inducing.}. Moreover, Proposition \ref{sumrule} no longer holds, and is replaced by
\begin{prop}\label{propzer}
For Any APM $A=(a_{i,j})_{1\leq i,j\leq n}$ of type 4, when expressed in terms of sixth roots of unity, we have the following 
property:
$$\sum_{i,j=1}^n a_{i,j} =0$$
\end{prop}
\begin{proof}
We use the same argument as in the proof of Proposition \ref{sumrule}. All $n$ paths entering from the top
must exit on the right. Each of them has a turning weight $\omega-(-\omega^2)=-1$. All $n$ paths entering from the left
must exit on the bottom. Each of them has turning weight $-\omega^2-\omega=+1$. The total turning weight is therefore $n-n=0$.
\end{proof}

We also have the following refined sum rules:
\begin{prop}\label{magic}
For Any APM $A=(a_{i,j})_{1\leq i,j\leq n}$ of type 4, when expressed in terms of sixth roots of unity, we have the following 
properties:
\begin{equation*}
\sum_{j=1}^n a_{i,j} \in \omega^2\, \Z, \qquad
\sum_{i=1}^n a_{i,j} \in \omega\, \Z , \qquad \sum_{i={\rm Max}(1,1-\ell)}^{{\rm Min}(n,n-\ell)} a_{i,i+\ell} \in  \Z \ .
\end{equation*}
\end{prop}
\begin{proof}
We use the fomulation of the matrix entries $a_{i,j}$ as the sum of turning weights of all paths visiting the vertex $i,j$.
Focussing on a given row $i$ of the matrix, the horizontal occupied edges form a  union of segments of edges
belonging each to a different path:
$[a_0,a_1]$, $[a_2,a_3]$,..., $[a_{2k-2},a_{2k-1}]$ with $k>1$, 
$a_0=0\leq a_1\leq \cdots \leq a_{2k-2}\leq a_{2k-1}=n+1$.
Inside each segment all vertices are transmitters, with turning weights 0. All the junctures $a_1,a_2,...,a_{2k-2}$
are turning points. Inspecting the possibilities of local configurations of paths through these points,
there are two possibilities of exiting the row at each $a_j$ for odd $j<2k-1$:
\begin{eqnarray*}
&&\raisebox{-.6cm}{\hbox{\epsfxsize=5.cm \epsfbox{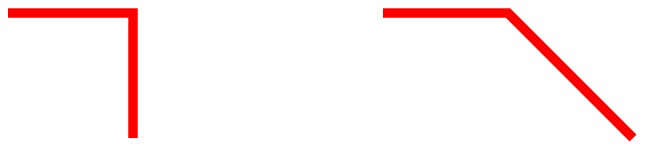}}} \\
&&\hskip .6cm 1 \hskip 1.5cm  \hskip 1.3cm -\omega \hskip 1.6cm \hskip 1.55cm\hskip 1.75cm 
\end{eqnarray*}
and two possibilities to enter the row at each $a_j$ for even $j>0$:
\begin{eqnarray*}
&&\raisebox{-.6cm}{\hbox{\epsfxsize=5.cm \epsfbox{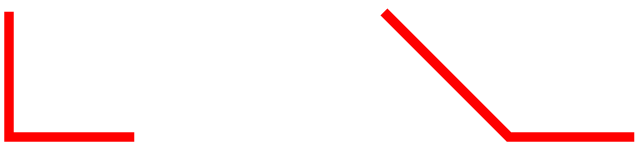}}} \\
&&\hskip .3cm -1 \hskip 1.5cm  \hskip 1.3cm \omega \hskip 1.6cm \hskip 1.55cm\hskip 1.75cm 
\end{eqnarray*}
Let $\theta_j$ denote the corresponding turning weight, with $\theta_j\in \{1,-\omega\}$ for odd $j$ and $\theta_j\in \{ -1, \omega\}$
for even $j$.
Summing over all turning weights along the row $i$ gives $\sum_{j=1}^n a_{i,j}= \sum_{j=1}^{k-1} (\theta_{2j-1}+\theta_{2j})$: this
includes double and also
triple osculations, as the latter must have a transmitter diagonal which does not affect the turning weight.
The quantity $\theta_{2j-1}+\theta_{2j}$ may take only the values: $1-1=0$, $1+\omega=-\omega^2$, $-\omega-1=\omega^2$ and $-\omega+\omega=0$, 
all in $\omega^2 \Z$, 
and the first assertion of the Proposition follows. The second and third ones follow from a similar argument.
\end{proof}

\begin{example}
As a illustration, the $6\times 6$ APM of type 4 corresponding to the configuration on the 
left in Fig.~\ref{fig:other}  reads:
$$
\raisebox{-2.cm}{\hbox{\epsfxsize=4.cm \epsfbox{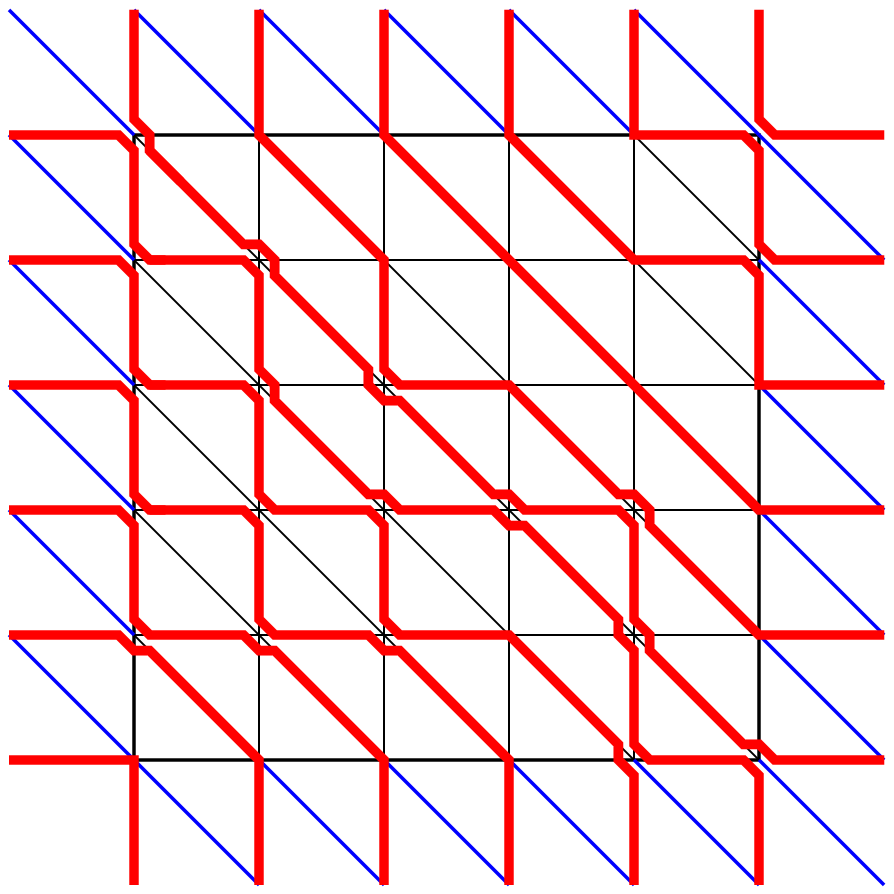}}}
\mapsto \begin{pmatrix}
 -\omega &\omega^2 &  \omega^2 & \omega^2 & -1 & 0\\
0 & 1 & -\omega^2 &  0 &  \omega & 0 \\
0 &  -\omega & -1 & -\omega & 0 & -1 \\
0 & 0 & -\omega^2 & 0 & 1 & \omega\\
\omega^2 &\omega^2 & \omega^2 & -\omega & 0 & \omega\\
1 &-\omega^2 & -\omega^2 & -\omega^2 & \omega & -\omega^2
\end{pmatrix}
$$
The sum rules of Props.\ref{propzer} and \ref{magic} are easily checked. We find row sums: $4\omega^2$, $-2\omega^2$, $2\omega^2$, $-2\omega^2$, $3\omega^2$, $-5\omega^2$, 
column sums: $-2\omega$, $-2\omega$, $\omega$, $-2\omega$, $2\omega$, $3\omega$, 
and diagonal sums: $1$, $0$, $0$, $0$, $1$, $1$, $1$, $-1$,  $-2$, $-1$, $0$, each adding up to $0$.
\end{example}

\section{Discussion/Conclusion}
\label{sec:conclusion}

In this paper we have considered the two-dimensional ice model of statistical physics on the triangular 
lattice, the 20V model,  from a combinatorial point of view. In particular, we have defined analogs of 
the known DWBC for the 6V model, and investigated their possible 
combinatorial content, using as much as possible the underlying integrable structure of the models.

\subsection{DWBC1,2: summary and perspectives}
\label{sec:summary}
The first class of boundary conditions DWBC1,2 have displayed a remarkably rich combinatorial
content. We have shown in particular that the configurations of the models are equinumerous to
the domino tilings of a quasi-square Aztec-like domain with a central cross-shaped hole
and with quarter-turn symmetry. We also presented a refined version, in the same
spirit as the Mills-Robins-Rumsey conjecture \cite{tsscpp} for DPP, fully proved in \cite{BDFPZ1}.
Note that this coincidence of (refined) partition functions is still lacking a direct bijective interpretation.
However, such a canonical bijection is not known even for the ASM-DPP correspondence. 

The 20V model presents interesting new features compared to the 6V model. Its osculating path
version involves Schr\"oder paths with three different kinds of steps (horizontal, vertical, diagonal). 
Note that the {\it same} paths, but with a stronger non-intersecting condition are involved in the
description of domino tilings of the holey square.
It is easy to keep track say of the diagonal steps when dealing with a single path (with a weight $\gamma$
per diagonal step), as well as when dealing with families of such non-intersecting paths (see Ref.~\cite{DFG3}
for the problem of tiling Aztec rectangles with defects). The corresponding decoration is easy
to implement in the holey square tiling problem (where $\gamma$ is simply a weight for one kind of domino).
Unfortunately, it is easily checked that the partition
functions of the 20V-DWBC1,2 and that of domino tilings of the holey square no longer match for $\gamma\neq 1$.
Besides, it turns out that weights incorporating 
a factor $\gamma$ per diagonal step of path in the 20V model are non-longer integrable in general, namely cannot
be obtained by special choices of spectral parameters. On the other hand, it would be interesting to find a weighting 
of the 20V model equivalent to the $\gamma$-deformation of the tiling.

Another easily implementable natural weight in the particular context of the quarter-turn symmetric domino tilings of the holey Aztec square is a weight $\theta$ per path in the formulation as non-intersecting paths
on a cone of Section \ref{sec:cone}. The suitably modified partition function reads
$\det({\mathbb I}_n+\theta\, M_n)$ in the notations of Theorem \ref{thm:T4}: when evaluating the determinant via the Cauchy-Binet formula, we indeed pick a factor $\theta$ per row of the chosen submatrix of $M_n$.
It would also be interesting to find a weighting of the 20V model that corresponds to this.

Finally, like in the DPP case \cite{BDFPZ2}, further refinements could in principle be derived, by use of the
Desnanot-Jacobi relation. We hope to return to these questions in future work.

\subsection{The DWBC3 conjecture}
\begin{figure}
\begin{center}
\includegraphics[width=14cm]{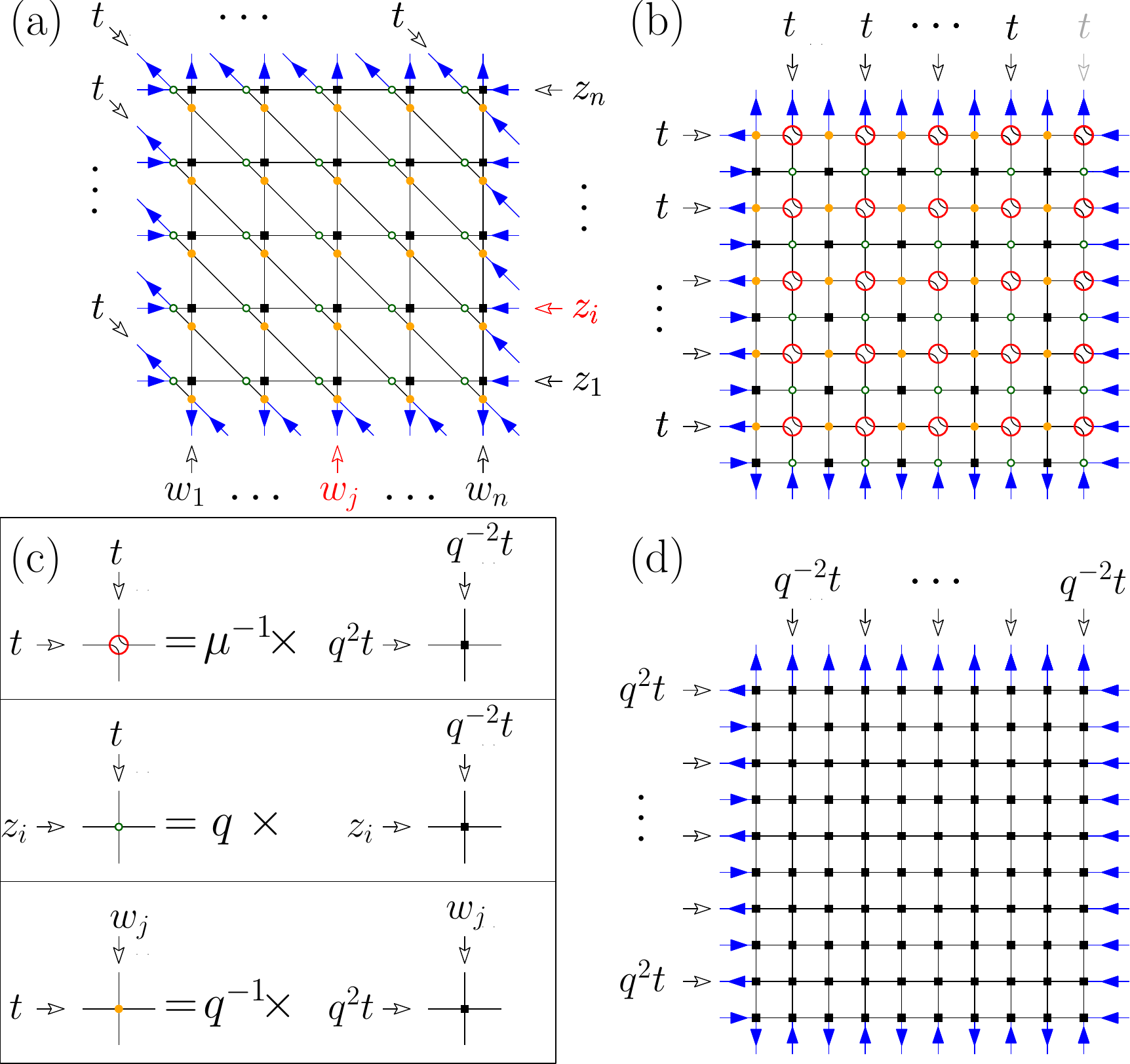}
\end{center}
\caption{\small Starting from the 20V-DWBC3 model on an $n\times n$ grid in its equivalent Kagome lattice formulation (a), 
we deform the diagonal lines so as to form an intermediate square grid of size $2n\times 2n$ (b) with four sublattices 
corresponding to the 3 original sublattices 1,2,3 of the Kagome lattice indicated 
respectively with black squares (1), empty circles (2), filled circles (3) plus an extra sublattice 4 of kissing points (circled), including 
the top right kissing point with an added trivial corner oriented line. 
Thanks to the equivalence to 6V weights indicated in (c), the model is finally reexpressed as a staggered 6V model
with appropriate spectral parameters (d).
}
\label{fig:staggered}
\end{figure}
Proving the DWBC3 conjecture \ref{DWBC3conj} is the next challenge. Let us mention that the 
partition function $Z^{{20V}_{BC3}}(n) $ of the 20V-DWBC3 model on an $n\times n$ grid, due to the integrable
nature of its weights, can be related to a partition function of the 6V model on a $2n\times 2n$ grid
of square lattice, by unraveling the diagonal lines in a way similar to that of Section \ref{sec:unravel}.

More precisely, we start from the general DWBC3 partition function, with arbitrary Kagome spectral
parameters $z_i,w_j$, $i,j=1,2,...,n$ and homogeneous $t_k=t$, $k=1,2,...,2n-1$.
First, we note that the $(n-1)$ bottom diagonal lines are ``imprisoned" due to the alternating orientations of
external arrows on the West and South boundaries, and the corresponding lines cannot be expelled like in the case of
Section  \ref{sec:unravel}. The top $n$ on the other hand could in principle be disposed of in the same way as before.
However, it proves more interesting to keep them and to deform the 
lines in the manner illustrated in Fig.~\ref{fig:staggered}. The idea is to place the former nodes of the sublattices 2 and 3 of the Kagome on square (sub-)lattices, still denoted by 2,3, and to
form artificial new nodes at the kissing between pairs of deformed diagonal lines, and positioned on a square sublattice denoted by 4 (see Fig.~\ref{fig:staggered}~(b)).
At these nodes, the kissing condition is guaranteed by choosing vertex weights $(a_4,b_4,c_4)= (1,0,1)$,
where the vanishing of the weight $b_4$ ensures the transmission of the arrow orientation from top to right 
and from left to bottom for these artificially created vertices.
By careful inspection of the vertex weights whenever some $t_k=t$ is involved, we find that the
weights for sublattices 2 and 3 may be expressed (up to a global multiplicative factor $q$ for sublattice 2 and $q^{-1}$ for sublattice 3) with the \emph{same definition} as that used 
for the sublattice 1 of the Kagome lattice (or
equivalently given by \eqref{eq:6Vweights}) provided we change $t\to q^2 t$ in the definition of horizontal spectral parameters
and $t\to q^{-2}t$ in the definition of the vertical ones (see Fig.~\ref{fig:staggered}~(c)). Remarkably, with these modified spectral parameters and this very same unified definition, 
the 6V weights on the sublattice 4 would be 
\begin{equation}
\label{sublat4} \big(A(q^2t ,q^{-2}t ),B(q^2t,q^{-2}t),C(q^2 t,q^{-2} t)\big)=(q^2-q^{-2})t\times(1,0,1)
\ ,
\end{equation}
which reproduce precisely the desired weights $(a_4,b_4,c_4)$ up to the proportionality factor $\mu=(q^2-q^{-2})t$. Otherwise stated, the weights 
on the four sublattices correspond (up to global factors $1$, $q$, $q^{-1}$ and $\mu^{-1}$ for sublattices $1,2,3,4$ respectively) to 6V weights for a consistent set of spectral parameters on the horizontal and vertical lines of a $2n\times 2n$ grid. 

The other spectral parameters $z_i$ and $w_j$ have remained unchanged in the process.
Taking them to the homoegeous 20V values $z_i=q^6 t$ and $w_j=q^{-6}t$, we obtain
a new 6V model now on a $2n\times 2n$ grid, but with {\it staggered} boundary conditions
and weights: 
\begin{eqnarray*}
(a_1,b_1,c_1)&=&(A(q^6 t,q^{-6}t),B(q^6 t,q^{-6}t),C(q^6 t,q^{-6}t))=\mu\, (1,\sqrt{2},1)\\
(a_2,b_2,c_2)&=&q\, (A(q^6 t,q^{-2}t),B(q^6 t,q^{-2}t),C(q^6 t,q^{-2}t))=q^3\mu\,(\sqrt{2},1,1)\\
(a_3,b_3,c_3)&=&q^{-1}\, (A(q^2 t,q^{-6}t),B(q^2 t,q^{-6}t),C(q^2 t,q^{-6}t))=q^{-3}\mu\,(\sqrt{2},1,1)\\
(a_4,b_4,c_4)&=&(1,0,1) \ ,
\end{eqnarray*}
respectively on the sublattices 1,2,3 (as in \eqref{eq:homval}) and 4.
Recalling the choice of spectral parameter $t$ \eqref{eq:valt}, leading to $\mu^3=1/2$, to ensure that the original 20V weights are all $1$,
the net result is a re-expression of the total number of configurations of the
20V-DWBC3 as:
$$ Z^{{20V}_{BC3}}(n) =\mu^{n^2}(q^3\mu)^{n^2}(q^{-3}\mu)^{n^2}\, Z^{6V_{\rm staggered}}_{WS}(2n) =\frac{1}{2^{n^2}} \, 
Z^{6V_{\rm staggered}}_{WS}(2n)\ ,$$
where $Z^{6V_{\rm staggered}}_{WS}(2n)$ denotes the partition function of the staggered 6V model 
on a $2n\times 2n$ grid, with weights
$(a,b,c)$ respectively equal to $(1,\sqrt{2},1)$, $(\sqrt{2},1,1)$, $(\sqrt{2},1,1)$ and $(1,0,1)$
on the sublattices 1,2,3,4, and with alternating external arrow orientations on the West and 
South boundaries (indicated by the index $WS$),
entering arrows on the East and outgoing arrows on the North (as in Fig.~\ref{fig:staggered}(d)).

In fact, the {\it same} transformation could be performed on the 20V model on an $n\times n$ grid
with arbitrary boundary conditions. In particular, this allows to re-express the partition function for the
20V-DWBC4 model as a staggered 6V model with the same definition of weights as above, but with
alternating external arrow orientations on all (West, South, East, North) boundaries:
$$ Z^{{20V}_{BC4}}(n)=\frac{1}{2^{n^2}} \, 
Z^{6V_{\rm staggered}}_{WSEN}(2n)\ .$$
The same transformation for the DWBC1 20V model leads to an alternative re-expression as:
$$ Z^{{20V}_{BC1}}(n)=\frac{1}{2^{n^2}} \, 
Z^{6V_{\rm staggered}}(2n)\ ,$$
where the staggered 6V model has domain wall boundary conditions, i.e. all horizontal external 
arrows pointing towards the domain, and all vertical external arrows pointing away from the domain.

There are no known general determinantal formulas for the partition functions of the staggered 6V model. However,
the latter can be solved by use of algebraic Bethe Ansatz, so there is some hope of transforming the
relevant partition functions into some simpler objects. We also leave this investigation to some future work.

\subsection{Symmetry classes}

Among the many questions that remain, an interesting direction which is similar to that in the case
of 6V/ASM/Rhombus Tilings  correspondences is the introduction of symmetry classes, obtained by restricting to configurations with some specified symmetry.

Our problem allows for less possible symmetries than the original 6V one, as we have broken the 
natural symmetries of the square domain in our choices of DWBC. Nevertheless, we have identified two
interesting symmetry classes for the DWBC1,2 models, one for DWBC3 and two for DWBC4.

In osculating Schr\"oder path language, the first symmetry, common to all DWBC1,2,3 and 4, is simply the symmetry under reflection w.r.t.
the first diagonal, which clearly respects all four choices of boundary conditions. In the APM language
(with 0 and sixth root of unity entries),
this symmetry amounts, for the APM $A$, to the condition:
$$ A_{n+1-j,n+1-i}^* = A_{i,j} \qquad (i,j=1,2,...,n)\  ,$$
where the complex conjugation interchanges $\omega$ and $\omega^2$.
We denote the APMs having this symmetry by Symmetric Alternating Phase Matrices (SAPM), which come in four types.

The second symmetry is more subtle and applies only to DWBC1 and 2. 
It is the composition of a reflection w.r.t. the second diagonal
and the complementation which interchanges occupied  and empty edges, while the central diagonal line
remains entirely occupied (DWBC1) or entirely empty (DWBC2). 
Note that in the cases of DWBC3 and 4 this transformation would be incompatible
with the boundary conditions.
In the APM language,
this amounts to the condition:
$$ A_{j,i}^* = A_{i,j} \qquad (i,j=1,2,...,n)\  , $$
namely that the corresponding APMs be Hermitian.
We denote the APMs having this symmetry by Transpose Conjugate Alternating Phase Matrices (TCAPM).

The last symmetry occurs only for DWBC4: it is the rotation of the grid by $180^\circ$.
In the APM language, this amounts to the condition:
$$A_{n+1-i,n+1-j}=-A_{i,j} \qquad (i,j=1,2,...,n)\ .$$
We denote the APM having this symmetry by Half-Turn (symmetric) Alternating Phase Matrices (HTAPM).

Using transfer matrix techniques, we found the following sequences for the various symmetry classes of APM:
\begin{eqnarray*}
{\rm Type \ 1,2:}&& {\rm SAPM:} \qquad 1,\ 3,\ 13,\ 85,\ 861 \\
	        && {\rm TCAPM:}\quad  \,\, 1,\ 2,\ 6,\ 28,\ 204 \\
{\rm Type\  3:}&& {\rm SAPM:}\qquad 1, \ 3, \ 15, \ 135, \ 2223\\
{\rm Type\ 4:}&& {\rm SAPM:}\qquad 1, \ 3, \ 27, \ 639 \ \\	
		   && {\rm HTAPM:}\quad \,1, \ 1, \ 7, \ 53 \ \\
\end{eqnarray*}

We remark that the first 5 TCAPMs of type $1,2$ are enumerated by the $q=2$ q-Bell numbers $B_n(q)$
defined recursively by $B_0(q)=1$ and $B_{n+1}(q)=\sum_{k=0}^n {n \choose k}_q \, B_k(q)$
where ${n \choose k}_q=\prod_{i=1}^k (1-q^{n+1-i})/(1-q^i)$  is the q-binomial coefficient. It is tempting to make a conjecture, but clearly some extra numerical effort is needed here.

\subsection{Arctic phenomenon}

To conclude, we expect, at least for the case of 20V with DWBC1,2, the existence of an ``arctic phenomenon" similar to that observed
for ASMs \cite{CP2010,CPS,COSPO} as well as the more general 6V-DWBC \cite{CNP,CPZ}. This was
our original motivation for considering the 20V model with DWBC, and will be the subject of future work.

\appendix
\section{The passage from Eq.~\eqref{eq:IK} to Eq.~\eqref{eq:IKlimit} in the homogeneous case}
\label{appendixA}
We wish to get an expression for \eqref{eq:IK} when $z_i\to z$ and $w_j\to w$ for all $i$ and $j$.
Following \cite{BDFPZ1}, we start with the following identity, valid for any power series $f(z,w)$:
\begin{equation}
\frac{1}{\prod\limits_{1\leq i<j\leq n}(z_j-z_i)(w_j-w_i)}\det\limits_{1\leq i,j\leq n}\left(f(z_i,w_j)\right)=
\det\limits_{1\leq i,j\leq n}\left(f[z_1,\dots, z_i][w_1,\dots ,w_j]\right)
\label{eq:matid}
\end{equation}
where
\begin{equation}
f[z_1,\dots, z_i][w_1,\dots ,w_j]:=\sum_{k=1}^i\sum_{l=1}^j \frac{f(z_k,w_l)}{\prod\limits_{k'=1\atop k'\neq k}^i(z_k-z_{k'})
\prod\limits_{l'=1\atop l'\neq l}^j(w_l-w_{l'})}\ .
\label{eq:deffbracket}
\end{equation}
Indeed, introducing the lower triangular matrix
\begin{equation}
L[z_1,\dots, z_n]_{i,j}=\left\{
\begin{matrix}
\prod\limits_{k=1\atop k\neq j}^i\frac{1}{z_j-z_k} & \hbox{if}\ i\geq j \\ 
0 & \hbox{if}\ i<j \\
\end{matrix}
\right.
\label{eq:Lbracket}
\end{equation}
and the matrix $H$ with elements $H_{i,j}=f(z_i,w_j)$, the left hand side of \eqref{eq:matid} is nothing but the product $\det(L[z_1,\dots, z_n])\det(L[w_1,\dots, w_n]^{t})\det(H)$
while the right hand side is nothing but $\det(L[z_1,\dots, z_n])\cdot H\cdot L[w_1,\dots, w_n]^{t})$ so that the identity follows immediately.

We may now evaluate
\begin{equation*}
\begin{split}
f[z_1,\dots, z_i][w_1,\dots ,w_j]&=\left(\frac{1}{2{\rm i}\, \pi}\right)^2
\oint dt\oint dt'\ \frac{1}{\prod\limits_{k=1}^i(t-z_k)}\frac{1}{\prod\limits_{l=1}^j(t'-z_l)}\ f(t,t')
\\
&= \sum_{m,m'=0}^\infty f_{m,m'} \oint dt\oint dt'\ \frac{1}{\prod\limits_{k=1}^i(t-z_k)}\frac{1}{\prod\limits_{l=1}^j(t'-z_l)}t^m\, t'^{m'}\\
&=\sum_{m,m'=0}^\infty f_{m,m'} \, {\rm Res}_\infty\frac{t^m}{\prod\limits_{k=1}^i(t-z_k)} \, {\rm Res}_\infty\frac{t^{m'}}{\prod\limits_{l=1}^j(t-w_l)}
\end{split}
\end{equation*}
where we introduced the coefficient $f_{m,m'}$ of the series expansion  $f(z,w)=\sum\limits_{m,m'=0}^\infty f_{m,m'}z^m w^{m'}$ and where the contours in the integrals were deformed so as
to pick the residue at infinity.  Clearly this residue is non-zero only if $m\geq i-1$ (respectively $m'\geq j-1$) with the result
\begin{equation*}
{\rm Res}_\infty\frac{t^m}{\prod\limits_{k=1}^i(t-z_k)}=
\sum_{p_1,p_2,\dots, p_i \geq 0 \atop \sum\limits_{k=1}^i p_k=m+1-i} \prod_{k=1}^{i} z_k^{p_k}=h_{m+1-i}(z_1,\dots,z_i)
\end{equation*}
in terms of the complete symmetric polynomial $h_m$.
This yields 
\begin{equation*}
f[z_1,\dots, z_i][w_1,\dots ,w_j]=
\sum_{m\geq i-1\atop m'\geq j-1} f_{m,m'} h_{m+1-i}(z_1,\dots, z_i)h_{m'+1-j}(w_1,\dots, w_j)
\end{equation*}
and in particular
\begin{equation*}
\begin{split}
f[\underbrace{z,\dots, z}_{i}][\underbrace{w,\dots ,w}_{j}]=&
\sum_{m\geq i-1\atop m'\geq j-1} f_{m,m'} h_{m+1-i}(\underbrace{z,\dots, z}_{i})h_{m'+1-j}(\underbrace{w,\dots, w}_{j})\\
&=
\sum_{m\geq i-1\atop m'\geq j-1} f_{m,m'}\ {m \choose i-1}\, z^{m+1-i} \ {m' \choose j-1}\, w^{m'+1-j}\\
&=f(z+r,w+s)\vert_{r^{i-1}s^{j-1}}\ . \\
\end{split}
\end{equation*}
This leads to the desired limit
\begin{equation*}
\frac{1}{\prod\limits_{1\leq i<j\leq n}(z_j-z_i)(w_j-w_i)}\det\limits_{1\leq i,j\leq n}\left(f(z_i,w_j)\right)\underset{z_k\to z\atop  w_l\to w}{\rightarrow}
\det\limits_{1\leq i,j\leq n}\left(f(z+r,w+s)\vert_{r^{i-1}s^{j-1}}\right)\ .
\end{equation*}
Using 
\begin{equation*}
\frac{1}{a(i,j)\, b(i,j)}=\frac{1}{z_i-w_j}\ \frac{1}{q^{-2}z_i-q^2w_j}=\frac{q^2}{(1-q^4)w_j}\left(\frac{1}{z_i-w_j}-\frac{1}{z_i-q^4w_j}\right)\ , 
\end{equation*}
we deduce
\begin{equation*}
\begin{split}
&\frac{\prod\limits_{i=1}^n c(i,i)\prod\limits_{i,j=1}^n\left(a(i,j)\, b(i,j)\right)}{\prod\limits_{1\leq i<j\leq n}(z_i-z_j)(w_j-w_i)}\det\limits_{1\leq i,j\leq n}\left( \frac{1}{a(i,j)\, b(i,j)}\right) 
\\
&=\frac{\prod\limits_{i=1}^n c(i,i)\prod\limits_{i,j=1}^n\left(a(i,j)\, b(i,j)\right)\prod\limits_{j=1}^n\left(\frac{q^2}{(1-q^4)w_j}\right)}{\prod\limits_{1\leq i<j\leq n}(z_i-z_j)(w_j-w_i)}\det\limits_{1\leq i,j\leq n}\left( \frac{1}{z_i-w_j}-\frac{1}{z_i-q^4w_j}\right) 
\\
&\underset{z_k\to z\atop  w_l\to w}{\rightarrow} (-1)^{\frac{n(n-1)}{2}}\prod_{i=1}^n \left((q^2-q^{-2})\sqrt{z\, w}\right)\prod_{i,j=1}^n\left((z-w)(q^{-2}z-q^2 w)\right)
\prod_{j=1}^n \frac{q^2}{(1-q^4)w}\\
& \quad \times 
\det\limits_{1\leq i,j\leq n}\left.\left(\left(\frac{1}{(z+r)-(w+s)}-\frac{1}{(z+r)-q^4(w+s)}\right)\right\vert_{r^{i-1}s^{j-1}}\right)\ .\\
\end{split}
\end{equation*}
This is nothing but  \eqref{eq:IKlimit}. 

\section{A proof of Eq.~\eqref{eq:IKlimitbis}}
\label{appendixB}
We now wish to get an expression for \eqref{eq:IK} when $z_i\to z$, $i=1,\dots, n$, $w_j\to w$, $j=1,\dots n-1$ and $w_n\to w\ u$.
Following again \cite{BDFPZ1}, we have the following general identity:
\begin{equation}
\frac{1}{\prod\limits_{1\leq i<j\leq n}(z_j-z_i)\prod\limits_{1\leq i<j\leq n-1}(w_j-w_i)}\det\limits_{1\leq i,j\leq n}\left(f(z_i,w_j)\right)=
\det\limits_{1\leq i,j\leq n}\left(
\begin{matrix}
f[z_1,\dots, z_i][w_1,\dots ,w_j]& j\leq n-1\\
f[z_1,\dots, z_i][w_n]& j=n\\
\end{matrix}
\right)
\label{eq:matidbis}
\end{equation}
where $f[z_1,\dots, z_i][w_1,\dots ,w_j]$ is as in \eqref{eq:deffbracket}. This identity is obtained as in Appendix~\ref{appendixA}
by multiplying the matrix $H$ with elements $H_{i,j}=f(z_i,w_j)$
to the left by the matrix $L[z_1,\dots, z_n]$ of \eqref{eq:Lbracket} and to the right by  the transpose of the matrix $L^{(n)}[w_1,\dots, w_{n-1}]$
with elements
\begin{equation*}
L^{(n)}[w_1,\dots, w_{n-1}]_{i,j}=\left\{
\begin{matrix}
\prod\limits_{k=1\atop k\neq j}^i\frac{1}{w_j-w_k} & \hbox{if}\ n-1\geq  i\geq j\\
\delta_{n,j} &  \hbox{if}\ i=n\\
0 & \hbox{if}\ i<j \\
\end{matrix}
\right.\ ,
\end{equation*}
then taking the determinant.
We have in particular in the desired limit
\begin{equation*}
\begin{split}
& \frac{1}{\prod\limits_{1\leq i<j\leq n}(z_j-z_i)\prod\limits_{1\leq i<j\leq n-1}(w_j-w_i)}\det\limits_{1\leq i,j\leq n}\left(f(z_i,w_j)\right)\\
& \qquad \qquad \underset{z_k\to z\ , \  w_{l<n}\to w \atop w_n\to w\, u}{\rightarrow}
\det\limits_{1\leq i,j\leq n}\left(
\begin{matrix}
f(z+r,w+s)\vert_{r^{i-1}s^{j-1}}& j\leq n-1 \\
f(z+r,w\, u)\vert_{r^{i-1}}& j=n 
\end{matrix}
\right)\\
\end{split}
\end{equation*}
We now deduce
\begin{equation*}
\begin{split}
&\frac{\prod\limits_{i=1}^n c(i,i)\prod\limits_{i,j=1}^n\left(a(i,j)\, b(i,j)\right)}{\prod\limits_{1\leq i<j\leq n}(z_i-z_j)(w_j-w_i)}\det\limits_{1\leq i,j\leq n}\left( \frac{1}{a(i,j)\, b(i,j)}\right) 
\\
&=\prod\limits_{i=1}^{n-1} c(i,i)\prod\limits_{i=1}^n\prod\limits_{j=1}^{n-1}\left(a(i,j)\, b(i,j)\right)\prod\limits_{j=1}^{n-1}\left(\frac{q^2}{(1-q^4)w_j}\right)\\
&\qquad \times c(n,n)\prod\limits_{i=1}^n\left(a(i,n)\, b(i,n)\right)\left(\frac{q^2}{(1-q^4)w_n}\right)\frac{1}{\prod\limits_{i=1}^{n-1}(w_n-w_i)}
\\
&\qquad \times \frac{1}{\prod\limits_{1\leq i<j\leq n}(z_i-z_j)\prod\limits_{1\leq i<j\leq n-1}(w_j-w_i)}\det\limits_{1\leq i,j\leq n}\left( \frac{1}{z_i-w_j}-\frac{1}{z_i-q^4w_j}\right) 
\\
&\!\!\!\!\!\!\!\! \underset{z_k\to z\ , \  w_{l<n}\to w \atop w_n\to w\, u}{\rightarrow} (-1)^{\frac{n(n-1)}{2}}\prod_{i=1}^{n-1} \left((q^2-q^{-2})\sqrt{z\, w}\right)\prod_{i=1}^{n}\prod_{j=1}^{n-1}\left((z-w)(q^{-2}z-q^2 w)\right)
\prod_{j=1}^{n-1} \frac{q^2}{(1-q^4)w}\\
&\qquad \times (q^2-q^{-2})\sqrt{z\, w\, u}\prod_{i=1}^{n}\left((z-w\, u)(q^{-2}z-q^2 w\, u)\right)\frac{q^2}{(1-q^4)w\, u}\ \frac{1}{\prod\limits_{i=1}^{n-1}(w\, u-w)}
\\
& \qquad \times 
\det\limits_{1\leq i,j\leq n}\left(
\begin{matrix}
\left.\left(\frac{1}{(z+r)-(w+s)}-\frac{1}{(z+r)-q^4(w+s)}\right)\right\vert_{r^{i-1}s^{j-1}} & j\leq n-1\\
\left.\left(\frac{1}{(z+r)-w\ u}-\frac{1}{(z+r)-q^4w\, u}\right)\right\vert_{r^{i-1}} & j=n\\
\end{matrix}
\right)\ .
\\
\end{split}
\end{equation*}
This is nothing but \eqref{eq:IKlimitbis}.

\bibliographystyle{amsalpha} 

\bibliography{20V}

\end{document}